\documentclass{amsart}
 \usepackage[fontsize=11pt]{fontsize}   

 \usepackage[a4paper, left=1.4cm, right=1.3cm, top=1cm,bottom=1cm]{geometry}

\usepackage{latexsym}
\usepackage{indentfirst}
\usepackage{amssymb,amsfonts,amsmath,amsthm}
\usepackage{graphicx}
\usepackage{mathrsfs}
\usepackage{array}
\usepackage{color}
\usepackage{epstopdf}
\usepackage[all]{xy}
\usepackage{url}
\usepackage{enumitem}
 \usepackage[normalem]{ulem} 

\usepackage{tikz-cd}

\usepackage{color}
\definecolor{ghcolor}{RGB}{0, 150, 200} 
\definecolor{winestain}{rgb}{0.5,0,0}
\usepackage[linktocpage=true, hidelinks, colorlinks, linkcolor=blue, citecolor=magenta, bookmarksopen=true, urlcolor=brown]{hyperref}

 \newtheorem{thm}{Theorem}[section]
\newtheorem{theorem}[thm]{Theorem}
\newtheorem{prop}[thm]{Proposition}

\newtheorem{lem}[thm]{Lemma}

\newtheorem{lemma}[thm]{Lemma}

\newtheorem{cor}[thm]{Corollary}

\theoremstyle{definition}
\newtheorem{exam}[thm]{Example}

\newtheorem{rmk}[thm]{Remark}
\newtheorem{rem}[thm]{Remark}
\newtheorem{remark}[thm]{Remark}
\newtheorem{dfn}[thm]{Definition}
\newtheorem{convention}[thm]{Convention}
\newtheorem{example}[thm]{Example}

\newtheorem{defn}[thm]{Definition}
\newtheorem{notation}[thm]{Notation}
\newtheorem{Notation}[thm]{Notation}
\newtheorem{construction}[thm]{Construction}

\numberwithin{equation}{section}

\newcommand{\frakA}{{\mathfrak A}}

\newcommand{\frakS}{{\mathfrak S}}

\newcommand{\bA}{{\mathbb A}}
\newcommand{\bB}{{\mathbb B}}

\newcommand{\bM}{{\mathbb M}}
\newcommand{\bN}{{\mathbb N}}

\newcommand{\bQ}{{\mathbb Q}}

\newcommand{\bZ}{{\mathbb Z}}

\newcommand{\calA}{{\mathcal A}}

\newcommand{\calI}{{\mathcal I}}
\newcommand{\calJ}{{\mathcal J}}
\newcommand{\calK}{{\mathcal K}}

\newcommand{\calO}{{\mathcal O}}

\newcommand{\calS}{{\mathcal S}}

\newcommand{\rA}{{\mathrm A}}

\newcommand{\rC}{{\mathrm C}}

\newcommand{\rH}{{\mathrm H}}

\newcommand{\rW}{{\mathrm W}}

\newcommand{\Zp}{{\bZ_p}}

\newcommand{\Qp}{{\bQ_p}}

\newcommand{\Ainf}{{\mathrm{A_{inf}}}}

\newcommand{\BdRp}{{\mathrm{B_{dR}^+}}}

\newcommand{\BBdRp}{{\bB_{\mathrm{dR}}^+}}           
\newcommand{\BBdR}{{\bB_{\mathrm{dR}}}}              

\newcommand{\Ext}{{\mathrm{Ext}}}           
\newcommand{\Hom}{{\mathrm{Hom}}}           
\newcommand{\id}{{\mathrm{id}}}             
\newcommand{\Ker}{{\mathrm{Ker}}}           
\newcommand{\Rep}{{\mathrm{Rep}}}           
\newcommand{\RGamma}{{\mathrm{R\Gamma}}}    
\newcommand{\Sh}{{\mathrm{Sh}}}             
\newcommand{\Tot}{{\mathrm{Tot}}}           
\newcommand{\Vect}{{\mathrm{Vect}}}          

\newcommand{\GL}{{\mathrm{GL}}}             

\newcommand{\an}{{\mathrm{an}}}             
\newcommand{\cris}{{\mathrm{cris}}}         
\newcommand{\dR}{{\mathrm{dR}}}             
\newcommand{\pd} {{\mathrm{pd}}}                           
\newcommand{\perf}{\mathrm{perf}}           
\newcommand{\st}{{\mathrm{st}}}             

\usepackage{relsize}
\usepackage[bbgreekl]{mathbbol}
\DeclareSymbolFontAlphabet{\mathbb}{AMSb} 
\DeclareSymbolFontAlphabet{\mathbbl}{bbold}
\newcommand{\Prism}{{\mathlarger{\mathbbl{\Delta}}}} 

\newcommand{\okprism}{{(\ok)_\prism}}
\newcommand{\okpris}{{(\mathcal{O}_K)_\prism}}

\newcommand{\okprislog}{{(\mathcal{O}_K)_{\prism, \mathrm{log}}}}
\newcommand{\okprisast}{{(\mathcal{O}_K)_{\prism, \ast}}}

\newcommand{\okprisperf}{(\calO_K)^{\perf}_{\Prism}}

\newcommand{\opris}{{\mathcal{O}_\prism}}

\newcommand{\baroprism}{{\overline{\O}_\prism}}

\label{to move around}
 \renewcommand{\O}{{\mathcal{O}}}

\label{Topology, THH, TC, TP}

 \label{colors}

\label{arrows}

\newcommand \into {\hookrightarrow }
\renewcommand \to {\rightarrow}
\newcommand \onto {\twoheadrightarrow}
\renewcommand{\projlim}{\varprojlim}

\label{•}
 
 \label{matrix, groups}

\newcommand{\vect}{\mathrm{Vect}}

\def\Mat{\mathrm{Mat}}

\label{module, homological alg}
 
\newcommand{\rep}{{\mathrm{Rep}}}

\def\inf{{\mathrm{inf}}}
\def\sup{\mathrm{sup}}

\newcommand{\gal}{{\mathrm{Gal}}}

\newcommand{\spf}{{\mathrm{Spf}}}

\def\an{\mathrm{an}}
\def\perf{\mathrm{perf}}

\newcommand{\Lie}{\mathrm{Lie}}

\label{condesend, solid math}

\label{!}
\label{ALL: p-adic Hodge}

\label{LAV}

\newcommand{\dan}{\text{$\mbox{-}\mathrm{an}$}}

\newcommand{\dla}{\text{$\mbox{-}\mathrm{la}$}}

\newcommand{\dpa}{\text{$\mbox{-}\mathrm{pa}$}}
\renewcommand{\log}{\mathrm{log}}

\newcommand{\bdrplusm}{{\bfb_{\dR, m}^+}}

\newcommand{\Kpinfty}{{K_{p^\infty}}}
\newcommand{\kpinfty}{{K_{p^\infty}}}

\newcommand{\Kinfty}{{K_{\infty}}}
\newcommand{\kinfty}{{K_{\infty}}}

\label{LAV: Lie gp Lie alg}

\newcommand{\gammak}{{\Gamma_K}}

\newcommand{\gk}{{G_K}}

\label{period ring: Fontaine}

\newcommand\dif{{\mathrm{dif}}}
\newcommand\rig{{\mathrm{rig}}}

\newcommand{\ainf}{{\mathbf{A}_{\mathrm{inf}}}}

\newcommand{\bcrisplus}{{\mathbf{B}^+_{\mathrm{cris}}}}

\newcommand{\bdrplus}{{\mathbf{B}^+_{\mathrm{dR}}}}
\newcommand{\bdr}{{\mathbf{B}_{\mathrm{dR}}}}

\label{period ring: relative}

\label{period ring: Gao-Min-Wang}
\newcommand{\bdrpluskpinfty}{{\B_{\dR, \kpinfty}^+}} 

\newcommand{\ndR}{{\mathrm{ndR}}}

  \newcommand{\kpinftyt}{{\kpinfty[[t]]}}
 \newcommand{\kinftylambda}{{\kinfty[[\lambda]]}}

\newcommand{\bdrplusl}{{\mathbf{B}^+_{\dR,  L}}}
 \newcommand{\bdrpluslhatgpa}{{(\B_{\dR, L}^+)^{\hat{G}\dpa}}}

\newcommand{\mic}{\mathrm{MIC}} 
\newcommand{\bm}{\mathbb{M}}

\newcommand{\dacc}[1]{\{\!\{ #1 \}\!\}}

 \newcommand{\fg}{{\mathrm{fg}}}

\label{period ring: bb font}

\newcommand{\bbdr}{{\mathbb{B}_{\mathrm{dR}}}}
\newcommand{\bbdrplus}{{\mathbb{B}_{\mathrm{dR}}^{+}}}

\newcommand*{\wt}[1]{\widetilde{#1}}

\label{period ring: AB ring}

\newcommand{\B}{  {\mathbf{B}}  }

\newcommand{\wtb}{   {\widetilde{{\mathbf{B}}}}  }

\label{period ring: integral Hodge}

\def \ok {{\mathcal{O}_K}}

\def \oc {{\mathcal{O}_C}}

\newcommand{\ocflat}{{\mathcal{O}_C^\flat}}

\label{peirod module: D}

\label{Cats: Fontaine}

\label{Cats: Rep and Higgs}

\label{*}
\label{ALL: site, sheaves}
  
\newcommand{\rgamma}{\mathrm{R}\Gamma}

 \newcommand{\Shv}{\mathrm{Shv}}

\label{site: etale proetale}

\label{site: prismatic}

\usepackage{relsize}
\usepackage[bbgreekl]{mathbbol}
\usepackage{amsfonts}

\DeclareSymbolFontAlphabet{\mathbb}{AMSb}
\DeclareSymbolFontAlphabet{\mathbbl}{bbold}
\newcommand{\prism}{{\mathlarger{\mathbbl{\Delta}}}}
 
\newcommand{\baropris}{{\overline{\O}_\prism}}

\label{syntomic coho}

\label{font: roman rm}

\label{font: mathbb}

\newcommand{\barK}{{\overline{K}}}

\newcommand{\zp}{{\mathbb{Z}_p}}
\newcommand{\qp}{{\mathbb{Q}_p}}

\label{font: mathcal}

\label{font: mathfrak}

\newcommand{\gs}{{\mathfrak{S}}}

\newcommand{\fkc}{{\mathfrak{c}}}

\newcommand{\fkt}{{\mathfrak{t}}}

 \label{font: hat}

\label{font: bar}

\label{font: bold}

\label{font: tilde}

 \label{font: mathbf}

\newcommand{\cbf}{\mathbf{c}}

\newcommand{\kbf}{\mathbf{k}}

\newcommand{\bfa}{\mathbf{A}}
\newcommand{\bfb}{\mathbf{B}}

\label{;'}

\label{stack proj}

\label{words}

\label{Gao added to dR paper}

\newcommand{\bbdrplusm}{{\mathbb{B}_{\mathrm{dR}, m}^{+}}}
 
\newcommand{\bdrplusml}{{\mathbf{B}^+_{\dR, m, L}}}
 
 \newcommand{\MIC}{{\mathrm{MIC}}} 
\newcommand{\bfB}{\mathbf{B}}

\label{address}
\author[]{Hui Gao}   \address{Department of Mathematics, Southern University of Science and Technology, Shenzhen 518055, China}   \email{gaoh@sustech.edu.cn}

\author[]{Yu Min}
\address{Department of Mathematics, Imperial College London, London SW7 2RH}
\email{y.min@imperial.ac.uk}

\author[]{Yupeng Wang}
\address{Beijing International Center of Mathematics research, Peking University, Yiheyuan 5, Beijing, 100190, China.}
\email{2306393435@pku.edu.cn}  
  
\begin{document}

\title[]{Prismatic crystals over the de Rham period sheaf}



\begin{abstract} \normalsize{
Let $\mathcal{O}_K$ be a mixed characteristic complete discrete valuation ring with perfect residue field.
We study $\mathbb{B}_\mathrm{dR}^+$-crystals on the  (log-) prismatic site of $\mathcal{O}_K$, which are crystals defined over the de Rham period sheaf.
We  first classify these  crystals   using certain log connections.
By constructing a Sen--Fontaine theory for $\mathbf{B}_{\mathrm{dR}}^+$-representations over a Kummer tower, we further classify these crystals   by (log-) nearly de Rham representations.
In addition, we compare  (log-)  prismatic  cohomology of these crystals with the corresponding Sen--Fontaine cohomology  and  Galois cohomology.
}
\end{abstract}

\date{\today}

\maketitle
\setcounter{tocdepth}{1}
\tableofcontents

\section{Introduction}

\subsection{Crystals and Galois representations}\label{subsec1.1}
In \cite{BS22}, Bhatt--Scholze introduce  the  prismatic site and use  it to recover all known \emph{integral} $p$-adic cohomology theories, which revolutionized the subject. In there, Bhatt--Scholze only consider cohomology with trivial coefficients. Since then, there has been much progress in understanding coefficients---that is, prismatic crystals---in these cohomology theories.
Let $K$ be a mixed characteristic complete discrete valuation field with perfect residue field, and let $\ok$ be its ring of integers.
In \cite{BS23},  Bhatt--Scholze show that $F$-crystals over $(\ok)_\prism$ (the absolute prismatic site of $\ok$) classify  integral crystalline representations; shortly after, Du--Liu \cite{DL23} show that $F$-crystals over the log-prismatic site classify integral semi-stable representations.
In some sense, these results show that  the classical study of \emph{algebraic} integral $p$-adic Hodge theory admit \emph{prismatic (hence geometric) interpretations}.

In another direction, recently, \emph{Hodge--Tate prismatic crystals} are shown to be closely related with classical Sen theory and $p$-adic non-abelian Hodge theory, cf. \cite{BL-a, BL-b, AHLB1, AHLB2, GMWHT, MW22, MW22log}. 
These studies not only reveal that Sen theory admits a prismatic interpretation, they also bring  forth interesting new categories of Galois representations which await further studies. In this paper, we continue this line of investigation and study the \emph{$\mathbb{B}_\mathrm{dR}^+$-crystals} over the prismatic (resp. log-prismatic) site.  Let us now quickly set up some notations to facilitate our discussions.
Since this paper is a natural continuation of \cite{GMWHT}, we shall be brief when we introduce notations that are common and standard in literatures, particularly those used repeatedly in \cite{GMWHT}.

 Let $(\calO_K)_{\Prism}$ (resp. $(\calO_K)_{\Prism,\log}$) be the absolute prismatic site (resp. absolute log-prismatic site) of $\ok$; let  $(\calO_K)_{\Prism}^{\perf}$ (resp. $(\calO_K)_{\Prism,\log}^{\perf}$) be the subsite consisting of perfect prisms (resp. perfect log prisms). Let $\mathcal{O}_\Prism$ and $\mathcal{I}_\prism$ be the structure sheaf and the structure ideal sheaf on these sites. For example, on  $\okprism$,  we have
\[ \mathcal{O}_\Prism((A, I))=A, \quad \text{resp. }\mathcal{I}_\prism((A, I)) =  I.\]
Define $\baroprism: =\mathcal{O}_\Prism/\mathcal{I}_\prism  $ and its rational version $\baroprism[1/p]$; they are called the integral (resp. rational) Hodge--Tate structure sheaf.
 We introduce the sheaf that will play a central role in this paper. 

\begin{defn}\label{nota intro}
  For each $m \geq 1$, define the \emph{de Rham period sheaf of level $m$} as
\[ \bbdrplusm := (\mathcal{O}_\Prism[1/p])/\mathcal{I}^m_\prism,  \]
so $\mathbb{B}_{\dR, 1}=  \baroprism[1/p]$.
Define the \emph{(positive) de Rham period sheaf} as
\[\bbdrplus:= \projlim_{m \geq 1}  (\mathcal{O}_\Prism[1/p])/\mathcal{I}^m_\prism.\]  
\end{defn}

  We introduce   notations for categories of crystals resp.  representations.
 
 \begin{dfn}\label{Dfn-A crystal}
   Let $\calS$ be one of $\{(\calO_K)_{\Prism},(\calO_K)_{\Prism,\log},(\calO_K)_{\Prism}^{\perf},(\calO_K)_{\Prism,\log}^{\perf}\}$.
    Let $\bA$ be a sheaf of rings on $\calS$.
   Define the category of $\bA$-crystals on $\calS$ by 
   \[ \Vect(\calS,\bA):= \projlim_{\calA\in \calS} \vect(\bA(\calA))\]
 \end{dfn}

Recall in \cite{GMWHT}, objects in $\vect(\okprisast, \baropris)$ resp. $\vect(\okprisast, \baropris[1/p])$ (where $\ast \in \{\emptyset, \log \}$) are also called integral (resp. rational) Hodge--Tate crystals. But we shall simply call objects in $\vect(\okprisast, \bbdrplus)$ as $\bbdrplus$-crystals. See Rem. \ref {remnamebdrcrystal} for some comments on these terminologies.

 \begin{defn}\label{defsemilinrep}
 Suppose $\mathcal G$ is a topological group that acts continuously on a topological ring $R$. We use $\rep_{\mathcal G}(R)$ to denote the category where an object is a finite free $R$-module $M$ (topologized via the topology on $R$) with a continuous and \emph{semi-linear} $\mathcal G$-action in the usual sense that
$$g(rx)=g(r)g(x), \forall g\in \mathcal G, r \in R, x\in M.$$
(The only case in this paper where the action is \emph{linear} is when $R=\zp$).
 \end{defn}

We use the following theorem to set up our stage and motivate the readers. 
Note the sheaves $\mathcal{O}_\Prism$ and $ \mathcal{O}_\Prism[1/\mathcal{I}_\Prism]^{\wedge}_p$ carry Frobenius structures; thus we can define $F$-crystals over these sheaves of rings: they are crystals eqiupped with certain ``$\varphi$-isogeny".
Recall associated to the perfectoid field $C$ one can define the Fontaine prism and the Fontaine log prism, cf. Notation \ref{notaainflog}:
\[ (\ainf, (\xi)), \text{ resp. } (\ainf, (\xi), M_\ainf) \]
Fix an algebraic closure $\overline{K}$ and denote $\gk=\gal(\overline{K}/K)$. 
Let $\rep_\gk(\zp)$ be the category of integral $p$-adic representations. 
Let  $\rep^{\cris}_{\gk}(\zp)$ (resp. $\rep^{\st}_{\gk}(\zp)$) be the sub-category of integral crystalline (resp. semi-stable) representations.
(In the following theorem and throughout the paper, we use $\into$ resp. $\simeq$ in commutative diagrams to signify a functor is fully faithful  resp. an equivalence.)

\begin{theorem} \label{thmintroBSDL}
\cite{BS23, DL23} We have a commutative diagram of tensor functors:
\begin{footnotesize}
\begin{equation*}
\begin{tikzcd}
{\mathrm{Vect}^{\varphi}((\calO_K)_{\Prism}, \mathcal{O}_\Prism)} \arrow[r, hook] \arrow[d, "\simeq"] & {\mathrm{Vect}^{\varphi}((\calO_K)_{\Prism, \log}, \mathcal{O}_\Prism)} \arrow[r, hook] \arrow[d, "\simeq"] & {\mathrm{Vect}^{\varphi}(\okprislog, \mathcal{O}_\Prism[\frac{1}{\mathcal{I}_\Prism}]^{\wedge}_p)} \arrow[d, "\simeq"] & {\mathrm{Vect}^{\varphi}(\okprism, \mathcal{O}_\Prism[\frac{1}{\mathcal{I}_\Prism}]^{\wedge}_p)} \arrow[l, "\simeq"'] \arrow[d, "\simeq"] \\
\rep^{\cris}_{\gk}(\zp) \arrow[r, hook]                                                       & \mathrm{Rep}^{\st}_{\gk}(\zp) \arrow[r, hook]                                                               & \rep_\gk(\zp)                                                                                                  & \rep_\gk(\zp) \arrow[l, "="']
\end{tikzcd}
\end{equation*}
\end{footnotesize}
Here,   the first and fourth (resp. second and third) vertical equivalences are induced by evaluating the crystals on the Fontaine prism (resp. Fontaine log prism).
\end{theorem}

We now recall our previous results on  rational   Hodge--Tate   crystals. Note that these crystals, in contrast to those in Thm. \ref{thmintroBSDL},  do not have $\varphi$-structures.
 
 Let $C$ be the $p$-adic completion of $\overline K$. 
Objects in $\rep_\gk(C)$ (cf. Def. \ref{defsemilinrep}) are called $C$-representations. For $W \in \rep_\gk(C)$, Sen  \cite{Sen80} canonically associates a \emph{linear operator} $\phi: W \to W$ which nowadays is called the Sen operator; its eigenvalues are in $\barK$ and are called the (Hodge--Tate-)Sen weights of $W$.
(In our set-up, the Hodge--Tate-Sen weight  of the the cyclotomic character is $1$.)

\begin{notation}\label{notaeu}
Let $W(k)$ be the ring of Witt vectors, and let $K_0=W(k)[1/p]$.
Let $\pi \in K$ be a \emph{fixed} uniformizer, and let $E(u)=\mathrm{Irr}(\pi, K_0) \in W(k)[u]$ be the minimal polynomial over $K_0$. Let $E'(u) =\frac{d}{du}E(u)$. 
\end{notation}

\begin{defn}\label{defnnht}
Say $W\in \rep_\gk(C)$ is \emph{nearly Hodge--Tate} (resp. \emph{log-nearly Hodge--Tate})  if all of its  Sen weights are in the subset
\[\mathbb{Z} + (E'(\pi))^{-1}\cdot \mathfrak{m}_{\O_{\overline{K}}}, \quad \text{ resp. } \mathbb{Z} +  (\pi\cdot E'(\pi))^{-1}\cdot \mathfrak{m}_{\O_{\overline{K}}}.  \]
 where $\O_{\overline{K}}$ is the ring of integers of $\overline{K}$ with $\mathfrak{m}_{\O_{\overline{K}}}$  its maximal ideal.
 That is: the Sen weights are near to being an integer up to a bounded distance.  
Write
\[\rep_\gk^{\mathrm{nHT}}(C), \quad \text{resp. }\rep_C^{\mathrm{lnHT}}(\gk)\]
 for the (tensor) subcategory of $\rep_\gk(C)$ consisting of these objects.
\end{defn}

\begin{theorem} \label{thmintroht}
\cite{GMWHT}
We have a commutative diagram of tensor functors:
\begin{footnotesize}
\begin{equation*}
\begin{tikzcd}
{\Vect( \okpris,\overline \calO_{\Prism}[\frac{1}{p}])} \arrow[d, "\simeq"] \arrow[r, hook] & {\Vect( \okprislog,\overline \calO_{\Prism}[\frac{1}{p}])} \arrow[d, "\simeq"] \arrow[r, hook] & {\Vect((\calO_K)^{\perf}_{\Prism, \log},\overline \calO_{\Prism}[\frac{1}{p}])} \arrow[d, "\simeq"] & {\Vect((\calO_K)^{\perf}_{\Prism},\overline \calO_{\Prism}[\frac{1}{p}])} \arrow[d, "\simeq"] \arrow[l, "\simeq"'] \\
\rep_\gk^{\mathrm{nHT}}(C) \arrow[r, hook]                                                  & \rep_\gk^{\mathrm{lnHT}}(C) \arrow[r, hook]                                                    & \rep_\gk(C)                                                                                         & \rep_\gk(C) \arrow[l, "="']
\end{tikzcd}
\end{equation*}
\end{footnotesize}
Here,   the first and fourth (resp. second and third) vertical equivalences are induced by evaluating the crystals on the Fontaine prism (resp. Fontaine log prism).
\end{theorem}

Our main result in this paper classifies $\bbdrplus$-crystals, and 
is a ``$\bdrplus$-lift" of  Thm. \ref{thmintroht}. Recall $\bdrplus$ is precisely the evaluation of the sheaf $\bbdrplus$ on the Fotaine (log-) prism.


\begin{defn}\label{defnndR}
Say $W\in \rep_\gk(\bdrplus)$ is \emph{nearly de Rham} (resp. \emph{log-nearly de Rham}) if the $C$-representation $W/tW$ is nearly Hodge--Tate (resp. log-nearly Hodge--Tate). Use
$$\rep_\gk^{\mathrm{ndR}}(\bdrplus) \quad \text{  resp. }  \rep_\gk^{\mathrm{lndR}}(\bdrplus) $$
 to  denote the subcategory of $\rep_\gk(\bdrplus)$ consisting of these objects.
\end{defn}

\begin{theorem}[cf. Thm. \ref{Prop-perfectisproet} and Thm. \ref{thm_ndR_rep}]\label{thmintrondr} 
We have a commutative diagram of tensor functors:
\begin{footnotesize}
\begin{equation*}
\begin{tikzcd}
{\Vect( \okpris, \bbdrplus)} \arrow[d, "\simeq"] \arrow[r, hook] & {\Vect( \okprislog, \bbdrplus)} \arrow[r, hook] \arrow[d, "\simeq"] & {\Vect((\calO_K)^{\perf}_{\Prism, \log}, \bbdrplus)} \arrow[d, "\simeq"] & {\Vect(\okprisperf, \bbdrplus)} \arrow[l, "\simeq"'] \arrow[d, "\simeq"] \\
\rep_\gk^{\mathrm{ndR}}(\bdrplus) \arrow[r, hook]                & \rep_\gk^{\mathrm{lndR}}(\bdrplus) \arrow[r, hook]                  & \rep_\gk(\bdrplus)                                                       & \rep_\gk(\bdrplus) \arrow[l, "="']
\end{tikzcd}
\end{equation*}
\end{footnotesize}
Here,   the first and fourth (resp. second and third) vertical equivalences are induced by evaluating the crystals on the Fontaine prism (resp. Fontaine log prism). All equivalences are bi-exact, i.e., the equivalences and their   quasi-inverses are  exact.
\end{theorem}

We also obtain comparison theorem between (log-) prismatic cohomology  and Galois cohomology.

\begin{theorem}[cf. Thm. \ref{thm-pris coho and Galois coho}] \label{thmintrocoho}  
Let $\bm \in \Vect( \okprisast, \bbdrplus)$ where $\ast \in \{\emptyset, \log \}$, and let $W$ be the associated object in $\rep_\gk(\bdrplus)$ via Thm. \ref{thmintrondr}.
There exists a $K$-linear quasi-isomorphism
\[
  \RGamma((\calO_K)_{\Prism, \ast},\bM) \simeq
    \RGamma(G_K, W).
\]
    \end{theorem}

\begin{remark}\label{rem_intro_extension}
We quickly discuss genesis/inspiration of Theorem \ref{thmintrondr}, and point out some technical difficulties.
\begin{enumerate}
    \item Recall in  \cite{Fon04}, an object $U\in \rep_\gk(C)$ is called \emph{almost Hodge--Tate}   if its Sen weights are integers; thus, the (log-) nearly Hodge--Tate representations in Thm. \ref{thmintroht} are those \emph{near} to an almost Hodge--Tate representation.
In addition, in \cite{Fon04}, an object
      $W\in \rep_\gk(\bdrplus)$ is called \emph{almost de Rham} if $W/tW$ is almost Hodge--Tate.
Indeed, the form of Def. \ref{defnndR} and Thm. \ref{thmintrondr} are inspired by the treatment in \cite{Fon04}.
In particular, (log-) nearly de Rham representations are those \emph{``near"} to an almost de Rham representation.

\item Although the statement of Thm. \ref{thmintrondr}  is inspired by some classical study of Fontaine, there has been \emph{substantial difficulty}   to obtain Thm. \ref{thmintrondr}.
Indeed, in \cite{Fon04}, practically all results about $ \rep_\gk(\bdrplus)$ are obtained via \emph{d\'evissage} techniques ---that is, by studying objects in $ \rep_\gk(\bdrplus/t^m)$---  starting from results on $\rep_\gk(C)$.
In our paper, we certainly also make use of {d\'evissage} arguments; however, it turns out such naive idea does not always work out.
This is particularly the case when we study  the category $\Vect( \okpris, \bbdrplusm)$, whose objects are torsion \emph{sheaves} (rather than just modules) and hence are much more subtle when considering questions such as \emph{extensions}.
In our paper, a most technical and difficult result, concerning classifying $\bbdrplus$-crystals as log connections (cf. Thm. \ref{thm-intro-logconn}), involve solving certain recurrence  relations, which we \emph{cannot} obtain via any  {d\'evissage argument}.

\item As discussed extensively in the introduction of  \cite{GMWHT} (cf. \S 1.5 there), there is also a \emph{stacky} approach (contrary to the site-theoretic approach in \cite{GMWHT}) to study Hodge--Tate crystals, as carried out in \cite{Dri20, BL-a, BL-b, AHLB1, AHLB2}. Currently, it is not clear if there could be a stacky approach to the study of $\bbdrplus$-crystals. We also regard this   as an indication of difficulties in this paper; conversely, we expect some of the explicit computations in this paper could contribute to the development of a stacky approach.
 \end{enumerate}
\end{remark}

\begin{rem} We discuss   relation between $\bbdrplus$-crystals and   $p$-adic Riemann--Hilbert correspondence.
\begin{enumerate}
    \item 
First, we would like to point out that even though $\bbdrplus$-crystals are studied using \emph{d\'{e}vissage} arguments, starting from results on Hodge--Tate crystals, they   contain genuinely richer information.
One similar scenario is in the paper of Liu--Zhu \cite{LZ17}, where they construct  a $p$-adic Riemann--Hilbert functor for $p$-adic local systems, by  \emph{d\'{e}vissage} starting  from a $p$-adic Simpson functor. However, it is the Riemann--Hilbert functor that plays the key role in establishing \emph{de Rham rigidity} property in \emph{loc. cit.}, which in turn has important applications to Shimura varieties.

 \item In the   upcoming work \cite{GMWdRrel}, we will  discuss 
$\bbdrplus$-crystals in the relative setting, i.e., on the (log-) prismatic site of a smooth formal scheme over $\ok$. These results will be related to the work of \cite{LZ17}. 
\end{enumerate}
\end{rem}

\subsection{Crystals and log connections}\label{subsec1.2}
The key tool in establishing Theorems \ref{thmintrondr} and \ref{thmintrocoho} is the category of \emph{log connections}, which \emph{bridges} together $\bbdrplus$-crystals and Galois representations.
Let us first introduce some related notations.

\begin{notation}\label{notaBKprism}
\begin{enumerate}
\item  Recall in Notation \ref{notaeu}, we defined $\pi$ and $E(u)$.
Let $\gs=W(k)[[u]]$, and equip it with a Frobenius $\varphi$ extending  the absolute Frobenius on $W(k)$ and such that $\varphi(u)=u^p$. Then $(\gs, (E(u)) \in \okprism$, and is called the Breuil--Kisin prism (associated to $\pi$).
One can equip it with a log structure and obtain a log prism $(\gs, (E), M_\gs) \in \okprislog$, cf. Notation \ref{notaBKlogprism}.

\item Let $\pi_0=\pi$, and for each $n \geq 1$, inductive fix some $\pi_n$ so that $\pi_n^p=\pi_{n-1}$. This compatible sequence defines an element $\pi^\flat \in \ocflat$. We can define a morphism of prisms (resp. log prisms)
\[(\gs, (E)) \to (\ainf, (\xi)), \quad \text{ resp. } (\gs, (E), M_\gs) \to (\ainf, (\xi), M_\ainf)\]
which is a $W(k)$-linear map and sends $u$   to the Teichm\"uller lift $[\pi^\flat]$.
\end{enumerate}
\end{notation}

\begin{notation} \label{notafields}
 We   introduce some field notations. Let $\mu_1$ be a primitive $p$-root of unity, and inductively, for each $n \geq 2$, choose $\mu_n$ a $p$-th root of $\mu_{n-1}$. Define the fields
$$K_{\infty}   = \cup _{n = 1} ^{\infty} K(\pi_n), \quad K_{p^\infty}=  \cup _{n=1}^\infty
K(\mu_{n}), \quad L =  \cup_{n = 1} ^{\infty} K(\pi_n, \mu_n).$$
Let $$G_{\kinfty}:= \gal (\overline K / K_{\infty}), \quad G_{\kpinfty}:= \gal (\overline K / K_{p^\infty}), \quad G_L: =\gal(\overline K/L).$$
Further define $\Gamma_K, \hat{G}$ as in the following diagram, where we let $\tau$ be a topological generator of $\gal(L/\kpinfty) \simeq \zp$, cf. Notation \ref{nota hatG} for more details.
\[
\begin{tikzcd}
                                       & L                                                                                             &                             \\
\kpinfty \arrow[ru, "<\tau>", no head] &                                                                                               & \kinfty \arrow[lu, no head] \\
                                       & K \arrow[lu, "\Gamma_K", no head] \arrow[ru, no head] \arrow[uu, "\hat{G}"', no head, dashed] &
\end{tikzcd}
\]
\end{notation}

 Let $\lambda :=\prod_{n \geq 0} (\varphi^n(\frac{E(u)}{E(0)})) $ be an element in  $\bcrisplus$ and hence in $\bdrplus$. For now, let us mention that we regard $\lambda$ as an ``\emph{analogue in the Kummer tower $\kinfty$}"  of Fontaine's element $t$   in the cyclotomic tower $\kpinfty$, cf. Rem. \ref{remuselambda}.

\begin{defn}
\begin{enumerate}
\item
A (finite free) \emph{log-$\lambda$-connection}  over $K[[\lambda]]$  is a finite  free $K[[\lambda]]$-module $M$ together a $K$-linear map
$\nabla_M: M \to M$ satisfying $\lambda$-Leibniz law; namely,
$$\nabla_M(fm) =f\nabla_M(m)+{\lambda}\frac{d}{d{\lambda}}(f)m, \quad \forall f \in K[[{\lambda}]], m \in M.$$
Write $\mathrm{MIC}_{\lambda}(K[[{\lambda}]])$ for the category of such objects.

\item
For each $m \geq 1$, equip $K[[\lambda]]/\lambda^m \simeq \oplus_{i=0}^{m-1} K\cdot \lambda^i$ with induced topology, and equip $K[[\lambda]]=\varprojlim_{m \geq 1}K[[\lambda]]/\lambda^m$ with inverse limit topology.
Let $a \in K$.
Say a log-$\lambda$-connection $(M, \nabla_M)$ in is \emph{$a$-nilpotent} if
  \[
       \lim_{n\to+\infty}a^n\prod_{i=0}^{n-1}(\nabla_M-i) = 0
   \]
   with respect to the induced topology on $M$.
   Write $\mathrm{MIC}_{\lambda}^a(K[[{\lambda}]])$ for the category of such objects.
\end{enumerate}
\end{defn}

The goal of this subsection is to discuss the following (all-encompassing) commutative diagram, where we see several categories of log connections sitting in between categories in Thm. \ref{thmintrondr}.
\begin{equation} \label{eqintrodiag}
\begin{tikzcd}
{\Vect( \okpris, \bbdrplus)} \arrow[d, "\simeq"] \arrow[rr, hook]            &  & {\Vect( \okprislog, \bbdrplus)} \arrow[d, "\simeq"]                        &  &                                      \\
{\mic_\lambda^{-E'(\pi)}(K[[\lambda]]) } \arrow[d, "\simeq"] \arrow[rr, hook] &  & {\mic_\lambda^{-\pi E'(\pi)}(K[[\lambda]]) } \arrow[d, "\simeq"] \arrow[rr] &  & {\mic_\lambda(\kinfty[[\lambda]])  } \\
\rep_\gk^{\mathrm{ndR}}(\bdrplus) \arrow[rr, hook]                           &  & \rep_\gk^{\mathrm{lndR}}(\bdrplus) \arrow[rr, hook]                        &  & \rep_\gk(\bdrplus) \arrow[u]
\end{tikzcd}
\end{equation}
Here, the    category  $\mathrm{MIC}_{\lambda}(\kinfty[[{\lambda}]])$ is the analogously defined category; also notice the direction of the right most vertical arrow (which, unlike other vertical arrows, is not an equivalence).

  We first discuss the top left square in Diagram \eqref{eqintrodiag}.

\begin{theorem}\label{thm-intro-logconn}
 (cf. \S\ref{secdRlogconn}.)
We have a commutative diagram of tensor functors:
\[
\begin{tikzcd}
{\Vect( \okpris, \bbdrplus)} \arrow[d, "\simeq"] \arrow[rr, hook] &  & {\Vect( \okprislog, \bbdrplus)} \arrow[d, "\simeq"] \\
{\mic_\lambda^{-E'(\pi)}(K[[\lambda]]) } \arrow[rr, hook]          &  & {\mic_\lambda^{-\pi E'(\pi)}(K[[\lambda]]) }
\end{tikzcd}
\]
here the left resp. right vertical arrow is induced by  evaluation at Breuil--Kisin prism resp. Breuil--Kisin log prism. Both equivalences are bi-exact.
\end{theorem}

\begin{rem}
  Interestingly, Thm. \ref{thm-intro-logconn}   refines results in \cite[\S 7.2]{BS23} in a very conceptual way, cf. Rem. \ref{rembs2372}.
\end{rem}

\begin{rem}    \label{rem:variables}
 The very fact that evaluations of \emph{crystals} should induce certain  \emph{differential modules} (i.e., the log connections in this case) is a familiar phenomenon (say, in the classical crystalline theory), but we have had substantial difficulty in obtaining Theorem \ref{thm-intro-logconn}.
 A major obstacle is a ``correct" description of a related cosimplicial ring $\gs_\dR^{\bullet, +}$, cf. \S \ref{sec: cosimp rings}, which is the ``de Rham variant" of the cosimplicial ring $\gs^{\bullet}$ built from the Breuil--Kisin prism. It turns out one needs to carefully choose certain ``variables" in $\gs_\dR^{\bullet, +}$, ---that we discovered by trial and error--- which then would make ensuing   computations possible. (One interesting remark is that this difficulty vanishes when we consider Hodge--Tate crystals as in Thm. \ref{thmintroht}.)
\end{rem}

 We now discuss the right most vertical functor in Diagram \eqref{eqintrodiag}.

 \begin{theorem} \label{thm-intro-KSF}
 (cf. \S \ref{sec_kummersen}.)
Let $W \in \rep_\gk(\bdrplus)$, and define
 \begin{equation*}
 D^+_{\dif, \kinfty}(W):= (W^{G_L})^{\tau\dla, \gamma=1},
 \end{equation*}
where the right hand side consists of elements in $W^{G_L}$ which are fixed by $\gal(L/\kinfty)$ and has \emph{locally analytic action} by the $p$-adic Lie group $\gal(L/\kpinfty)$ (cf. \S \ref{seclav}).
Then it is a finite free $K_\infty[[\lambda]]$- module such that the natural map
\[ D^+_{\dif, \kinfty}(W)\otimes_{\kinftylambda} \bdrplus \to W\]
is an isomorphism.
In addition, the induced Lie algebra action from $\gal(L/\kpinfty)=\langle \tau \rangle$ can be normalized (by some scalar $\fkc$ not specified here) to induce a $\kinfty$-linear map
\[
\frac{1}{\fkc}\cdot \log \tau: D^+_{\dif, \kinfty}(W) \to D^+_{\dif, \kinfty}(W),
\]
which defines  an object in $\mathrm{MIC}_\lambda(\kinftylambda)$.
We call the operator $\frac{1}{\fkc}\cdot \log \tau$   the \emph{Sen--Fontaine operator over the Kummer tower}.
 \end{theorem}

This theorem generalizes (indeed, lifts) the  Sen theory over the Kummer tower developed in \cite{GMWHT}, which treated the case $W \in \rep_\gk(C)$. It also parallels Fontaine's theory which attaches to  $W \in \rep_\gk(\bdrplus)$ a ``log-$t$-connection" over $\kpinftyt$.

Now, let us discuss the main ideas where we use Theorems \ref{thm-intro-logconn} and \ref{thm-intro-KSF} to  deduce Thm. \ref{thmintrondr} (say, in the prismatic seting).
Firstly, the functor
\[ {\Vect( \okpris, \bbdrplus)} \to  \rep_\gk^{\mathrm{ndR}}(\bdrplus)\]
is easily constructed because of our known results in the Hodge--Tate case Thm. \ref{thmintroht}.
Hence we can use Theorem \ref{thm-intro-logconn} to pass to the following functor
\begin{equation} \label{eqktokinfty}
   \mic_\lambda^{-E'(\pi)}(K[[\lambda]]) \to  \rep_\gk^{\mathrm{ndR}}(\bdrplus)
\end{equation}
To show it is fully faithful, it turns out it suffices to compare the log connection over $K[[\lambda]]$ to those over  $K_\infty[[\lambda]]$ constructed in Thm. \ref{thm-intro-KSF}. This comparison is proved in Thm. \ref{thm-compare-M-Ddif}, via techniques in locally analytic vectors.
To prove the functor is essentially surjective, we start from the known Hodge--Tate case Thm. \ref{thmintroht} and use a d\'evissage argument. Recall in Remark \ref{rem_intro_extension}, we mentioned a difficulty to study  \emph{extensions of prismatic crystals}; this difficulty is salvaged, as in the functor \eqref{eqktokinfty}, we are dealing with actual \emph{log connections}, which are amenable to d\'evissage techniques.

\begin{remark}
In the final stage of our redaction (after we obtained all the main results),  Zeyu Liu informed us that he independently obtained some partial results (under the prismatic setting) discussed in this paper, cf. \cite{Liu23}.
For the reader's convenience, we make some rough comparisons:
\begin{itemize}
\item  Liu obtained Thm. \ref{thm-intro-KSF} (and   related results on  Sen--Fontaine theory), by generalizing  Sen theory over the Kummer tower developed in our earlier work; our approach on Thm. \ref{thm-intro-KSF}  is similar, although we adopt  an axiomatic treatment which will be of   use in our future work in the relative case \cite{GMWdRrel}.
\item Liu can prove full faithfulness results in Theorems \ref{thm-intro-logconn} and  \ref{thmintrondr}, but not   essential surjectivity therein.  (As far as we understand, Liu  uses  different ``variables" for the cosimplicial ring $\gs_\dR^{\bullet, +}$ (cf. Rem. \ref{rem:variables}), leading to  complications in his approach).
\end{itemize}
\end{remark}

\subsection{Structure of the paper}

Our paper can be roughly divided into two main parts. The first part, from \S \ref{sec:crystal and strat} to \S \ref{sec: perfect crystal and rep}, is the ``prismatic/stratification" part. The second part, from \S \ref{seclav} to \S \ref{sec: ndR rep}, is the ``pro-\'etale" part.

In \S \ref{sec:crystal and strat}, we quickly relate $\bbdrplus$-crystals with stratifications, which then motivates the following \S \ref{sec: cosimp rings} and \S \ref{sec: stra and conn}, where we carry out detailed computations of the relevant cosimplicial rings and stratifications. In particular, our treatment in  \S \ref{sec: stra and conn} is axiomatic, which relates certain stratifications to certain log connections. In  \S \ref{secdRlogconn}, we \emph{specialize} these  axiomatic computations to   $\bbdrplus$-crystals.
 In \S \ref{sec:compa dR coho}, (log-) prismatic cohomology of these crystals is compared with  Sen--Fontaine  cohomology  of the corresponding log connections.
In \S \ref{sec: perfect crystal and rep}, we discuss $\bbdrplus$-crystals on the perfect prismatic site, and relate them with $\bdrplus$-representations of the Galois group $\gk$.

In \S \ref{seclav}, which is the beginning of our ``pro-\'etale" part, we start by reviewing locally analytic vectors, using an axiomatic approach.
We then specialize these axiomatic computations in \S \ref{subsecwtbi} resp. \S \ref{sec_kummersen} to compute locally analytic vectors in rings resp. representations. In particular, the later builds a Sen--Fontaine theory over a Kummer tower.
In \S \ref{sec_compa_Galois}, this Sen--Fontaine theory will be \emph{compared} with the log connections  coming from $\bbdrplus$-crystals (in \S \ref{secdRlogconn}); this implies the \emph{pro-\'etale comparison} linking (log-) prismatic cohomology to Galois cohomology.
Finally in \S \ref{sec: ndR rep}, we classify $\bbdrplus$-crystals by (log-) nearly de Rham representations.


\subsection*{Acknowledgement}
We thank Heng Du, Bernard Le Stum, and Tong Liu for useful discussions and correspondences.
We thank Bhargav Bhatt for suggesting the terminology ``$\bbdrplus$-crystals."
H.G. is  partially supported by the National Natural Science Foundation of China under agreement No. NSFC-12071201; in addition, he gratefully acknowledges the support of the Infosys Member Fund at the Institute for Advanced Study for providing an excellent working condition where part of the work is carried out.
Y.M. is supported by China Postdoctoral Science Foundation E1900503.
Y.W. is partially supported by CAS Project for Young Scientists in Basic Research, Grant No. YSBR-032.
All authors further thank Beijing International Center for Mathematical Research and Morningside Center of Mathematics where part of the work is carried out.

\section{$\bbdrplus$-crystals and stratifications}\label{sec:crystal and strat}

In this section, we define $\bbdrplus$-crystals, and relate them with stratifications.

\begin{notation} cf. Def. \ref{nota intro}. Let $\ast \in \{\emptyset, \log\}.$
Let $\BBdRp$ be the presheaf on $(\calO_K)_{\Prism, \ast}$   by sending a prism $(A,I)$ (resp. a log prism $(A,I,M)$) to \[A[\frac{1}{p}]^{\wedge}_I = \varprojlim_{m\geq 1} A[\frac{1}{p}]/I^m A[\frac{1}{p}]. \]
 Then $\BBdRp$ is a sheaf on $(\calO_K)_{\Prism}$  resp. $\okprislog$ by \cite[Thm. 2.2, Rem. 2.4]{BS23} resp. \cite[Rem. 2.10]{MW22log}.
Similarly, for any $m\geq 1$, the pre-sheaf $\bB_{\dR,m}^+$ sending a prism $(A,I)$ (resp. log prism $(A,I,M)$) to $A[\frac{1}{p}]/I^mA[\frac{1}{p}]$ is also a sheaf on $(\calO_K)_{\Prism, \ast}$. In particular, we have $\bB_{\dR,1}^+ = \overline \calO_{\Prism}[\frac{1}{p}]$.
\end{notation}

\begin{convention}\label{conv_topo}
Using the (complete) $p$-adic topology on $A$, we can topologize $A[\frac{1}{p}]$ and hence $A[\frac{1}{p}]/I^m A[\frac{1}{p}]$; the ring  $A[\frac{1}{p}]^{\wedge}_I$ is then equipped with the inverse limit topology.
\end{convention}

\begin{defn} Let $\ast \in \{\emptyset, \log\}.$
  Call an object in 
  \[\text{$\Vect(\okprisast, \bbdrplus)$ resp. $\Vect(\okprisast, \bbdrplusm)$ }\]  a \emph{$\bbdrplus$-crystal}   resp. a \emph{(truncated) $\bbdrplus$-crystal of level $m$} on $\okprisast$.   Note when $m=1$,  these  are precisely the rational Hodge--Tate crystals studied in  \cite{GMWHT}.
\end{defn}

\begin{rem} \label{remnamebdrcrystal}
We comment on the terminology of $\bbdrplus$-crystals.
    \begin{enumerate}
        \item  Objects in $\vect(\okprisast, \mathbb{B}^+_{\dR, 1})$   are called Hodge--Tate crystals because they are related with the Hodge--Tate locus of the prismatization of $\spf \ok$, cf. \cite{Dri20, BL-a}; this locus is also sometimes called the Hodge--Tate stack of $\spf \ok$.
        
        \item Since $\Vect(\okprisast, \bbdrplus)$ is a natural generalization of  Hodge--Tate crystals, and in particular (as stated in Thm. \ref{thmintrondr}) are classified by nearly de Rham representations, one could also try to call them ``de Rham crystals". Here, we choose to use the more technical term ``$\bbdrplus$-crystals". For one reason, there is already a well-known ``de Rham stack" in algebraic geometry to study de Rham cohomology, whereas objects in $\Vect(\okprisast, \bbdrplus)$ give rise to vector bundles with flat $\lambda$-connections (as stated in Thm. \ref{thm-intro-logconn}) and hence are not directly related with de Rham cohomology (say, over $\spf \ok$). In particular, one should not expect to use the usual de Rham stack to study these crystals; rather, one probably should look for certain ``$\bbdrplus$-stack". 
For another reason, (particularly in the relative case in \cite{GMWdRrel}), pro-\'etale realizations of $\bbdrplus$-crystals are precisely the ``$\bbdrplus$-local systems" introduced in \cite[\S 7]{Sch13}.             
    \end{enumerate}
\end{rem}

We recall the definition of a stratification.
\begin{convention}
    Let $A^{\bullet}$ be a cosimplicial ring.
     For any $0\leq i\leq n+1$, let $p_i:A^n\to A^{n+1}$
     be the $i$-th face map   induced by the order-preserving map $[n]\to[n+1]$ whose image does not contain $i$. For any $0\leq i\leq n$, let $\sigma_i:A^{n+1}\to A^n$ be the $i$-th degeneracy map induced by the order-preserving map $[n+1]\to[n]$ such that the preimage of $i$ is $\{i,i+1\}$. For any $0\leq i\leq n$, let $q_i:A^0\to A^n$ be the structure map induced by the map $[0]\to [n]$ sending $0$ to $i$. 
\end{convention}
\begin{dfn}
For a cosimplicial ring $A^{\bullet}$, a \emph{stratification with respect to $A^{\bullet}$} is a pair $(M,\varepsilon)$ consisting of a finite projective $A^0$-module $M$ and an $A^1$-linear isomorphism
 \[\varepsilon: M\otimes_{A^0,p_0}A^1\to M\otimes_{A^0,p_1}A^1,\]
 such that the \emph{cocycle condition} is satisfied:
 \begin{enumerate}
 \item  $p_2^*(\varepsilon)\circ p_0^*(\varepsilon) = p_1^*(\varepsilon): M\otimes_{A^0,q_2}A^2\to M\otimes_{A^0,q_0}A^2$;
\item  $\sigma_0^*(\varepsilon) = \id_M$.
 \end{enumerate}
Write $\mathrm{Strat}(A^\bullet)$ for the category of stratifications with respect to $A^\bullet$.
\end{dfn}


 We now construct a cosimplicial ring using the Breuil--Kisin prism (resp. the Breuil--Kisin log prism).

 \begin{notation}\label{notaBKlogprism}
Let $(\gs, (E))$ be the Breuil--Kisin prism in Notation \ref{notaBKprism} (with respect to the chosen uniformizer $\pi$).
  Let $M_{\frakS}\to \frakS$ be the log structure associated to the pre-log structure $\bN\xrightarrow{1\mapsto u}\frakS$. Then $(\frakS,(E),M_{\frakS})$  a log prism in  $(\calO_K)_{\Prism,\log}$, and is called the Breuil--Kisin log prism. (We refer the readers to \cite{Kos21} for foundations of log prismatic cohomology).
 \end{notation}
 
 These prisms are important because of the following result.
 \begin{lem}[\emph{\cite[Lem. 2.2]{GMWHT}}]\label{Lem-BKCover}~
 \begin{enumerate}
    \item  The Breuil--Kisin prism $(\frakS,(E))$ is a cover of the final object of the topos $\Sh((\calO_K)_{\Prism})$.

  \item  The Breuil--Kisin log prism $(\frakS,(E),M_{\frakS})$ is a cover of the final object of the topos $\Sh((\calO_K)_{\Prism,\log})$.
 \end{enumerate}
 \end{lem}

\begin{notation} \label{notagssimp}
Let $(\frakS^{\bullet},(E))$ resp. $(\frakS^{\bullet}_{\log},(E),M^{\bullet})$ be the cosimplicial ring of the \v Cech nerve  of the Breuil--Kisin prism resp. the Breuil--Kisin log prism in the prismatic site resp. log-prismatic site.
By \cite[Exam. 2.6]{BS23} and \cite[Lem. 5.0.2]{DL23}, we see that for any $n\geq 0$,
 \begin{equation}\label{Equ-CechBK}
     \frakS^n = \rW(k)[[u_0,\dots,u_n]]\{\frac{u_0-u_1}{E(u_0)},\dots,\frac{u_0-u_n}{E(u_0)}\}^{\wedge}
 \end{equation}
 and
 \begin{equation}\label{Equ-CechBKlog}
      \frakS^n_{\log} = \rW(k)[[u_0,1-\frac{u_1}{u_0},\dots,1-\frac{u_n}{u_0}]]\{\frac{1-\frac{u_1}{u_0}}{E(u_0)},\dots,\frac{1-\frac{u_1}{u_n}}{E(u_0)}\}^{\wedge},
 \end{equation}
 with the obvious structure morphisms of cosimplicial rings.
 Clearly, there exists a natural morphism of cosimplicial rings $\frakS^{\bullet}\to\frakS_{\log}^{\bullet}$ which is given by identifying $u_i$'s.

 Let $\frakS^{\bullet,+}_{\dR, m}$ and $\frakS^{\bullet,+}_{\log,\dR, m}$ be the evaluations of $\bB_{\dR,m}^+$ at $\frakS^{\bullet}$ and $\frakS_{\log}^{\bullet}$ respectively.
 Namely, for each $n\geq 0$ and  each $m <+\infty$, we have
 \[\frakS_{\dR, m}^{n,+} =\frakS^n[\frac{1}{p}]/E(u_0)^m\frakS^n[\frac{1}{p}], \text{ resp. } \frakS_{\log, \dR,m}^{n,+} =  \frakS_{\log}^n[\frac{1}{p}]/E(u_0)^m\frakS_{\log}^n[\frac{1}{p}].\]
   In addition, define $\frakS^{\bullet,+}_{\dR}$ and $\frakS^{\bullet,+}_{\log,\dR}$ so that
 \[\frakS_{\dR}^{n,+} := \varprojlim_{m\geq 1}\frakS_{\dR, m}^{n,+}, \text{ resp. } \frakS_{\log, \dR}^{n,+} := \varprojlim_{m \geq 1}\frakS_{\log, \dR, m}^{n,+} .\]
 Using the topology discussed in Convention \ref{conv_topo}, all these cosimplicial rings are topological cosimplicial rings. In addition, the natural morphism
\[\frakS_{\dR, m}^{n,+} \to \frakS_{\log, \dR, m}^{n,+} \]
is continuous.
\end{notation}

\begin{convention} \label{convention-m-infty}
In many of our constructions, we allow $m$ to vary from $1$ to $+\infty$. The case $m=+\infty$ is always understood to mean the ``inverse limit" case. For example, in the following Prop. \ref{Prop-dRStratification},  we use
\[\Vect((\calO_K)_{\Prism},\bB_{\dR,+\infty}^+), \quad \text{ resp. }  \mathrm{Strat}(\frakS_{\dR,+\infty}^{\bullet,+})\]
to mean
\[\Vect((\calO_K)_{\Prism},\bB_{\dR}^+), \quad \text{ resp. }  \mathrm{Strat}(\frakS_{\dR}^{\bullet,+}).\]
\end{convention}

 Using the language of stratification, we get
 \begin{prop}\label{Prop-dRStratification}
   Let $1 \leq m\leq +\infty$.
   Evaluation at the Breuil--Kisin prism resp. the Breuil--Kisin log prism induces  an equivalence of categories (cf. Convention \ref{convention-m-infty})
   \[\Vect((\calO_K)_{\Prism},\bB_{\dR,m}^+) \xrightarrow{\simeq} \mathrm{Strat}(\frakS_{\dR,m}^{\bullet,+}) \]
   resp.
   \[\Vect((\calO_K)_{\Prism, \log},\bB_{\dR,m}^+) \xrightarrow{\simeq} \mathrm{Strat}(\frakS_{\log, \dR,m}^{\bullet,+}) \]
    Both equivalences are bi-exact.
 \end{prop}
  \begin{proof}
   The result is a standard consequence of Lemma \ref{Lem-BKCover}, as $\bB_{\dR,m}^+$ satisfies  $p$-complete faithfully flat descent \cite{Mat22}. Bi-exactness follow from similar argument as \cite[Prop. 2.6]{GMWHT}.
 \end{proof}

The above proposition motivates us to study the stratifications in detail. In \S \ref{sec: cosimp rings}, we will study the cosimplicial rings; then in \S \ref{sec: stra and conn}, we study the stratifications.

\begin{rmk}\label{remnobdrtheory}
One can similarly define the sheaf  $\BBdR$ by inverting $\mathcal{I}_\prism$ on $\bbdrplus$.
  It would be interesting to define and study $\BBdR$-crystals over $\okprism$ resp. $\okprislog$.
  However, we do not know if \cite[Prop. 2.7]{BS23} holds true for $\BBdR$, hence we do not know if we can use   stratifications to study $\BBdR$-crystals. This seems  an even more serious issue in the relative case. One possible solution is to introduce certain \emph{filtered $\BBdR$-crystals}; this will be discussed in \cite{GMWdRrel}.
\end{rmk}

\section{Stratifications I: structure of cosimplicial rings} \label{sec: cosimp rings}

 In this  section, we study the structure of the cosimplicial rings $\frakS^{\bullet}_{\dR}$ and $\frakS_{\log,\dR}^{\bullet}$.

 \begin{construction}
  We start with the case for $n = 0$. In this case, we have
 \[\frakS^+_{\dR} = \frakS^+_{\log,\dR} = \varprojlim_{m\geq 1}\frakS[\frac{1}{p}]/E^m\frakS[\frac{1}{p}].\]
 Note that $\frakS_{\dR}^+$ is a Noetherian local ring with the maximal ideal $(E)$, and can be regarded as a subring of $\bdrplus$.
  By standard commutative algebra, $\frakS_{\dR}^+$ is a discrete valuation ring with the local parameter $E$ and   residue field $\frakS_{\dR}^+/E\frakS_{\dR}^+ = K$. Thus, as  a subring  of $\bdrplus$, we have $\frakS_{\dR}^+ = K[[E]]$.
Indeed, $\gs_\dR^+= K[[A]]$, where $A$ is any uniformizer.
  \end{construction}

Recall in Notation \ref{notafields}, we defined an element $\lambda: = \prod_{n\geq 0}\varphi^n(\frac{E(u)}{E(0)})$, where $\varphi$ and the product is taken inside $\bcrisplus$; hence we can regard $\lambda$ as an element in $\bdrplus$.

\begin{lemma} \label{lemgsdeunif}
The elements $E(u),u-\pi, \lambda \in \bdrplus$ are all inside $\frakS_{\dR}^+$  and are   uniformizers of $\frakS_{\dR}^+$. Thus we have
\[\frakS_{\dR}^+ =K[[E]]=K[[u-\pi]]=K[[\lambda]].  \]
\end{lemma}
\begin{proof}
We have $\frac{E}{u-\pi} \in K[u] =K[u-\pi] \subset K[[u-\pi]]$, and clearly $E$ is a uniformizer in $K[[u-\pi]]$, hence $K[[E]]=K[[u-\pi]]$.
The element (and notation) $\lambda$, (as far as we know) is first defined in \cite[1.1.1]{Kis06}, and is an element that converges on the open unit disk (over $K_0$); in particular, taking Taylor expansion of $\lambda$ on the point $\pi$ shows that $\lambda$ is in $K[[u-\pi]]$, and is a uniformizer.
\end{proof}

\begin{remark}\label{rem3unif}
The three uniformizers in Lem. \ref{lemgsdeunif} are ``useful" in their own ways.
\begin{enumerate}
\item The uniformizer $E(u)$ is a generator of the ideal in the Breuil--Kisin prism, and hence can be regarded as the most ``natural" choice. This will mirror the case when we consider similar situations in the relative prismatic setting in \cite{GMWdRrel}, where in certain cases we need to work with an abstract generator $d$ for a general prism $(A, (d))$.
\item The uniformizer $u-\pi$, as we shall see in the sequel, particularly in \S \ref{subsec_fullcrystal_rep}, is \emph{computationally} the most convenient one. This is because the $G_K$-action on $u-\pi$ is easy to describe.
\item The uniformizer $\lambda$ will  be used in \S \ref{sec_kummersen} to construct  Sen--Fontaine theory over the Kummer tower. Note one could equally use $u-\pi$ in \S \ref{sec_kummersen}. Our choice is motivated by the (partially aesthetical) observation that $\lambda$ is the analogue of Fontaine's element $t$; cf. Rem. \ref{remuselambda} for some more comments.
\end{enumerate}
\end{remark}

Let $*\in\{\emptyset,\log\}$, the cosimplicial ring $\frakS_{\ast}^{\bullet}$ and hence $\frakS_{\ast, \dR}^{\bullet}$ are cosimplicial $W(k)$-algebras.
Note that $E(u_i)$ for different $i$'s generate the same ideal in $\frakS_{\ast}^{n}$ for each $n\geq 0$, hence the quotient cosimplicial ring $\frakS_{\ast}^{\bullet}[1/p]/E \simeq \frakS_{\ast, \dR}^{\bullet}/E$ is a cosimplicial $K$-algebra. We claim that there exists a canonical way to make $\frakS_{\ast, \dR}^{\bullet}$ a cosimplicial $K$-algebra. Indeed, we will show this in a more general setting.

Consider $(A,I)\in (\calO_K)_{\Prism}$. Its structure map $\ok \to A/I$ induces a map $K \to  (A/I)[\frac{1}{p}]$, which is either injective or zero. In either case, we still use $\pi$  to denote the image of $\pi \in \ok$ inside $A/I$ and $A/I[\frac{1}{p}]$.
Note also by deformation theory there exists a unique $W(k)$-algebra structure on $A$ making the following diagram commutative:
\[\begin{tikzcd}
W(k) \arrow[d] \arrow[r, dashed] & A \arrow[d] \\
\ok \arrow[r]                    & A/I
\end{tikzcd}\]
 thus, we shall regard $E(u) \in W(k)[u]$ also as a polynomial over $A[\frac{1}{p}]^\wedge_I$ in the following.


\begin{prop}
 Let $(A,I)\in (\calO_K)_{\Prism}$, and use notations above. There exists a unique lifting $\Pi\in A[\frac{1}{p}]^{\wedge}_I =\bbdrplus(A, I)$ of $\pi\in A/I[\frac{1}{p}]$ such that $E(\Pi) = 0$.
\end{prop}
\begin{proof}
  It suffices to consider the case when $(A/I)[1/p] \neq 0$ since otherwise
  \[A[\frac{1}{p}]^{\wedge}_I\subset IA[\frac{1}{p}]^{\wedge}_I\subset I^2A[\frac{1}{p}]^{\wedge}_I\subset \cdots,\]
  which implies $A[\frac{1}{p}]^{\wedge}_I=0$.

  We first assume $(A,I)$ is oriented; that is, $I = (d)$ is principal.
  Let $\tilde \pi \in A$ be a lifting of $\pi \in A/I$.
  Then we see that $E(\tilde \pi)\in dA[\frac{1}{p}]^{\wedge}_I$. We show that there exists a lifting $\Pi$ of $\pi$ satisfying $E(\Pi) = 0$ at first.

  Put $x_0=\tilde \pi$ and assume we have constructed $x_0,\dots,x_n$ for some $n\geq 0$ satisfying
  \begin{equation}\label{Equ-prop-KstructureI}
      E(x_0+x_1d+\cdots+x_nd^n)\in d^{n+1}A[\frac{1}{p}]^{\wedge}_I.
  \end{equation}
  We write
  \[E(x_0+x_1d+\cdots+x_nd^n) = d^{n+1}y_n\]
  for some $y_n\in A[\frac{1}{p}]^{\wedge}_I$.
  Then we construct $x_{n+1}$ as follows:

  Since $E'(\pi)$ is a unit in $K$ and hence in $A/I[\frac{1}{p}]$, we deduce that $E'(x_0+x_1d+\cdots+x_nd^n)$ is a unit of $ A[\frac{1}{p}]^{\wedge}_I$ as it goes to $E'(\pi)$ modulo $d$.
  Then for any $y\in A[\frac{1}{p}]^{\wedge}_I$, by Taylor's expansion, we have that
  \begin{equation*}
      \begin{split}
          E(x_0+x_1d+\cdots+x_nd^n+yd^{n+1}) & \equiv E(x_0+\cdots+x_nd^n)+E'(x_0+\cdots+x_nd^n)yd^{n+1} \mod d^{n+2}\\
          & \equiv d^{n+1}y_n +E'(x_0+\cdots+x_nd^n)yd^{n+1} \mod d^{n+2}.
      \end{split}
  \end{equation*}
  Since $E'(x_0+\cdots+x_nd^n)$ is a unit, if we put $x_{n+1} = -E'(x_o+\cdots+x_nd^n)^{-1}y_n$, then (\ref{Equ-prop-KstructureI}) holds true for $n+1$ as desired.

  Now, $\Pi: = \sum_{n\geq 0}x_nd^n$ is well-defined in $ A[\frac{1}{p}]^{\wedge}_I$ and for any $n\geq 0$,
  \[E(\Pi) \equiv E(\sum_{i=0}^nx_id^i)\equiv 0\mod  d^{n+1} A[\frac{1}{p}]^{\wedge}_I.\]
  So we obtain that $E(\Pi) = 0$.

  To complete the proof in the oriented case, we have to show the uniqueness of $\Pi$.
  Assume there exists another lifting $\tilde \Pi$ of $\pi$ such that $E(\tilde \Pi) = 0$.
  Then we have $\tilde \Pi-\Pi \in dA[\frac{1}{p}]^{\wedge}_I$.
  By Taylor's expansion, we see that
  \[0 = E(\tilde \Pi)-E(\Pi) = \sum_{i=1}^{\deg E}\frac{E^{(i)}(\Pi)}{i!}(\tilde \Pi-\Pi)^n = E'(\Pi)(\tilde \Pi-\Pi)(1+(\tilde \Pi-\Pi)z)\]
  for some $z\in  A[\frac{1}{p}]^{\wedge}_I$.
  By noting that $E'(\Pi)$ is a unit in $ A[\frac{1}{p}]^{\wedge}_I$, we deduce that
  \[(\tilde \Pi-\Pi)(1+(\tilde \Pi-\Pi)z) = 0.\]
  Since $\tilde \Pi-\Pi \in dA[\frac{1}{p}]^{\wedge}_I$, for any $m\geq 1$, $1+(\tilde \Pi-\Pi)z$ is a unit   in $A[\frac{1}{p}]^{\wedge}_I/d^mA[\frac{1}{p}]^{\wedge}_I$. So $1+(\tilde \Pi-\Pi)z$ defines a unit in $A[\frac{1}{p}]^{\wedge}_I$. This implies that $\tilde \Pi = \Pi$ as desired.

  Now, we are able to deal with the general case. Assume $(A,I)$ is a prism. Then by \cite[Lem. 3.1(3)]{BS22}, there exists an oriented prism $(B,IB = (d))$ covering  $(A,I)$ in $(\calO_K)_{\Prism}$. Let $(B^{\bullet},dB^{\bullet})$ be the \v Cech nerve of this cover and let $\Pi_{B^{\bullet}}$ be the unique lifting of $\pi$ in $B^{\bullet}[\frac{1}{p}]^{\wedge}_I$ which is a root of $E(u)$ as shown above.
 By uniqueness criterion, the images of $\Pi_B$ under face maps $p_0,p_1:B\to B^1$ are $\Pi_{B^1}$. As $\BBdRp$ is a sheaf,  $\BBdRp(A, I)$ is the equalizer of $\BBdRp(B, I) \rightrightarrows \BBdRp(B^1, I)$.
Thus $\Pi_B$ lifts to a \emph{unique} element $\Pi_A \in \BBdRp(A, I)$ satisfying $E(\Pi_A) = 0$ (as $E(\Pi_B) = 0$). Any $\Pi_A$ satisfying our proposition has to   map to the unique $\Pi_B$ and $\Pi_{B^1}$, and hence has to be unique itself, again since $\bbdrplus$ is a sheaf.
\end{proof}

\begin{rmk}
    One can show that $A[1/p]^{\wedge}_I$ is a $K$-algebra in another way: Let $\calO_{\Prism,\calO_K}$ be the presheaf on $(\calO_K)_{\Prism}$ sending each bounded prism $(A,I)$ to $A\otimes_{\rW(k)}\calO_K$. As $\calO_K$ is faithfully flat over $\rW(k)$, we see $\calO_{\prism,\calO_K}$ is indeed a sheaf. Let $\calJ_{\Prism}$ be the kernel of the canonical surjection $\calO_{\Prism,\calO_K}\to\overline \calO_{\Prism}$, which is then a sheaf on $(\calO_K)_{\Prism}$ sending each bounded prism $(A,I)$ to $J:=\Ker(A\otimes_{\rW(k)}\calO_K\to A/I)$. By \cite[Th. 5.8]{Mat22}, we see that $\bB_{\dR,\Prism,K}^+:=\calO_{\Prism,\calO_K}[1/\pi]^{\wedge}_{\calJ_{\Prism}}$ is also a sheaf, which carries each prism $(A,I)$ to $(A\otimes_{\rW(k)}K)^{\wedge}_J$. Clearly, we have a canonical morphism $\iota:\bB_{\dR,\Prism}^+\to \bB_{\dR,\Prism,K}^+$. We claim this is an isomorphism. Indeed, as $(\frakS,(E))$ is a covering of final object of $\Shv((\calO_K)_{\Prism})$, it is enough to check $\iota$ induces an isomorphism at any prism $(A,I)$ over $(\frakS,(E))$. In this case, by noting that $J$ is principally generated by $u-\pi$ (via identifying $u$ with its image along $\frakS\to A$), one can conclude by using the same argument in the proof of \cite[Prop. 8.12]{Col02}.
\end{rmk}
 
\begin{cor}
For  $*\in\{\emptyset,\log\}$, $\frakS_{\ast, \dR}^{\bullet}$ is a cosimplicial $K$-algebra lifting the cosimplicial $K$-algebra $\frakS_{\ast}^{\bullet}[1/p]/E \simeq \frakS_{\ast, \dR}^{\bullet}/E$.
\end{cor}
\begin{proof}
Thanks to above proposition, for any prism $(A,I)\in (\calO_K)_{\Prism}$, there exists a unique way to endow $A[\frac{1}{p}]^{\wedge}_I$ with a structure of $K$-algebra which is compatible with that on $A/I[\frac{1}{p}]$. In particular, all morphisms $A[\frac{1}{p}]^{\wedge}_I\to B[\frac{1}{p}]^{\wedge}_I$ induced by morphisms $(A,I)\to (B,IB)$ in $(\calO_K)_{\Prism}$ are $K$-linear.
Similar argument applies to $\okprislog$ as well.
\end{proof}

 Now we are going to study the structure  of cosimplicial ring  $\frakS_{\ast,\dR}^{\bullet,+}$. We begin with a lemma.

\begin{lemma}[\emph{\cite[Prop. 2.12]{GMWHT}}]  \label{Lem-reduced structure}~
Let $\ast \in\{\emptyset,\log\}$. Let $n\geq 0$ and let $1\leq i\leq n$.
 Define elements in $\gs_\ast^n$ (cf. Notation \ref{notagssimp})
\begin{equation}
x_i =
\begin{cases}
  \frac{u_0-u_i}{E(u_0)}    & \text{ if } \ast =\emptyset  \\
 \frac{1-\frac{u_i}{u_0}}{E(u_0)} & \text{ if } \ast =\log
\end{cases}
\end{equation}
By abuse of notations, still use $x_i$ to denote its reduction in $\gs_{\ast, \dR}^{+, n}/E(u_0)\gs_{\ast, \dR}^{+, n}$.
Then the natural map
\begin{equation}\label{eq-map-HT-simp-ring}
  \left( \calO_K\{x_1,\dots,x_n\}^{\wedge}_{\pd} \right) [\frac{1}{p}] \to \frakS_{\ast, \dR}^{n,+}/E(u_0)\frakS_{\ast,\dR}^{n,+}
\end{equation}
is an  isomorphism of topological rings (cf. Convention \ref{conv_topo}).
Here  $\calO_K\{x_1,\dots,x_n\}^{\wedge}_{\pd}$ denotes the $p$-completion of the free pd-polynomial rings over $\calO_K$ generated by  $x_i$'s.
\end{lemma}

We now construct suitable variables "$X_i$'s" for our cosimplicial rings $\gs^\bullet_{\ast, \dR}$.

\begin{construction}\label{construction:variable}
Let  $A(u)$ (regarded   as a formal power series in $K[[u-\pi])$ be a uniformizer of $\frakS_{\dR}^+$.
Note that  $A'(u)=\frac{d A}{d(u-\pi)}$ is a unit in $\frakS_{\dR}^+$. So we see that
 \begin{equation*}
     \begin{split}
         \frac{A(u_0)-A(u_i)}{A(u_0)} & = -\frac{\sum_{k\geq 1}\frac{A^{(k)}(u_0)}{k!}(u_i-u_0)^k}{A(u_0)}\\
         & = \frac{u_0-u_i}{u_0-\pi}\cdot \frac{\sum_{k\geq 1}\frac{A^{(k)}(u_0)}{k!}(u_i-u_0)^{k-1}}{\sum_{l\geq 1}\frac{A^{(l)}(\pi)}{l!}(u_0-\pi)^{l-1}}
     \end{split}
 \end{equation*}
 where the two infinite summations in the final expression are both units in $\frakS_{\dR}^{n,+}$. Thus we have
 \begin{equation}\label{Equ-Congruence-I}
     \frac{A(u_0)-A(u_i)}{E'(\pi)A(u_0)}\equiv \frac{u_0-u_i}{E'(\pi)(u_0-\pi)}\equiv \frac{u_0-u_i}{E(u_0)}\mod E(u_0)\frakS_{\dR}^{n,+}.
 \end{equation}
and
 \begin{equation}\label{Equ-Congruence-II}
 \frac{A(u_0)-A(u_i)}{\pi E'(\pi)A(u_0)} \equiv \frac{u_0-u_i}{\pi E'(\pi)(u_0-\pi)}\equiv     \frac{1-\frac{u_i}{u_0}}{E(u_0)} \mod E(u_0)\frakS_{\log,\dR}^{n,+}.
 \end{equation}
\end{construction}

\begin{prop} \label{Prop-structure}
Let $\ast \in\{\emptyset,\log\}$.
Let
\begin{equation*}
a=
\begin{cases}
  -E'(\pi), &  \text{if } \ast=\emptyset \\
 -\pi E'(\pi), &  \text{if } \ast=\log
\end{cases}
\end{equation*}
Let $n\geq 0$ and let $1\leq i\leq n$. Let $A(u)$ be any uniformizer in $\frakS_{\dR}^+$ as above.
Define elements in $\gs_{\ast, \dR}^{n, +}$
\[  X_i =X_{A, i} = \frac{A(u_i)-A(u_0)}{aA(u_0)} . \]
Hence in concrete expressions
\begin{equation}
X_i =
\begin{cases}
      \frac{A(u_0)-A(u_i)}{E'(\pi)A(u_0)}    & \text{ if } \ast =\emptyset  \\
     \frac{A(u_0)-A(u_i)}{\pi E'(\pi)A(u_0)}  & \text{ if } \ast =\log.
\end{cases}
\end{equation}
Then there is a natural map
\begin{equation} \label{eq-map-dr-simp-ring}
  \varprojlim_{m}((\calO_K[[A(u_0)]]/(A(u_0))^m)\{X_1,\dots,X_n\}^{\wedge}_{\pd}[\frac{1}{p}])  \to \gs_{\ast, \dR}^{n, +},
\end{equation}
which is an isomorphism of topological rings.
In addition,  the face maps  $p_i:\frakS_{\ast, \dR}^{n,+}\to \frakS_{\ast, \dR}^{n+1,+}$ and the degeneracy maps $\sigma_i:\frakS_{\ast, \dR}^{n+1,+}\to\frakS_{\ast, \dR}^{n,+}$, which is defined by switching indices of the  $u_i$'s, can now be expressed by the following formulae, using the $X_i$'s.
 \begin{equation}\label{Equ-Face}
        \begin{split}
           & p_i(X_j) = \left\{\begin{array}{rcl}
                (X_{j+1}-X_1)(1+aX_1)^{-1}, & i=0 \\
                X_j, & j<i\\
                X_{j+1}, & 0<i\leq j;
            \end{array}\right.\\
           & p_i(A(u_0))=\left\{\begin{array}{rcl}
               A(u_{1}) = A(u_0)(1+aX_1),  &  i = 0\\
               A(u_0),  & i>0;
            \end{array}\right.
        \end{split}
    \end{equation}
    \begin{equation}\label{Equ-Degeneracy}
        \begin{split}
           & \sigma_i(X_j)=\left\{\begin{array}{rcl}
               0, &  i=0,j=1\\
               X_j, & j\leq i\\
               X_{j-1}, & j>i;
           \end{array}\right.\\
           &\sigma_i(A(u_0)) = A(u_0), \qquad \forall i.
        \end{split}
    \end{equation}
\end{prop}
\begin{proof}
The most difficult part is the existence of the map \eqref{eq-map-dr-simp-ring}, which requires us to analyse boundedness of the sequence $X_j^{[n]}$; this will be done in  Lemma \ref{Structure map} below. Note for different uniformizers $A\in\frakS_{\dR}^+$, by computations in Construction \ref{construction:variable}, the $X_{A,i}$'s are different by an element in $1+\mathfrak{m}_{\frakS_{\dR}^+}$, where $\mathfrak{m}_{\frakS_{\dR}^+}$ denotes the maximal ideal of $\frakS_{\dR}^+$.  
Thus, for much of the analysis, it suffices to treat just  one uniformizer.

With the boundedness result of   Lemma \ref{Structure map} below, the map \eqref{eq-map-dr-simp-ring} is well-defined.
 The reduction modulo $E(u_0)$ (equivalently, modulo $A(u_0)$) of the map  induces the isomorphism \eqref{eq-map-HT-simp-ring}, because of the equalities (modulo $E(u_0)$) in \eqref{Equ-Congruence-I} and \eqref{Equ-Congruence-II}. 
 Thus   derived Nakayama lemma implies \eqref{eq-map-dr-simp-ring}  is also an isomorphism (note both sides of \eqref{eq-map-dr-simp-ring}  are $E(u_0)$-torsion free).

The structural (face and degeneracy) maps are easy to verify. The only small computation is the expression for $p_0(X_j)$. Note
   \begin{equation*}
            p_0(X_j)  = \frac{A(u_{j+1})-A(u_{1})}{aA(u_1)}
=\frac{A(u_{j+1})-A(u_0)  -(A(u_1) -A(u_0))   }{ aA(u_0)  } \cdot \frac{A(u_0)}{A(u_1)}
             = (X_{j+1}-X_1)\cdot \frac{A(u_0)}{A(u_1)}
    \end{equation*}
    and then note
    \begin{equation}\label{Equ-Face-II}         \frac{A(u_1)}{A(u_0)} = (1+aX_1).
    \end{equation}
    (This last simple expression is written as it will be used later several times).
\end{proof}

\begin{cor}
The natural map $\frakS_{\dR}^{\bullet,+}\to \frakS_{\log,\dR}^{\bullet,+}$ of cosimplicial rings is determined such that in degree $n$, the element ``$X_i$" (with $a=-E'(\pi)$) in $\frakS_{\dR}^{n,+}$ maps to ``$\pi X_i$" (with $a=-\pi E'(\pi)$) in $\frakS_{\log,\dR}^{n,+}$.
\end{cor}
\begin{proof}
This is obvious, as the cosimplicial ring map is induced by identifying $u_i$'s.
\end{proof}

The following technical lemma is used to construct the map   \eqref{eq-map-dr-simp-ring} in Prop. \ref{Prop-structure}.
\begin{lem}\label{Structure map}
Let $*\in \{\emptyset,\log\}$. Let  $1\leq j\leq i$. Let $m\geq 1$.
Consider the ring 
\[\gs_{\ast, \dR}^+/E(u_0)^m=\bB_{\dR}^+(\frakS_*^i)/E(u_0)^m=\frakS_*^i/E(u_0)^m[\frac{1}{p}].\]
\begin{enumerate}
    \item Let \begin{equation*}
v_{j,*}=
\begin{cases}
    u_0-u_j   & \text{ if } \ast =\emptyset  \\
    1-\frac{u_j}{u_0} & \text{ if } \ast =\log
\end{cases}
\end{equation*}
Let $$Y_{j,*}=\frac{v_{j,*}}{E(u_0)}\in \frakS_{*}^i.$$
Then the image of $\{Y_{j,*}^{[n]}\}$ is bounded in $\gs_{\ast, \dR}^+/E(u_0)^m$. 

    \item For any uniformizer $A(u)\in \frakS_{\dR}^+$, the sequence $\{X_j^{[n]}\}_{n\geq 1}$ for $X_j=X_{A,j}=\frac{A(u_0)-A(u_j)}{aA(u_0)}$ introduced in Proposition \ref{Prop-structure} is also bounded in $\gs_{\ast, \dR}^+/E(u_0)^m$.
\end{enumerate}
\end{lem}
\begin{proof}
Item (1). Consider the sequence $\{Y_{j,*}^{[n]}\}$.
Without loss of generality, we may assume $i=j=1$, and put $Y=Y_{1,*}$ for short.
Consider the subring $\frakS_*^1[\frac{E(u_0)}{p^2}]\subset \frakS_*^1[\frac{1}{p}]$. 

We claim: for any $n\geq 1$, there exists some $\epsilon_n\in\{\pm 1\}$ and $y_n\in \frakS^1_*[\frac{E(u_0)}{p^2}]$ such that 
\begin{equation}\label{eqnboundness}
    \frac{Y^{p^l}}{p^{1+p+\cdots+p^{l-1}}} = \epsilon_l\delta_l(Y)+\frac{E(u_0)}{p}y_l.
\end{equation}
Granting the claim, we conclude that $Y^{[n]}\in\frakS_*^1[\frac{E(u_0)}{p^2}]$ for all $n$. Now consider their images in $\frakS^1_*[\frac{E(u_0)}{p^2}]/E(u_0)^m[\frac{1}{p}]$. We can see that all $Y^{[n]}$ lie in $\frac{1}{p^{2m-2}}\frakS^1_*/E(u_0)^m$ as $(\frac{E(u_0)}{p^2})^n$ is 0 in $\frakS_*^1/E(u_0)^m[\frac{1}{p}]$ for all $n\geq m$. Then the result follows.

It remains to prove the claim. By \cite[Lem. 2.11]{GMWHT}, for any $n\geq 0$, there exists some $z_n\in\frakS_*^1$ such that
\[\varphi(\delta_{n+1}(Y)) = \delta_{n}(Y)^p+p\delta_{n}(Y) = E(u_0)z_n,\]
which implies that for any $n\geq 0$,
\begin{equation}\label{eqnboundness-key equality}
    \frac{\delta_{n}(Y)^p}{p}=-\delta_{n+1}(Y)+\frac{E(u_0)}{p}z_n.
\end{equation}
In particular, by letting $n=0$ in (\ref{eqnboundness-key equality}), we see that (\ref{eqnboundness}) is true when $l=1$ for $\epsilon_1=-1$ and $y_1=z_0$. Now assume (\ref{eqnboundness}) holds true for some $l = n \geq 1$. Then by taking the $p$-th power on both sides of (\ref{eqnboundness}) and applying (\ref{eqnboundness-key equality}), we have
\[\begin{split}
\frac{Y^{p^{n+1}}}{p^{1+p+\cdots+p^n}} = & \epsilon_n^p\frac{\delta_n(Y)^p}{p}+\sum_{i=1}^{p-1}\frac{\binom{p}{i}}{p}\epsilon_n^{p-i}\delta_n(Y)^{p-i}\frac{E(u_0)^i}{p^i}y_n^i+\frac{E(u_0)}{p^{p+1}}y_n\\
= & -\epsilon_n^p\delta_{n+1}(Y)+\frac{E(u_0)}{p}z_n+\sum_{i=1}^{p-1}\frac{\binom{p}{i}}{p}\epsilon_n^{p-i}\delta_n(Y)^{p-i}\frac{E(u_0)^i}{p^i}y_n^i+\frac{E(u_0)^p}{p^{p+1}}y_n^p.
\end{split}\]
This implies that (\ref{eqnboundness}) is true for $l=n+1$ with
\[\epsilon_{n+1}=-\epsilon_n^p \text{ and }
y_{n+1} = z_n+\sum_{i=1}^{p-1}\frac{\binom{p}{i}}{p}\epsilon_n^{p-i}\delta_n(Y)^{p-i}\frac{E(u_0)^{i-1}}{p^{i-1}}y_n^i+\frac{E(u_0)^{p-1}}{p^{p}}y_n^p.\]
By induction, we see that the claim above is true.

Item (2). Again, it suffices to consider the sequence $X_1^{[n]}$. Write $X_1 = Y_{1,*}(1+\sum_{i\geq 1}a_iE(u_0)^i)$ with $a_i\in K$ (cf. Construction \ref{construction:variable}). Then for any $n\geq 0$, 
\[X_1^{[n]} = Y_{1,*}^{[n]}(1+\sum_{i\geq1}a_iE(u_0)^i)^n.\]
So to conclude, we have to show $\{(1+\sum_{i\geq 1}a_iE(u_0)^i)^n\}_{n\geq 1}$ is bounded in $\frakS/(E(u_0)^m)[\frac{1}{p}]$. Fix an $N=N(m)\geq 1$ so that $a_i\in\frac{1}{p^N}\calO_K$ for any $1\leq i\leq m-1$. By the expansion of $(1+\sum_{i\geq 1}a_iE(u_0)^i)^n$, it is clear that for any $n\geq 1$ and any $1\leq j\leq m-1$, there are polynomials $F_{n,j}(T_1,\dots,T_{m-1})\in\bZ[T_1,\dots,T_{m-1}]$ with the degree at most $m-1$ such that
\[(1+\sum_{i\geq 1}a_iE(u_0)^i)^n=1+\sum_{j=1}^{m-1}F_{n,j}(a_1,\dots,a_{m-1})E(u_0)^j.\]
As $F_{n,j}(a_1,\dots,a_{m-1})\in p^{-(m-1)N}\calO_K$ for all $n\geq 1$, we see that $\{(1+\sum_{i\geq 1}a_iE(u_0)^i)^n\}_{n\geq 1}$ is bounded as desired.
\end{proof}

 \section{Stratifications II: axiomatic computation of log connections}\label{sec: stra and conn}
As we have seen in Prop. \ref{Prop-structure}, the de Rham cosimplicial rings are rather similar in both the prismatic and log-prismatic setting. Hence in this section, we develop an axiomatic approach to compute stratifications on these cosimplicial rings. In next section \S \ref{secdRlogconn}, these axiomatic computations will be specialized to treat the $\bbdrplus$-crystals.

\subsection{Stratifications and recurrence relations}

\begin{notation}\label{Notation:comsimplicialring}
Let $R$ be a $p$-adically complete and flat $\zp$-algebra with $a\in R$ a nonzero divisor which is invertible in $R[\frac{1}{p}]$. Let $T=T_0$ and $T_i, i  \geq 1$ be a sequence of independent variables.
Let \[X_i=\frac{T_i-T}{aT}.\]
Define a (topological) cosimplicial $R[\frac{1}{p}]$-algebra $S^\bullet=S^\bullet(R, a)$ where for each $n \geq 0$,
\[ S^n :=\projlim_m \left(R[[T]]/T^m\{X_1, \cdots, X_n\}_{\mathrm{pd}}^\wedge[1/p]    \right) \]
So in particular, $S^0=(R[1/p])[[T]]$.
The structure maps in the cosimplicial ring are   induced by the simplex maps with respect to    indices of the $T_i$'s. Thus, we have
\begin{equation} \label{newEqu-Face}
        \begin{split}
           & p_i(X_j) = \left\{\begin{array}{rcl}
                (X_{j+1}-X_1)(1+aX_1)^{-1}, & i=0 \\
                X_j, & j<i\\
                X_{j+1}, & 0<i\leq j;
            \end{array}\right.\\
            & p_i(T)=\left\{\begin{array}{rcl}
               T_{1} = T(1+aX_1),  & i=0\\
               T,  & i>0;
            \end{array}\right.
        \end{split}
    \end{equation}
    \begin{equation} \label{newEqu-StratificationD}
        \begin{split}
           & \sigma_i(X_j)=\left\{\begin{array}{rcl}
               0, &  i=0,j=1\\
               X_j, & j\leq i\\
               X_{j-1}, & j>i;
           \end{array}\right.\\
           & \sigma_i(T) = T, \qquad \forall i.
        \end{split}
    \end{equation}
    Note $(1+aX_1)^{-1} = \sum_{n\geq 0}n!(-a)^nX_1^{[n]}$ is well-defined as $R$ is $p$-complete.
\end{notation}

\begin{notation}
Note $S^n$ is $T$-torsionfree since it can be regarded as a subring of $R[[T,X_1,\dots,X_n]]$.
Note for any $0\leq j\leq n$, $T_j = T(1+aX_j)$. So we have $TS^n = T_jS^n$ and hence that $S^{\bullet}/T^mS^{\bullet}$ is a well-defined cosimplicial $R[1/p]$-algebra for any $m\geq 1$. Similar as in Convention \ref{convention-m-infty}, we define $S^{\bullet}/T^mS^{\bullet}:=S^{\bullet}$ for $m= +\infty$.
\end{notation}

\begin{example}\label{example_verify_cosimpring}
Let $A(u)$ be a uniformizer of $\frakS_{\dR}^+$ as in Prop. \ref{Prop-structure}. The map sending $T_i$ to $A(u_i)$ induces the following isomorphisms:
\begin{align*}
S^\bullet(\ok, -E'(\pi)) & \simeq  \gs_{\dR}^\bullet \\  S^\bullet(\ok, -\pi E'(\pi))& \simeq \gs_{\log, \dR}^\bullet
\end{align*}
\end{example}

\begin{notation}
Let $(M,\varepsilon) \in \mathrm{Strat}(S^\bullet/T^mS^{\bullet})$ be a stratification for some fixed $1\leq m\leq +\infty$. Via embedding $M\hookrightarrow M\otimes_{S^0,p_0}S^1$, for any $x\in M$, we write
 \begin{equation}\label{Equ-Epsilon}
     \varepsilon(x) = \sum_{n\geq 0}\phi_n(x)X_1^{[n]}
 \end{equation}
with $\{\phi_n\}_{n\geq 0}$ a collection of $R$-linear endomorphism of $M$ satisfying $\lim_{n\to+\infty}\phi_n = 0$ with respect to the topology on $M$.
\end{notation}

\begin{prop}\label{prop: strat iff}
The formula  \eqref{Equ-Epsilon}   gives a stratification satisfying the cocycle condition if and only if the following conditions are satisfied:
\begin{enumerate}
\item $\phi_0 = \id_M$;
\item $\phi_n =\prod_{i=0}^{n-1}(\phi_1 - ia)$ for each $n \geq 2$; and
\item   $\phi_n$ converges to zero as $n \to +\infty$.
\end{enumerate}
In this case, $\phi_1$ satisfies that for any $f(T)\in S^0$ and any $x\in M$,
\begin{equation}\label{Equ-strat iff}
  \phi_1(f(T)x) = f(T)\phi_1(x)+aTf'(T)x,
\end{equation}
and $\varepsilon$ satisfies that for any $x\in M$,
\begin{equation}\label{Equ-strat iffII}
    \varepsilon(x) = (1+aX_1)^{a^{-1}\phi_1}(x) = \sum_{n\geq 0}\left(\prod_{i=0}^{n-1}(\phi_1-ia)(x)\right)\cdot X_1^{[n]}.
\end{equation}
\end{prop}
 \begin{proof}
 We fix an $x\in M$.
By (\ref{newEqu-StratificationD}), it is easy to see that
 \begin{equation}
     \sigma_0^*(\varepsilon)(x) = \phi_0(x).
 \end{equation}
  By (\ref{newEqu-Face}), we obtain that
 \begin{equation*}
     \begin{split}
         p_2^*(\varepsilon)\circ p_0^*(\varepsilon)(x) & = p_2^*(\varepsilon)(\sum_{n\geq 0}\phi_n(x)(1+aX_1)^{-n}(X_2-X_1)^{[n]})\\
          & = \sum_{m,n\geq 0}\phi_m(\phi_n(x))X_1^{[m]}(1+aX_1)^{-n}(X_2-X_1)^{[n]})\\
          & = \sum_{l,m,n\geq 0}\phi_l(\phi_{m+n}(x))(1+aX_1)^{-m-n}(-1)^mX_1^{[l]}X_1^{[m]}X_2^{[n]}\\
          & = \sum_{l,m,n\geq 0}\phi_l(\phi_{m+n}(x))(1+aX_1)^{-m-n}(-1)^m\binom{l+m}{l}X_1^{[l+m]}X_2^{[n]}.
     \end{split}
 \end{equation*}
 and that
 \begin{equation*}
     \begin{split}
         &p_1^*(\varepsilon)(x)=\sum_{n\geq 0}\phi_n(x)X_2^{[n]}.
     \end{split}
 \end{equation*}
 So $p_2^*(\varepsilon)\circ p_0^*(\varepsilon) = p_1^*(\varepsilon)$ if and only if for any $n\geq 0$,
 \begin{equation}\label{Equ-stratificationF}
     \phi_n(x) = \sum_{l,m\geq 0}\phi_l(\phi_{m+n}(x))(1+aX_1)^{-m-n}(-1)^m\binom{l+m}{l}X_1^{[l+m]}.
 \end{equation}
By the technical Lemma \ref{Key lemma} below, this inductive relation can be solved. (We postpone the long proof of  Lemma \ref{Key lemma} to the end of this section.)
Indeed, we have for each $n \geq 2$,
\begin{equation}
\phi_n(x) =\prod_{i=0}^{n-1}(\phi_1-ia)(x).
\end{equation}

Finally, we have to show (\ref{Equ-strat iff}) is true when $(M,\varepsilon)$ satisfies the cocycle condition. Indeed, by (\ref{Equ-Epsilon}), we have
\begin{equation*}
    \begin{split}
        f(T)x + \phi_1(f(T)x)X_1 & \equiv \varepsilon(f(T)x) \mod X_1^{[\geq 2]}\\
        & = \varepsilon(x\otimes_{S^0,p_0}p_0(f(T_1)))\\
        & = \varepsilon(x)f(T(1+aX_1))\\
        & \equiv (x+\phi_1(x)X_1)f(T(1+aX_1)) \mod X_1^{[\geq 2]}\\
        & \equiv  (x+\phi_1(x)X_1)(f(T)+aTf'(T)X_1) \mod X_1^{[\geq 2]}\\
        & \equiv f(T)x+(f(T)\phi_1(x)+aTf'(T)x)X_1 \mod X_1^{[\geq 2]}.
    \end{split}
\end{equation*}
This implies (by considering coefficient of $X_1$) that $\phi_1(f(T)x) = f(T)\phi_1(x)+aTf'(T)x$ as desired.
 \end{proof}


\subsection{Review of log connections}
We   recall the notion of a log connection (also  called ``regular connection" in the literature; here we follow the terminology of \cite[7.16]{BS23}).

\begin{defn} \label{deflogconn}
Let $F$ be a ring, and let $T$ be a variable.
Let $\Omega^1_{F[[T]]/F, \log}$ be the module of log differentials, which is a rank-1 free $F[[T]]$-module with $d\log T$ as a generator.
\begin{enumerate}
\item Recall a log connection is a finitely generated $F[[T]]$-module $M$ together a $F$-linear additive map
\[ \nabla_M: M \to M\otimes_{F[[T]]} \Omega^1_{F[[T]]/F, \log} \]
satisfying the \emph{Leibniz rule}; that is, for any $f(T)\in F[[T]]$ and any $x\in M$,
\[\nabla_M(f(T)x) = f(T)\nabla_M(x)+ x\otimes T\frac{d}{dT}f(T)  \cdot d\log T.\]
Denote the category of such objects as $\mic^\fg_\log(F[[T]])$.



\item Suppose now $F$ is a topological ring. For $m \geq 1$, equip $F[T]/T^m =\oplus_{i=0}^{m-1} F\cdot T^i \simeq F^m$ with the induced topology; equip $F[[T]] =\projlim_m F[T]/T^m$ with the inverse limit topology.
Let $a\in F$. Say a log connection $(M, \nabla_M)$  is \emph{$a$-nilpotent} if
  \begin{equation}\label{Equ-asmall}
       \lim_{n\to+\infty}a^n\prod_{i=0}^{n-1}(\nabla_M-i) = 0
   \end{equation}
   with respect to the induced topology on $M$.
\end{enumerate}
\end{defn}

Although log connection is an ``intrinsic" notion (i.e., independent of choices of the uniformizer $T$ or the basis $d\log T$), in explicit computations it is more convenient to use the following equivalent definition.

\begin{defn}\label{defn_log_T_conn}
\begin{enumerate}
\item A \emph{log-$T$-connection}  over $F[[T]]$  is a finitely generated $F[[T]]$-module $M$ together a $F$-linear map
$\nabla_{M, T}: M \to M$ satisfying \emph{$T$-Leibniz law}; namely,
$$\nabla_{M, T}(fx) =f\nabla_{M, T}(x)+T\frac{d}{dT}(f)x, \quad \forall f \in F[[T]], T \in M.$$
Write $\mathrm{MIC}^\fg_T(F[[T]])$ be the category of such objects.
Let $\mathrm{MIC}_T(F[[T]])$ (resp. $\mathrm{MIC}_T(F[[T]]/T^m)$ where $m \geq 1$)  be the subcategory of objects which are finite free over $F[[T]]$  (resp. $F[[T]]/T^m$).


\item Suppose $F$ is a topological ring as in Def. \ref{deflogconn}(3), then we can similarly define $a$-nilpotent log-$T$-connections.
 Use $\mathrm{MIC}_T^{\fg, a}(F[[T]])$ resp.
    $\mathrm{MIC}_T^a(F[[T]])$ resp. $\mathrm{MIC}_T^a(F[[T]]/T^m)$ to denote the subcategory of $a$-nilpotent connections which are finitely generated resp. finite free over $F[[T]]$ resp. finite free over $F[[T]]/T^m$.
\end{enumerate}
\end{defn}

\begin{remark}
Given $(M, \nabla_M) \in  \mic^\fg_\log(F[[T]])$, we can construct $(M, \nabla_{M, T}) \in  \mic^\fg_T(F[[T]])$ simply by defining
\[ \nabla_{M, T} :=\frac{\nabla_M}{d\log T}.  \]
This induces an equivalence of categories.
\end{remark}

\begin{convention} \label{conv_convenient}
Henceforth in this paper, we will work with the more convenient log-$T$-connections. By abuse of notations, we often will simply use $(M, \nabla_M)$ (omitting $T$ in the subscript) to denote such an object. Although occasionally, we will add in the uniformizer ``$T$" to stress its role.
\end{convention}

\begin{lemma} \label{lemchangeconn}
Suppose $F$ is a field.
Let $y=y(T) \in F[[T]]$ be another uniformizer, namely we have $F[[T]]=F[[y]]$.
For each $1\leq m \leq +\infty$, there is an equivalence of categories
$$\mathrm{MIC}_T(F[[T]]/T^m) \simeq \mathrm{MIC}_y(F[[y]]/y^m),$$
 which sends $(M, \nabla_{M, T})$ to $(M, \nabla_{M, y})$
where $\nabla_{M, y}$ on $M$ is defined via
$$\nabla_{M, y}= \frac{1}{y'}\cdot\frac{y}{T} \nabla_{M, T}$$
where $y'=\frac{dy}{dT}$.
If $F$ is a topological field and $a \in F$, then we also have
$$\mathrm{MIC}_T^a(F[[T]]/T^m) \simeq \mathrm{MIC}_y^a(F[[y]]/y^m).$$
\end{lemma}
\begin{proof}
This is easy.
\end{proof}

\subsection{Stratifications and log connections}

\begin{construction} \label{Constr_strat_to_conn}
Let $S^{\bullet}$ be the cosimplicial ring introduced in Notation \ref{Notation:comsimplicialring}.
Let $1\leq m\leq +\infty$.
Let   $(M,\varepsilon) \in \mathrm{Strat}(S^\bullet/T^mS^{\bullet})$ be a stratification.
Using the expression  (\ref{Equ-Epsilon}), we have  a $R$-linear map
$\phi_1: M\to M$. In addition, the formula \eqref{Equ-strat iff} implies that if we define
\[ \nabla_M:= \frac{\phi_1}{a}: M\to M, \]
then this is a log-$T$-connection.
\end{construction}

\begin{thm} \label{thm_strat_conn_axiom}
  The construction above, which maps $(M,\varepsilon) $ to $(M,\nabla_M)$,   induces an equivalence of categories
\[ \mathrm{Strat}(S^\bullet/T^mS^{\bullet}) \xrightarrow{\simeq} \mathrm{MIC}_{T}^a (R[1/p][[T]]/T^m).  \]
The equivalence preserves tensor products and dualities.
\end{thm}
\begin{proof}
  Let $(M,\varepsilon)$ be a stratification satsifying the cocycle condition with the induced log connection $(M,\nabla_M)$. Then we have $\phi_1 = a\nabla_M$. So by Proposition \ref{prop: strat iff}, we see that
 \[\phi_n = \prod_{i=0}^{n-1}(\phi_1-ia) = a^n\prod_{i=0}^{n-1}(\nabla_M-a)\to 0\]
 as $n\to +\infty$. So $(M,\nabla_M)$ is indeed $a$-nilpotent.

 Conversely, for a given $a$-nilpotent log-$T$-connection $(M,\nabla_M) \in \mathrm{MIC}_T^a (R[1/p][[T]]/T^m) $, we define $\phi_0=\mathrm{id}_M$, $\phi_1 = a\nabla_M$ and $\phi_n: = \prod_{i=0}^{n-1}(\phi_1-ia)$ for any $n\geq 2$. Then conditions (1)-(3) in Proposition \ref{prop: strat iff} are satisfied for $\phi_n$'s. In particular, (\ref{Equ-strat iffII}) allows us to define a stratification $\varepsilon$ on $M$ satisfying the cocycle condition. So we get the inverse functor.
 The compatibility of the equivalence with respect to tensor products and dualities is standard, and the proof is omitted.
\end{proof}

 Note Construction \ref{Constr_strat_to_conn} relies on  expression  (\ref{Equ-Epsilon}), and hence indeed relies on the construction (and choice) of $X_1$. (Indeed, say if we use $(T-T_1)/(aT)$ as ``$X_1$", then $\phi_1$ will become ``$-\phi_1$".) In the following, we give an \emph{intrinsic} construction of this log connection: i.e., an object in the \emph{intrinsic} Definition \ref{deflogconn}. (This construction serves as illustration purpose only, and will not be used in the sequel).

 \begin{lem}\label{Lem-intrinsic}
 Let $\calK$ the kernel of the degeneracy morphism $\sigma_0: S^1/T^mS^1\to S^0/T^mS^0$, which is a closed ideal of $S^1/T^mS^1$. For any $n\geq 1$, denote by $\calK^{[n]}$ the $n$-th complete pd-power of $\calK$.
\begin{enumerate}
\item  The ideal $\calK$ is the closed pd-ideal generated by $X_1 = \frac{T_1-T}{aT}$; that is, $\calK$ is the closed ideal generated by $\{X_1^{[n]}\mid n\geq 1\}.$
\item We have an isomorphism of $S^0/T^mS^0$-modules $\calK/\calK^{[2]} \cong (S^0/T^mS^0)\cdot X_1$.
\end{enumerate}
 \end{lem}
 \begin{proof}
   It suffices to show that (1) is true and then (2) follows as an easy consequence.

   We first assume $m = 1$. In this case, $S^1/TS^1 = R\{X^1\}_{\pd}^{\wedge}[\frac{1}{p}]$ and $\sigma_0$ is induced by sending $X_1$ to $0$. So the result is true in this case.

   For $m = +\infty$, we denote by $\calJ$ the closed ideal of $S^1$ generated by $\{X_1^{[n]}\}_{n\geq 1}$. Then $\calJ\subset \calK$. By noting that $S^0$ is $T$-torsionfree, we get
   \[\calK/T\calK = \Ker(S^1/TS^1\xrightarrow{\sigma_0}S^0/TS^0).\]
   So it follows from what we have proved that $\calK\subset T\calK+\calJ$. Since $S^1$ is $T$-complete and separated, we get $\calK = \calJ$ as desired.

   For $m<+\infty$, using $T$-torsionfreeness of $S^{\bullet}$, we obtain that \[\calK/T^m\calK = \Ker(S^1/T^mS^1\xrightarrow{\sigma_0}S^0/T^mS^0)\]
   as desired.
 \end{proof}

\begin{construction}\label{constructLogconnection}
Let   $(M,\varepsilon) \in \mathrm{Strat}(S^\bullet/T^mS^{\bullet})$ be a stratification. Define two morphisms
 \[\iota_1,\iota_2:M\to M\otimes_{S^0,p_1}S^1\]
 by $\iota_1(m) = \varepsilon(m\otimes_{S^0,p_0}1)$ and $\iota_2(m) = m\otimes_{S^0,p_1}1$. Since $\sigma^*_0(\varepsilon) = \id_M$, we see that $\iota_2-\iota_1: M\to M\otimes_{S^0,p_1}S^1$ takes values in $M\otimes_{S^0,p_1}\calK$. Therefore, we can define a map
\begin{equation}\label{eqintrinsic_conn}
 M\xrightarrow{\iota_1-\iota_2} M\otimes_{S^0,p_1}\calK \to M\otimes_{S^0,p_1}\calK/\calK^{[2]}
\end{equation}
 Recall we have face maps $p_0, p_1: S^0 \to S^1$, and $(p_1-p_0)(S^0) \subset \mathcal K$, and hence we can define
\begin{equation}
d :S^0 \xrightarrow{p_1-p_0} \mathcal K \to \calK/\calK^{[2]};
\end{equation}
this map satisfies Leibniz rule, and  is indeed the \emph{canonical differential map}.
To be more precise, if we identify the rank-1 module $\calK/\calK^{[2]}$ with $\Omega_{\log}$ by sending $X_1$ to $\frac{d\log T}{a}$ (where for brevity, we use $\Omega_{\log}$ to denote  $\Omega_{\log, R[1/p][[T]]/R[1/p]}$), then the above map is precisely the differential map
\[d: S^0 \to  \Omega_{\log}.\]
We now claim \eqref{eqintrinsic_conn} is a \emph{log connection} (as in Def. \ref{deflogconn}).
To check this \emph{intrinsic statement}, it does not hurt to use identification between  $\calK/\calK^{[2]}$ and $\Omega_{\log }$ above. Then everything is concrete, and \eqref{eqintrinsic_conn} becomes
\begin{equation}
m \mapsto \frac{1}{a}\phi_1(m)\otimes d\log  T
\end{equation}
which, by  formula \eqref{Equ-strat iff}, is a \emph{log connection}.
 \end{construction}

\subsection{A technical lemma on recurrence relation}

 \begin{lem}\label{Key lemma}
   Let $R$ be a $\bQ$-algebra and let $a\in R$.
   Let $M$ be a finite free $R[[T]]/(T^m)$-module for some $1\leq m\leq +\infty$, and let $\{\phi_n\}_{n\geq 0}$ be a collection of $R$-linear endomorphisms of $M$ with $\phi_0=\id_M$. Then the following are equivalent:
\begin{enumerate}
\item  For any $x\in M$ and any $n\geq 0$, the equation
   \begin{equation}\label{Equ-KeyI}
       \phi_n(x) = \sum_{l,m\geq 0}\phi_l(\phi_{m+n}(x))(1+aX)^{-m-n}(-1)^m\binom{l+m}{l}X^{[l+m]}
   \end{equation}
   holds true in $M[[X]]$.

     \item For any $x\in M$ and any $n\geq 0$, $\phi_n(x)= \prod_{i=0}^{n-1}(\phi-ia)(x)$ with $\phi = \phi_1$. Here, we require that $\prod_{i=0}^{n-1}(\phi-ia) = \id_M$ for $n = 0$.
\end{enumerate}
 \end{lem}
 \begin{proof}
   Applying $\partial_{X}$ to the right hand side of (\ref{Equ-KeyI}), we get
   \begin{equation}\label{Equ-KeyII}
       \begin{split}
           &\partial_X(\sum_{l,m\geq 0}\phi_l(\phi_{m+n}(x))(1+aX)^{-m-n}(-1)^m\binom{l+m}{l}X^{[l+m]})\\
           =&-\sum_{l,m\geq 0}(m+n)a\phi_l(\phi_{m+n}(x))(1+aX)^{-m-n-1}(-1)^m\binom{l+m}{l}X^{[l+m]}\\
           &+\sum_{l,m\geq 0}\phi_l(\phi_{m+n}(x))(1-aX)^{-m-n}(-1)^m\binom{l+m}{l}X^{[l+m-1]}.
       \end{split}
   \end{equation}

   Now assume $(1)$ is true. Then we get
   \begin{equation*}
       \begin{split}
           0=&\partial_X(\sum_{l,m\geq 0}\phi_l(\phi_{m+n}(x))(1+aX)^{-m-n}(-1)^m\binom{l+m}{l}X^{[l+m]})\\
           =&-\sum_{l,m\geq 0}(m+n)a\phi_l(\phi_{m+n}(x))(1-aX)^{-m-n-1}(-1)^m\binom{l+m}{l}X^{[l+m]}\\
           &+\sum_{l,m\geq 0}\phi_l(\phi_{m+n}(x))(1+aX)^{-m-n}(-1)^m\binom{l+m}{l}X^{[l+m-1]}.
       \end{split}
   \end{equation*}
   By letting $X=0$ in the right hand side of above formula, we get
   \begin{equation*}
       -na\phi_n(x) - \phi_{n+1}(x) + \phi_1(\phi_n(x)) = 0.
   \end{equation*}
   In other words, for any $n\geq 0$, we have
   \begin{equation}\label{Equ-KeyIteration}
        \phi_{n+1}(x) =(\phi_1 - na)(\phi_n(x)).
   \end{equation}
   So Item (2) follows from iterations.

   Conversely, we assume (2) is true. Then the formula (\ref{Equ-KeyIteration}) holds in this case and therefore all $\phi_n$'s are commutative. It is easy to see that (\ref{Equ-KeyI}) is true after letting $X = 0$ on the both sides. So to conclude (1), it is enough to show
   \[\partial_X(\sum_{l,m\geq 0}\phi_l(\phi_{m+n}(x))(1+aX)^{-m-n}(-1)^m\binom{l+m}{l}X^{[l+m]}) = 0.\]
   By (\ref{Equ-KeyII}), it suffices to show
   \begin{equation*}
       \begin{split}
           0 = &-\sum_{l,m\geq 0}(m+n)a\phi_l(\phi_{m+n}(x))(1+aX)^{-m-n-1}(-1)^m\binom{l+m}{l}X^{[l+m]}\\
           &+\sum_{l,m\geq 0}\phi_l(\phi_{m+n}(x))(1+aX)^{-m-n}(-1)^m\binom{l+m}{l}X^{[l+m-1]}.
       \end{split}
   \end{equation*}
   The formula (\ref{Equ-KeyIteration}) allows us to replace $-(m+n)a\phi_{m+n}(x)$ by $\phi_{m+n+1}(x) - \phi_1(\phi_{m+n}(x))$ in above formula. Using commutativity of $\phi_n$'s, it suffices to show that
   \begin{equation*}
       \begin{split}
           0 = &\sum_{l,m\geq 0}\phi_l(\phi_{m+n+1}(x))(1+aX)^{-m-n-1}(-1)^m\binom{l+m}{l}X^{[l+m]}\\
           &-\sum_{l,m\geq 0}\phi_1(\phi_l(\phi_{m+n}(x)))(1+aX)^{-m-n-1}(-1)^m\binom{l+m}{l}X^{[l+m]}\\
           &+\sum_{l,m\geq 0}\phi_l(\phi_{m+n}(x))(1+aX)^{-m-n}(-1)^m\binom{l+m}{l}X^{[l+m-1]}.
       \end{split}
   \end{equation*}
   However, we have that
   \begin{equation*}
       \begin{split}
           &\sum_{l,m\geq 0}\phi_l(\phi_{m+n+1}(x))(1+aX)^{-m-n-1}(-1)^m\binom{l+m}{l}X^{[l+m]}\\
           &-\sum_{l,m\geq 0}\phi_1(\phi_l(\phi_{m+n}(x)))(1+aX)^{-m-n-1}(-1)^m\binom{l+m}{l}X^{[l+m]}\\
           &+\sum_{l,m\geq 0}\phi_l(\phi_{m+n}(x))(1+aX)^{-m-n}(-1)^m\binom{l+m}{l}X^{[l+m-1]}\\
          =&\sum_{l\geq 0,m\geq 1}-\phi_l(\phi_{m+n}(x))(1+aX)^{-m-n}(-1)^m\binom{l+m-1}{l}X^{[l+m-1]}\\
           &-\sum_{l,m\geq 0}\phi_1(\phi_l(\phi_{m+n}(x)))(1+aX)^{-m-n-1}(-1)^m\binom{l+m}{l}X^{[l+m]}\\
           &+\sum_{l,m\geq 0}\phi_l(\phi_{m+n}(x))(1+aX)^{-m-n}(-1)^m\binom{l+m}{l}X^{[l+m-1]}\\
           =&\sum_{l,m\geq 0}\phi_l(\phi_{m+n}(x))(1+aX)^{-m-n}(-1)^m\binom{l+m-1}{l-1}X^{[l+m-1]}\\
           &-\sum_{l,m\geq 0}\phi_1(\phi_l(\phi_{m+n}(x)))(1+aX)^{-m-n-1}(-1)^m\binom{l+m}{l}X^{[l+m]}.
       \end{split}
   \end{equation*}
   Here, we use $\binom{l-1}{l} = 0$ and $\binom{l+m}{l} = \binom{l+m-1}{l}+\binom{l+m-1}{l-1}$. Using (\ref{Equ-KeyIteration}) again, we have
   \begin{equation*}
       \begin{split}
           &\sum_{l,m\geq 0}\phi_l(\phi_{m+n}(x))(1+aX)^{-m-n}(-1)^m\binom{l+m-1}{l-1}X^{[l+m-1]}\\
           &-\sum_{l,m\geq 0}\phi_1(\phi_l(\phi_{m+n}(x)))(1+aX)^{-m-n-1}(-1)^m\binom{l+m}{l}X^{[l+m]}\\
           =&\sum_{l,m\geq 0}\phi_l(\phi_{m+n}(x))(1+aX)^{-m-n}(-1)^m\binom{l+m-1}{l-1}X^{[l+m-1]}\\
           &-\sum_{l,m\geq 0}\phi_{l+1}(\phi_{m+n}(x))(1+aX)^{-m-n-1}(-1)^m\binom{l+m}{l}X^{[l+m]}\\
           &-\sum_{l,m\geq 0}la\phi_{l}(\phi_{m+n}(x))(1+aX)^{-m-n-1}(-1)^m\binom{l+m}{l}X^{[l+m]}\\
           =&\sum_{l,m\geq 0}\phi_l(\phi_{m+n}(x))(1+aX)^{-m-n}(-1)^m\binom{l+m-1}{l-1}X^{[l+m-1]}\\
           &-\sum_{l,m\geq 0}\phi_{l+1}(\phi_{m+n}(x))(1+aX)^{-m-n-1}(-1)^m\binom{l+m}{l}X^{[l+m]}\\
           &-aX(1-aX)^{-1}\sum_{l,m\geq 0}\phi_{l}(\phi_{m+n}(x))(1+aX)^{-m-n}(-1)^m\binom{l+m-1}{l-1}X^{[l+m-1]}.
       \end{split}
   \end{equation*}
   Here, we use $la\binom{l+m}{1}X^{[l+m]} = aX\binom{l+m-1}{l-1}X^{[l+m-1]}$. Finally, we get
   \begin{equation*}
       \begin{split}
           &\partial_X(\sum_{l,m\geq 0}\phi_l(\phi_{m+n}(x))(1+aX)^{-m-n}(-1)^m\binom{l+m}{l}X^{[l+m]})\\
           =&\sum_{l,m\geq 0}\phi_l(\phi_{m+n}(x))(1+aX)^{-m-n}(-1)^m\binom{l+m-1}{l-1}X^{[l+m-1]}\\
           &-\sum_{l,m\geq 0}\phi_{l+1}(\phi_{m+n}(x))(1+aX)^{-m-n-1}(-1)^m\binom{l+m}{l}X^{[l+m]}\\
           &-aX(1-aX)^{-1}\sum_{l,m\geq 0}\phi_{l}(\phi_{m+n}(x))(1+aX)^{-m-n}(-1)^m\binom{l+m-1}{l-1}X^{[l+m-1]}\\
           =&\sum_{l,m\geq 0}\phi_l(\phi_{m+n}(x))(1+aX)^{-m-n-1}(-1)^m\binom{l+m-1}{l-1}X^{[l+m-1]}\\
           &-\sum_{l,m\geq 0}\phi_{l+1}(\phi_{m+n}(x))(1+aX)^{-m-n-1}(-1)^m\binom{l+m}{l}X^{[l+m]}\\
           =&0,
       \end{split}
   \end{equation*}
   which is exactly what we want. This shows that (2) implies (1).
 \end{proof}

\section{$\bbdrplus$-crystals as log connections}\label{secdRlogconn}
In this section, we relate $\bbdrplus$-crystals with log connections.


Let $A(u)$ be a uniformizer of $\frakS_{\dR}^+$.
Recall  we have $\frakS^+_{\dR}=K[[A(u)]]$, and hence we can define categories of log-$A(u)$-connections following Def. \ref{defn_log_T_conn}.

\begin{theorem}\label{Thm-dRasLogconnection}
We have a commutative diagram of functors:
\begin{equation*}
\begin{tikzcd}
{\Vect((\calO_K)_{\Prism},\bB_{\dR,m}^+)} \arrow[d, "\simeq"] \arrow[rr, hook] &  & {\Vect((\calO_K)_{\Prism,\log},\bB_{\dR,m}^+)} \arrow[d, "\simeq"] \\
{\MIC_{ A(u)}^{-E'(\pi)}(\frakS_{\dR,m}^+)} \arrow[rr, hook]                     &  & {\MIC_{ A(u)}^{-\pi E'(\pi)}(\frakS_{\dR,m}^+)}
\end{tikzcd}
\end{equation*}
Here:
\begin{enumerate}
\item the two vertical functors are equivalences and induced by evaluations at the Breuil--Kisin prism $(\frakS,(E))$ resp. the  Breuil--Kisin log prism $(\frakS,(E),M_{\frakS})$; all these equivalences are   bi-exact.
\item the top row functor is induced by the forgetful morphism of sites $(\calO_K)_{\Prism,\log} \to (\calO_K)_{\Prism}$ forgetting log structures;
\item the bottom functor is the obvious inclusion. 
\end{enumerate}  All functors above preserve  tensor products and dualities.
\end{theorem}
 \begin{proof} Apply Thm. \ref{thm_strat_conn_axiom} to the set-up in Prop. \ref{Prop-structure} (cf. Example \ref{example_verify_cosimpring}).
 Bi-exactness of these equivalences follow from similar argument as \cite[Prop. 2.6]{GMWHT}.
 \end{proof}


 \begin{rmk} \label{rembs2372}
 Thm. \ref{Thm-dRasLogconnection} gives   conceptual explanations of results in \cite[\S 7.3]{BS23}. Indeed, \cite[Cor. 7.17]{BS23} can be refined as the following composite:
\[ \vect(\okpris, \opris\langle \mathcal{I}_\prism/p \rangle[1/p]) \to \vect(\okpris, \bbdrplus) \simeq \mic_{E}^{-E'(\pi)}(\gs^+_{\dR}) \into \mic_E(\gs_\dR^+) \]
Here the first category is defined in \cite{BS23}, and the first arrow is induced by the obvious morphism of structure sheaves; the second arrow follows from Thm. \ref{Thm-dRasLogconnection}; the last category is exactly ``$\vect^{\nabla, \log}(\mathcal{S})$"   in \cite[Constr. 7.16]{BS23}.
\end{rmk}

In the following, we discuss standard properties of the equivalences in Thm. \ref{Thm-dRasLogconnection}.
In the following,
let $\ast \in \{ \emptyset, \log \}$,
let $a=-E'(\pi)$ if $\ast=\emptyset$, and $a=-\pi E'(\pi)$ if $\ast=\log$.
 Let $A(u)$ be a uniformizer of $\frakS_{\dR}^+$ and $1\leq m\leq \infty$.

\begin{remark}[Compatibility with reductions]
 Let $1\leq m<n\leq \infty$.
We have a commutative diamgram,
\begin{equation*}
\begin{tikzcd}
{\Vect(\okprisast,\bB_{\dR,n}^+)} \arrow[d, "\simeq"] \arrow[rr] &  & {\Vect(\okprisast,\bB_{\dR,m}^+)} \arrow[d, "\simeq"] \\
{\MIC_{ A(u)}^a(\frakS_{\dR,n}^+)} \arrow[rr]                      &  & {\MIC_{ A(u)}^a(\frakS_{\dR,m}^+)}
\end{tikzcd}
\end{equation*}
where both rows are the  obvious reduction morphisms.
\end{remark}

 \begin{rmk}[Compatibility with Hodge--Tate case]
When $m=1$, a $\mathbb{B}_{\dR, 1}^+$-crystal is precisely  a (rational) Hodge--Tate crystal treated in \cite{GMWHT}.
A log-$A(u)$-connection is simply a $K$-vector space  with a $K$-linear endomorphism.
In this degenerate case,  an $a$-nilpotent log-$A(u)$-connection $(M,\nabla_M)$ consists of a finite dimensional $K$-space $M$ together with a $K$-linear endomorphism $\nabla_M$ of $M$ such that
   \[\lim_{n\to+\infty}a^n\prod_{i=0}^{n-1}(\nabla_M-i) = 0.\]
   The equivalences in    Thm. \ref{Thm-dRasLogconnection} then recover the known results about rational Hodge--Tate crystals  proved in  \cite{GMWHT}.
 \end{rmk}


 \begin{exam}[Breuil--Kisin twists]
 \label{Exam-IdealSheaf}
   For any $n\in \bZ$, $\calI_{\Prism}^n\otimes_{\calO_{\Prism}}\BBdRp$ is a $\bbdrplus$-crystal on $(\calO_K)_{\Prism,*}$.
Its evaluation  on the Breuil--Kisin (log-) prism, denoted by $M_n$, is $\frakS_{\dR}^+\cdot A(u_0)^n$, i.e.,  the ideal of  $\frakS_{\dR}^+$ generated by $A(u_0)^n$.
   Using (\ref{Equ-Face-II}), the stratification $(\frakS_{\dR}^+\cdot A(u_0)^n,\varepsilon)$ is determined by
   \[\varepsilon(A(u_0)^n) = (1+aX_1)^nA(u_0)^n\equiv (1+naX_1)A(u_0)^n \mod X_1^{[\geq 2]}.\]
   Let $\nabla_{M_n}$ be the log-$A(u)$-connection on $\frakS_{\dR}^+\cdot A(u_0)^n$. Then we see that
   \[\nabla_{M_n} = n\cdot \id_{M_1}+A(u)\frac{d}{dA(u)}.\]
   More generally, for   $\bM \in \Vect(\okprisast,\bB_{\dR}^+)$ with  associated log-$A(u)$-connection $(M,\nabla_M)$ and for any $n\in\bZ$, the log-$A(u)$-connection associated to $\calI_{\Prism}^n\otimes_{\calO_{\Prism}}\bM$ is given by $(M,n\cdot\id_M+\nabla_M)$. By noting that for any $x\in M$ and $n\geq 0$,
   \[\nabla_M(A(u)^nx) = A(u)^n\nabla_M(x)+A(u)\frac{d}{dA(u)}(A(u)^n)x = A(u)^n(\nabla_M(x)+nx),\]
   we see that the associated log-$A(u)$-connection $\nabla_{M_n}$ of $\calI_{\Prism}^n\otimes_{\calO_{\Prism}}\bM$ is the restriction of that of $\bM$ via identifying $\calI_{\Prism}^n\otimes_{\calO_{\Prism}}\bM$ with $\calI^n_{\Prism}\bM$.
 \end{exam}


 \begin{exam}[Short exact sequences]\label{Exam-Reduction}
   For  $\bM \in \Vect(\okprisast,\bB_{\dR, n}^+)$, let $\calI_{\Prism}^m\bM\subset\bM$ denotes the image of the natural morphism
   \[\calI_{\Prism}^m\otimes_{\calO_{\Prism}}\bM\to\bM\]
   induced by the natural inclusion $\calI_{\Prism}^m\subset \calO_{\Prism}$. It is easy to see that $\calI_{\Prism}^m\bM \in \Vect(\okprisast,\bB_{\dR, n-m}^+)$ and is isomorphic to $\calI_{\Prism}^m\otimes_{\calO_{\Prism}}\overline \bM_{n-m}$, where for any $1\leq i\leq n$, $\overline \bM_i$ denotes the reduction of $\bM$ modulo $\calI^i_{\Prism}$.
   Thus we get an exact sequence of sheaves
   \begin{equation}\label{Equ-Reduction}
       \xymatrix@C=0.45cm{
         0\ar[r]& \calI_\prism^m\bM\ar[r]& \bM\ar[r]&\overline \bM_m\ar[r]& 0.
       }
   \end{equation}
  Let $(M,\nabla_M)$ be the log-$A(u)$-connection associated to $\bM$. Then by Example \ref{Exam-IdealSheaf}, via equivalences in Theorem \ref{Thm-dRasLogconnection}, we see that the isomorphism \[\calI_{\Prism}^m\bM\cong \calI^{m}_{\Prism}\otimes_{\calO_{\Prism}}\bM_{n-m}\]
   identifies the log-$A(u)$-connection associated to $\calI^m_{\Prism}\bM$ with $( M/A(u)^{n-m}M,m\cdot\id_{M_{n-m}}+\nabla_M)$ and that the exact sequence (\ref{Equ-Reduction}) induces
   an exact sequence
   \[\xymatrix@C=0.45cm{
     0\ar[r]& M/A(u)^{n-m}M \ar[rr]^{\qquad\times A(u)^m} && M \ar[r] & M/A(u)^m \ar[r] &0
   }\]
   of log-$A(u)$-connections.
 \end{exam}

\section{Cohomology of  crystals I: vs. Sen--Fontaine cohomology} \label{sec:compa dR coho}
  In this section, we   compare  (log-) prismatic cohomology of $\bbdrplus$-crystals, first with \v Cech-Alexander cohomology in Prop. \ref{propCechcompa}, then with Sen--Fontaine cohomology  of the  the corresponding log connections  in Theorem \ref{Thm-dRCohomology}.
The Sen--Fontaine cohomology will be further compared with Galois cohomology in \S \ref{sec_compa_Galois}.

Throughout this section, let  $1\leq m\leq \infty$.
Let $\ast \in \{ \emptyset, \log \}$.
Let $a=-E'(\pi)$ if $\ast=\emptyset$, and $a=-\pi E'(\pi)$ if $\ast=\log$.
Let $A(u)$ be a uniformizer of $\frakS_{\dR}^+$.

 \subsection{Comparison with \v Cech-Alexander cohomology}
  \begin{lem}\label{Lem-Cech-to-derived}
    Let $\bm \in \Vect(\okprisast, \bbdrplusm)$.
    For any $i\geq 1$, and for  any object $\frakA = (A,I)\in(\calO_K)_{\Prism}$ when $\ast=\emptyset$ (resp.  any object $\frakA = (A,I,M)\in \okprislog$ when $\ast=\log$)
    \[\rH^i(\frakA,\bM) = 0.\]
 \end{lem}
  \begin{proof}
By standard d\'evissage (using short exact sequences in   Example \ref{Exam-Reduction}), it suffices to treat the  $m=1$ case. Then the result follows from the similar arguments in the proof of \cite[Lem. 3.12]{Tia23}.
 \end{proof}



  \begin{prop}\label{propCechcompa}
   Let $\bm \in \Vect(\okprisast, \bbdrplusm)$. Then there is a quasi-isomorphism
   \[\RGamma(\okprisast,\bM)\cong \Tot(\bM(\frakS_\ast^{\bullet}))\]
  \end{prop}
\begin{proof}
 This follows from   Lemma \ref{Lem-Cech-to-derived} together with \v Cech-to-derived spectral sequence.
\end{proof}

  \begin{cor}\label{Cor-CohoDim}
  Let $\bm \in \Vect(\okprisast, \bbdrplusm)$,
  then the   cohomology
  \[ \RGamma(\okprisast,\bM) \]
   is concentrated in degree $[0,1]$.
  \end{cor}
  \begin{proof}
    By standard d\'evissage, we are reduced to the case for $m=1$, and then \cite[Thm. 4.5]{GMWHT} applies.
     \end{proof}

\subsection{Comparison with de Sen--Fontaine cohomology}
\begin{notation} \label{notaSFcoho}
    Let $\bM \in \Vect((\calO_K)_{\Prism,\ast},\bB_{\dR,m}^+)$.
Let $(M,\varepsilon)$ and $(M,\nabla_M)$ be the stratification  and log-$A(u)$-connection associated to $\bM$, constructed using a choice of $\pi$ and hence the Breuil--Kisin prism $(\gs, (E), \ast).$
We call the complex
\[ [M \xrightarrow{a\nabla_M} M] \]
the Sen--Fontaine complex associated to $\bM$ (with respect to $\pi$); this terminology comes from the close relation between $\nabla_M$ and Sen--Fontaine theory, cf. \S \ref{sec_compa_Galois}.
Note the ``$a$" in the $a\nabla_M$ operator makes our theory compatible with that in Hodge--Tate crystals.
\end{notation}

  \begin{construction}\label{construction:rho}
   Use Notation \ref{notaSFcoho}, we now relate Sen--Fontaine complex with (log-) prismatic cohomology.   
  We denote the complex induced from $\Tot(\bM(\frakS_{\ast}^{\bullet}))$  by
  \[M^{\bullet}:=[M^0\xrightarrow{d^0}M^1\xrightarrow{d^1}M^2\to\cdots].\]
  For any $n\geq 0$, we denote by $q_0 $ the structure morphism induced by the map $\{0\}\to\{0,\dots,n\}$ with the target $0$.
  Using $q_0$, we obtain a canonical isomorphism
  \[M^n:=\bM(\frakS_{\ast}^n,(E))\cong M^0\otimes_{\frakS_{\dR,m}^+,q_0}\frakS_{\ast,\dR,m}^{+, n}.\]
    Using Proposition \ref{prop: strat iff}, we get that for any $x\in M^0 = M$
  \begin{equation}\label{Equ-d0}
          d^0(x)=\varepsilon(x)-x = \sum_{n\geq 1}a^n \left( \prod_{i=0}^{n-1}(\nabla_M-i)(x)\right) X_1^{[n]} = H(\nabla_M,X_1)(a\nabla_M(x)),
  \end{equation}
  where
  \begin{equation}\label{Equ-H}
      H(Y,X) = \frac{(1+aX)^{Y}-1}{X}:=\sum_{n\geq 1}a^{n-1}\left(\prod_{i=1}^{n-1}(Y-i)\right) X^{[n]}
  \end{equation}
  and then $H(\nabla_M,X_1)$ is well-defined by $a$-nilpotency of $(M,\nabla_M)$. In other words,
  \[d^0 = H(\nabla_M,X_1)\circ(a\nabla_M).\]
  Therefore, the following diagram
\begin{equation}\label{Diag-MapofComplex}
\begin{tikzcd}
M \arrow[rr, "a\nabla_M"] \arrow[dd, "="] &  & M \arrow[dd, "{H(\nabla_M,X_1)}"] &        \\
                                           &  &                                   &        \\
M_0 \arrow[rr, "d^0"]                      &  & M^1 \arrow[r, "d^1"]              & \cdots
\end{tikzcd}
\end{equation}
  commutes and induces a morphism of complexes
  \begin{equation}\label{Equ-rho}
      \rho_{\bM}:[M\xrightarrow{a\nabla_M}M]\to \Tot(\bM(\frakS_{*}^{\bullet})).
  \end{equation}
  By construction, the morphism $\rho_{\bM}$ is functorial in $\bM$. 
    \end{construction}


  \begin{prop}\label{Prop-rho}
    The morphism $\rho_{\bM}$ above is a $K$-linear quasi-isomorphism.
  \end{prop}
  \begin{proof}
    Let $C(\rho_{\bM})$ be the cone of $\rho_{\bM}$. It suffices to show $C(\rho_{\bM})\simeq 0$.
    By standard d\'evissage (using short exact sequences in   Example \ref{Exam-Reduction}), it suffices to treat the  $m=1$ case. Compare \cite[Const. 3.8]{GMWHT} with Construction \ref{construction:rho}, and then one can conclude by using \cite[Cor. 3.10]{GMWHT}.
 \end{proof}

    \begin{thm}\label{Thm-dRCohomology}
     Use Notation \ref{notaSFcoho}. There exists a $K$-linear quasi-isomorphism
    \[ [ M \xrightarrow{a\nabla_M} M]  \simeq \RGamma(\okprisast,\bM),  \]
    which is functorial in $\bM$.
  \end{thm}
  \begin{proof}
Combine  Propositions \ref{propCechcompa} and \ref{Prop-rho}.
  \end{proof}

  \begin{cor}\label{Cor-UsualvsLog}
  Let $\bM\in \Vect((\calO_K)_{\Prism},\bB_{\dR,m}^+)$, which can be regarded as an object in $\Vect(\okprislog,\bB_{\dR,m}^+)$ (cf. Theorem \ref{Thm-dRasLogconnection}). Then there exists a quasi-isomorphism
     \[\RGamma((\calO_K)_{\Prism},\bM) \simeq\RGamma((\calO_K)_{\Prism,\log},\bM) \]
     which is functorial in $\bM$.
  \end{cor}
  \begin{proof}
  Let $(M,\nabla_M)$ be the associated  log-$A(u)$-connection to $\bm$ \emph{as a prismatic crystal}, and hence also to $\bm$ \emph{as a log-prismatic crystal}. Thus by Thm. \ref{Thm-dRCohomology}, we have quasi-isomorphisms
  \[\RGamma((\calO_K)_{\Prism},\bM)\simeq 
  [ M \xrightarrow{-E'(\pi)\nabla_M} M]
 \simeq   [ M \xrightarrow{-\pi E'(\pi)\nabla_M} M]  \simeq\RGamma((\calO_K)_{\Prism,\log},\bM);\]
 where the second quasi-isomorphism follows from the fact that $\pi$ is invertible in $\gs^+_\dR$.
\end{proof}

\section{$\bbdrplus$-crystals and $\BdRp$-representations} \label{sec: perfect crystal and rep}

In this  section, we study relation between $\bbdrplus$-crystals and Galois representations.
We show that the category of $\bbdrplus$-crystals on the \emph{perfect} (log-) prismatic site is equivalent to the category of $\bdrplus$-representations of $G_K$. When the crystal (on the perfect site) comes from a crystal on the full (log-) prismatic site, we use the explicit description of its associated log connection to give an explicit description of the associated $\bdrplus$-representation.

 \subsection{Crystals on the perfect  site}

\begin{notation}\label{notaainflog}
Recall in Notation \ref{notafields}, we defined a compatible sequence of $p^n$-th root of unity $\mu_n$; the sequence $(1, \mu_1, \cdots, \mu_n, \cdots)$ defines an element $\epsilon \in \ocflat$. Let $[\epsilon] \in \ainf$ be its Teichm\"uller lift. Define $\xi := \frac{[\epsilon]-1}{[\epsilon]^{\frac{1}{p}}-1}$, which is a generator of the kernel of $\theta: \ainf \to \oc$. 
Then we have the Fontaine prism $(\ainf, (\xi))$.
 There is a natural morphism of   prisms
  $$\iota: (\frakS,(E(u)) )\xrightarrow{u\mapsto [\pi^{\flat}]}(\Ainf,(\xi) ).$$ 
  Equip $\Ainf$ with the log structure $M_{\Ainf}$ induced by $\bN\xrightarrow{1\mapsto[\pi^{\flat}]}\Ainf$. Then $(\Ainf,(\xi),M_{\Ainf})$ is a log prism in $(\calO_K)_{\Prism,\log}$
  , which is referred as Fontaine log prism.
  There is a natural morphism of log prisms
  $$\iota: (\frakS,(E(u)),M_{\frakS})\xrightarrow{u\mapsto [\pi^{\flat}]}(\Ainf,(\xi),M_{\Ainf}).$$ 
 \end{notation}

\begin{notation}\label{notabdrplus}
Let $\bdrplus$ and $\bdr$ be the usual Fontaine's de Rham period ring. Recall
 $\bdrplus= \projlim_m \ainf[1/p]/(\xi)^m$ as an algebraic ring. Equip $\ainf/(\xi)^m$ with  the induced   $p$-adic topology from $\ainf$, and consider $\ainf[1/p]/(\xi)^m=(\ainf/(\xi)^m)[1/p]$ as a $\qp$-Banach space.
Taking inverse limit, we can regard $\bdrplus$ as a Fr\'echet space.
 For $1\leq m\leq \infty$, denote $\bdrplusm:=\BdRp/\xi^m\BdRp$. Clearly we have
 \[ \bdrplusm  =\bB_{\dR,m}^+(\Ainf,(\xi),M_{\Ainf}) =\bB_{\dR,m}^+(\Ainf,(\xi)) \]
\end{notation}


  \begin{thm}\label{Prop-perfectisproet}
  Let $1\leq m\leq \infty$.
We have a commutative diagram of tensor equivalences of categories
   \begin{equation*}
\begin{tikzcd}
{ \Vect((\calO_K)_{\Prism}^{\perf},\bbdrplusm)} \arrow[rr, "\simeq"] \arrow[d, "\simeq"] &  & {\Vect((\calO_K)_{\Prism,\log}^{\perf},\bB_{\dR,m}^+)} \arrow[d, "\simeq"] \\
{ \Rep_{G_K}(\bfB_{\dR,m}^+)} \arrow[rr, "="]                                             &  & { \Rep_{G_K}(\bfB_{\dR,m}^+)}
\end{tikzcd}
   \end{equation*}
 Here, the top functor is induced by the forgetful functor $(\calO_K)_{\Prism,\log}^{\perf} \to (\calO_K)_{\Prism}^{\perf}$ forgetting log structures; the two vertical functors are induced by evaluations at $(\Ainf,(\xi))$ and $(\Ainf,(\xi),M_{{\Ainf}})$ respectively.
  \end{thm}
  \begin{proof}
    By  \cite[Prop. 2.18]{MW22log}, we even have an equivalence of sites
    \[ (\calO_K)_{\Prism}^{\perf} \simeq (\calO_K)_{\Prism,\log}^{\perf}.\]
    Thus, the top row is automatically an equivalence. It now suffices to prove the left vertical equivalence.

    By \cite[Lem. 5.3]{GMWHT}, $(\Ainf,(\xi))$ is a cover of the final object of $\Sh((\calO_K)_{\Prism}^{\perf})$. Let $(\bfa_{\inf}^{\bullet},(\xi))$ be the corresponding \v Cech nerve and put $\bfB_{\dR,m}^{\bullet,+}:=\bB_{\dR,m}^+(\bfa_{\inf}^{\bullet},(\xi))$.
    By \cite[Prop. 2.7]{BS23}, the category $\Vect((\calO_K)_{\Prism}^{\perf},\bbdrplusm)$ is equivalent to the category of stratifications with respect to $\bfB_{\dR,m}^{\bullet,+}$.
    By the following Lemma \ref{Lem-Galois descent}, these in turn are equivalent to the category of stratifications with respect to the cosimplicial ring $\rC(G_K^{\bullet},\bfB_{\dR,m}^+)$, and hence equivalent to $\Rep_{G_K}(\bfB_{\dR,m}^+)$ by Galois descent.
     \end{proof}

     The following lemma is used in Proposition \ref{Prop-perfectisproet}.

    \begin{lem}\label{Lem-Galois descent}
      For any $1\leq m\leq \infty$, there is a canonical isomorphism of cosimplicial rings
      \[\bfB_{\dR,m}^{\bullet,+}\xrightarrow{\cong}\rC(G_K^{\bullet},\bfB_{\dR,m}^+),\]
      where for a topological ring $A$, $\rC(G_K^{\bullet},A)$ denotes the cosimplicial ring of continuous functions from $G_K^{\bullet}$ to $A$.
    \end{lem}
    \begin{proof}
      This follows from a similar argument for the proof of \cite[Lem. 5.5]{GMWHT} but for the convenience of readers, we provide a proof as follows:
          First, note that for any $n\geq 0$, $\rA_{\inf}^{n}/\xi\bfa_{\inf}^{n}$ is the self copruduct of $(n+1)$-copies of $\calO_C$ in the category of perfectoid $\calO_K$-algebras. So there is a canonical morphism
      \[\bfa_{\inf}^{n}/\xi\bfa_{\inf}^{n}\to C(G_K^n,\calO_C)\]
      and hence a canonical morphism of cosimilicial rings
      \[\bfa_{\inf}^{\bullet}/\xi\bfa_{\inf}^{\bullet}\to C(G_K^{\bullet},\calO_C).\]
      By \cite[Th. 3.10]{BS22}, this amounts to a canonical morphism of cosimplicial prisms
      \[(\bfa_{\inf}^{\bullet},(\xi))\to (C(G_K^{\bullet},\Ainf),(\xi)).\]
      So we get a canonical morphism of cosimplicial rings
      \[\bfB_{\dR,m}^{\bullet,+}\to C(G_K^{\bullet},\bfB_{\dR,m}^+).\]
      To see this is an isomorphism, by d\'evissage, we are reduced to the case for $m = 1$, which follows from \cite[Lem. 5.5]{GMWHT} immediately.
    \end{proof}

  \subsection{Representations from crystals on the full (log-) prismatic site} \label{subsec_fullcrystal_rep}

We set up some notations for discussions in this subsection.

\begin{construction}
  Let $\ast \in \{ \emptyset, \log \}$. Consider the composite functor
 \[\Vect((\calO_K)_{\Prism, \ast},\bB_{\dR,m}^+) \to \Vect((\calO_K)^{\mathrm{perf}}_{\Prism,\ast},\bB_{\dR,m}^+) \to \Rep_{G_K}(\bfB_{\dR,m}^+).\]
Given $\bm \in \Vect((\calO_K)_{\Prism, \ast},\bB_{\dR,m}^+)$, let $\bm^{\mathrm{perf}}$ be its restriction to $ \Vect((\calO_K)^{\mathrm{perf}}_{\Prism,\ast},\bB_{\dR,m}^+)$, and let $W \in  \Rep_{G_K}(\bfB_{\dR,m}^+)$ be the corresponding representation. In this subsection, we use the stratification on $\bm$ to give an explicit description of $G_K$-action on $W$.

 Let us recall some notations.
Recall
\begin{equation*}
a=
\begin{cases}
  -E'(\pi), &  \text{if } \ast=\emptyset \\
 -\pi E'(\pi), &  \text{if } \ast=\log
\end{cases}
\end{equation*}
Recall also the variable (related with the uniformizer $A(u)=u-\pi$)
\[X_1=\frac{u_1-u_0}{a(u_0-\pi)}.\]
(Let us note that, as pointed out in Rem. \ref{rem3unif}, the uniformizer $u-\pi$ shall be convenient for our computations, cf. the explicit formula in Prop. \ref{Prop-MatrixCocycle}.)
\end{construction}

 \begin{construction} \label{const_function_gk}
Recall evaluation of $\bm$ on the Breuil--Kisin prism (resp.  Breuil--Kisin log prism) gives a stratification $(M, \varepsilon)$ and then a log-$(u-\pi)$-connection $(M,\nabla_M)$.
By Proposition \ref{prop: strat iff} and Construction \ref{constructLogconnection}, the stratification
 \[\varepsilon: M\otimes_{\frakS_{\dR,m}^{0,+},p_0} \frakS_{\ast, \dR,m}^{1,+}\to M\otimes_{\frakS_{\dR,m}^{0,+},p_1}\frakS_{\ast, \dR,m}^{1,+}\]
is related with $\nabla_M$ such that for any $x\in M$,
\begin{equation}\label{eqsec7vare}
\varepsilon(x) = \sum_{n\geq 0}\prod_{i=0}^{n-1}a^n(\nabla_M-i)(x)X_1^{[n]}.
\end{equation}
The stratification of  $\bm^{\mathrm{perf}}$ is precisely the base change of $(M, \varepsilon)$ along  $\frakS_{\ast,\dR, m}^{\bullet,+}\to\bfB_{\dR,m}^{\bullet,+}$. To understand this new stratification and hence the $G_K$-action on $W$, by Lemma \ref{Lem-Galois descent}, it suffices to determine the image of $X_1$ under the composite
\[\frakS_{\ast,\dR,m}^{1, +}\to  \bfB_{\dR,m}^{1,+} \simeq   \rC(G_K,\bfB_{\dR,m}^+).\]
 \end{construction}

 \begin{notation}
 Let $\chi:G_K\to \bZ_p^{\times}$ be the $p$-adic cyclotomic character such that for any $g\in G_K$, $g(\epsilon) = \epsilon^{\chi(g)}$. Denote by $c:G_K\to \Zp$ the cocycle such that for any $g\in G_K$, $g(\pi^{\flat}) = \epsilon^{c(g)}\pi^{\flat}$.
   \end{notation}


  \begin{lem}\label{Lem-Xasfunction}
As above, let $X_1 =\frac{u_1-u_0}{a(u_0-\pi)}\in\frakS_{\ast,\dR,m}^{1,+}$ with respect to the uniformizer $A=u-\pi$ of $\frakS_{\dR}^+$.  
Then as a function from $G_K$ to $\bfB_{\dR,m}^+$, we have that for any $g\in G_K$,
    \[X_1(g) = \frac{([\epsilon]^{c(g)}-1)[\pi^{\flat}]}{a([\pi^{\flat}]-\pi)}.\]
  \end{lem}
  \begin{proof}
    By the proof of Lemma \ref{Lem-Galois descent}, we see that as functions from $G_K$ to $\Ainf$, $u_0(g) = [\pi^{\flat}]$ is the constant function while $u_1(g) = g([\pi^{\flat}]) = [\epsilon]^{c(g)}[\pi^{\flat}]$ is the evaluation function. Then the lemma follows.
  \end{proof}

  \begin{prop}\label{Prop-MatrixCocycle}
    Keep notations as above. Let $W$ be the $\bfB_{\dR,m}^+$-representation induced by $\bM$. Via the morphism of (log-) prisms in Notation \ref{notaainflog}, there is an identification     $$M\otimes_{\frakS_{\dR,m}^+}\bfB_{\dR,m}^+ \simeq W. $$
    For $g\in G_K$, its action on $W$ is   determined such that for any $x\in M$,
    \begin{equation}\label{Equ-MatrixCocycle}
        g(x) = \sum_{n\geq 0}a^n\prod_{i=0}^{n-1}(\nabla_M-i)(x)(\frac{([\epsilon]^{c(g)}-1)[\pi^{\flat}]}{a([\pi^{\flat}]-\pi)})^{[n]}.
    \end{equation}
  \end{prop}
  \begin{proof} One simply plug Lemma \ref{Lem-Xasfunction} into \eqref{eqsec7vare}.
   Note that the infinite summation on right hand side of \eqref{Equ-MatrixCocycle} converges. In fact,  the following Lemma \ref{lemma_converge}  implies $\frac{1}{n!}\cdot (\frac{([\epsilon]^{c(g)}-1)[\pi^{\flat}]}{a([\pi^{\flat}]-\pi)})^{n}$ converges to $0$ in $\bdrplusm$. Indeed, it suffices to consider the case $a=-\pi E'(\pi)$, then   one can easily compute
   \[v_p(\theta(X_1))= v_p(\theta(\frac{([\epsilon]^{c(g)}-1)[\pi^{\flat}]}{a([\pi^{\flat}]-\pi)})   ) \geq     v_p(\theta(\frac{\xi \cdot ([\epsilon]^{1/p}-1)}{E([\pi^\flat])})  ) =v_p( \mu_p-1) =\frac{1}{p-1} \]
  \end{proof}

\begin{lemma}  \label{lemma_converge} \cite[Lem. 3.2]{Ber14}.
Let $f(T)=\sum_{k \geq 0} a_kT^k \in \barK[[T]]$. Let $x \in \bdrplusm$. Then the series $f(x)$ converges in $\bdrplusm$ (using the topology in Notation \ref{notabdrplus}) if and only if the series $f(\theta(x))$ converges in $C$.
\end{lemma}


\section{Locally analytic vectors: axiomatic computations}\label{seclav}

 In this section, we axiomatize some computations of locally analytic vectors in the literature. This will streamline the computations in our concrete context.

Let us very quickly recall the theory of locally analytic vectors, see \cite[\S 2.1]{BC16} and \cite[\S 2]{Ber16} for more details.
Recall the multi-index notations: if $\cbf = (c_1, \hdots,c_d)$ and $\kbf = (k_1,\hdots,k_d) \in \mathbb{N}^d$ (here $\mathbb{N}=\mathbb{Z}^{\geq 0}$), then we let $\cbf^\kbf = c_1^{k_1} \cdot \ldots \cdot c_d^{k_d}$. Recall that a $\Qp$-Banach space $W$ is a $\Qp$-vector space with a complete non-Archimedean  norm $\|\cdot\|$ such that $\|aw\|=\|a\|_p\|w\|, \forall a \in \Qp, w \in W$, where $\|a\|_p$ is the   $p$-adic norm on $\Qp$.

\begin{defn}\label{defLAV}
\begin{enumerate}
\item
Let $G$ be a $p$-adic Lie group, and let $(W, \|\cdot \|)$ be a $\Qp$-Banach representation of $G$.
Let $H$ be an open subgroup of $G$ such that there exist coordinates $c_1,\hdots,c_d : H \to \Zp$ giving rise to an analytic bijection $\cbf : H \to \mathbb{Z}_p^d$.
 We say that an element $w \in W$ is an $H$-analytic vector if there exists a sequence $\{w_\kbf\}_{\kbf \in \mathbb{N}^d}$ with $w_\kbf \to 0$ in $W$, such that $$g(w) = \sum_{\kbf \in \mathbb{N}^d} \cbf(g)^\kbf w_\kbf, \quad \forall g \in H.$$
Let $W^{H\dan}$ denote the space of $H$-analytic vectors.

\item $W^{H\dan}$ injects into $\mathcal{C}^{\an}(H, W)$ (the space of analytic functions on $H$ valued in $W$), and we endow it with the induced norm, which we denote as $\|\cdot\|_H$. We have $\|w\|_H=\sup_{\kbf \in \mathbb{N}^d}\|w_{\kbf}\|$, and $W^{H\dan}$ is a Banach space.

\item
We say that a vector $w \in W$ is \emph{locally analytic} if there exists an open subgroup $H$ as above such that $w \in W^{H\dan}$. Let $W^{G\dla}$ denote the space of such vectors.

\item We can naturally extend these definitions to the case when $W$ is a Fr\'echet- or LF- representation of $G$, and use $W^{G\dpa}$ to denote the \emph{pro-analytic} vectors, cf. \cite[\S 2]{Ber16}.
\end{enumerate}
\end{defn}

\begin{defn}\label{defnquasibasis}
Let $M$ be a finitely generated  module over a ring $R$, say $m_1, \cdots, m_d \in M$ form a \emph{quasi-basis} if $M=\oplus_{i=1}^d R \cdot m_i$.
\end{defn}

\begin{remark}
Note that any finitely generated module over a PID admits quasi-bases, which are of a fixed cardinality. However, for a general ring, this notion might be ill-behaved. In our application, we are normally given a quasi-basis from the start, which itself come from certain modules over a PID (say, $\bdrplus$ or $F[[x]]$ with $F$ a field).
\end{remark}

\begin{convention}\label{conv:identification}
Let $M$ be a $A$-module where $A$ is a ring.   Let $B\subset A$ be a subring, and let $N \subset M$ be a sub-$B$-module.
If the natural map $N\otimes_B A \to M$ is an isomorphism, then we call it an \emph{identification}, and simply write
\[ N\otimes_B A \xrightarrow{\simeq} M \]
or just
\[N\otimes_B A = M\]
\end{convention}

 \begin{lemma} \label{lemLAVqbasis}
Let $G$ be a $p$-adic Lie group, let $R$ be a $p$-adic Banach ring with a $G$-action, and let $M$ be a finitely generated $R$-module. Suppose there exists a quasi-basis $m_1, \cdots, m_d$ of $M$ such that each $m_i \in M^{G\dla}$, then $m_1, \cdots, m_d$  form a quasi-basis for the $R^{G\dla}$-module $M^{G\dla}$; namely,
$$M^{G\dla} =\oplus_{i=1}^d R^{G\dla} \cdot m_i; $$
hence we have a natural identification
\[M^{G\dla}\otimes_{R^{G\dla}}  R =M.\]
\end{lemma}
\begin{proof}
This is easy generalization of \cite[Prop. 2.3]{BC16} which treated the case when $m_1, \cdots, m_d$ is a (genuine) basis.
\end{proof}

\begin{defn}
Let $(H, \|\cdot \|)$ be a $\Qp$-Banach algebra such that $\|\cdot \|$ is sub-multiplicative, and let $W \subset H$ be a $\Qp$-subalgebra. Let $T$ be a variable, and let  $W \dacc{T}_n$ be the vector space consisting of $\sum_{k \geq 0} a_k T^k$ with $a_k \in W$, and $p^{nk} a_k \to 0$ when $k \to +\infty$. For $h \in H$ such that $\|h \|\leq p^{-n}$, denote $W \dacc{h}_n$ the image of the evaluation map $W \dacc{T}_n \to H$ where $T \mapsto h$.
\end{defn}

 \begin{notation}\label{notaG_n}
 Let $G$ be a $p$-adic Lie group of dimension $d$.
As explained above \cite[Lem. 2.4]{BC16}, there exists some compact open subgroup $G_1$, which admits local coordinates $c_i: G_1 \to \zp$ for each $1\leq i \leq d$, and such that if we let $G_n:=G_1^{p^{n-1}}$, then $\mathbf{c}(G_n) =(p^n\zp)^d, \forall n$.
 \end{notation}

\begin{prop}\label{propaxiomring}
Let $G, G_n$ be as in Notation \ref{notaG_n}, and let $(W, \|\cdot \|)$ be a $\Qp$-Banach algebra with a $G$-action.
Let $\delta \in \Lie(G)\otimes_\zp W^{G\dla}$, which induces a $\qp$-linear map $\delta: W^{G\dla} \to W^{G\dla}$ satisfying Leibniz law.
Suppose the following conditions are satisfied:
\begin{itemize}
\item  There exists $a \in W^{G\dla}$ such that $\delta(a)=1$;
\item  for each $n \geq 1$, there exists some $a_n \in W^{G\dla, \delta=0}$ and $r(n) \geq 1$ such that for any $m \geq r(n)$, we have $a-a_n \in W^{G_m\dan}$ and $||a-a_n||_{G_m} \leq p^{-n}$. (That is: $a$ can be uniformly approximated).
\end{itemize}
Choose the $r(n)$'s so that it becomes a strictly increasing sequence. Then we have
\begin{equation}
 W^{G\dla} = \bigcup_{n \geq 1} (W)^{G_n\dan, \delta=0}\dacc{a-a_n}_n
\end{equation}
\end{prop}
\begin{proof}
This axiomatizes several computations made in the literature, the first being \cite[Thm. 4.2, Prop. 4.12]{BC16}; similar ideas are used in \cite[Thm. 5.4]{Ber16}, \cite[Thm. 5.3.5]{GP21} etc (the final instance will be reviewed in Construction \ref{verifyaxiomwtb}). Here, let us only give a brief sketch (the readers can consult e.g. \cite{BC16} for a detailed computation).

Suppose $x\in W^{G_\ell\dan}$ for some $\ell \geq 1$. Then \cite[Lem. 2.6]{BC16} (which easily adapts to our \emph{fixed} $\delta \in \mathrm{Lie} G\otimes W^{G\dla}$) implies that there exists some $n \gg 0$, such that if we define
$$ z_i = \sum_{k \geq 0} (-1)^k (a - a_n)^k \delta ^{k+i}(x) \binom{k+i}{k},$$
then for any $m\geq r(n)$, we have $z_i \in W^{G_m\dan}$, and
\begin{equation} \label{eqxbeta}
x = \sum_{i \geq 0} z_i (a - a_n)^i \in W^{G_m\dan}.
\end{equation}
Finally, note that since $\delta (a - a_n) =\delta (a)=1$, we have $\delta(z_i)=0$.
\end{proof}

\begin{prop}\label{propaxiommodule}
Let $G, G_n$ be as in Notation \ref{notaG_n}, let $R$ be a $\qp$-Banach algebra with a $G$-action such that $R^{G\dla}=R$ (that is, the $G$-action is locally analytic). Suppose there exists  $\delta \in \Lie(G)\otimes_\zp R$ such that the itemized conditions in Prop. \ref{propaxiomring} are satisfied: that is, there exist $a \in R$ such that $\delta(a)=1$ and it can be approximated by $a_n$.

Let $M$ be a finitely generated $R$-module with a quasi-basis equipped with a $R$-semi-linear $G$-action and such that $M^{G\dla}=M$. Then $M^{\delta=0}$ is a finitely generated $R^{\delta=0}$-module with a quasi-basis and there is a natural identification
$$M^{\delta=0} \otimes_{R^{\delta=0}} R \xrightarrow{\simeq} M.$$
\end{prop}
\begin{proof}
This again axiomatizes several computations made in the literature, the first being \cite[Thm. 6.1]{Ber16}; cf. also \cite[Rem. 6.1.7]{GP21} which carries out similar computations. Again, we only give a brief sketch.

Let  $D_\delta=\Mat(\delta)$ be the matrix of $\delta: M\to M$ with respect to the given  quasi-basis of $M$.
Then it suffices to show that there exists $H\in \GL_d(R)$ such that
 \begin{equation}\label{eqdgamma}
 \delta(H)+D_{\delta}H = 0
 \end{equation}
For $k \in \mathbb N$, let $D_k = \Mat(\delta^k)$. For $n$ large enough, the series given by
$$H = \sum_{k \geq 0}(-1)^k D_k\frac{(a-a_n)^k}{k!}$$
converges   to the desired  solution of \eqref{eqdgamma}; note that the entries of $H$ are in $R$ precisely by Prop. \ref{propaxiomring}.
\end{proof}

\section{Locally analytic vectors in rings}\label{subsecwtbi}
In this section, we show that the rings  $\wt{\mathbf{B}}^I_L$ and $\bdrplusml$ satisfy the axioms in Prop. \ref{propaxiomring}. This is in preparation for next section, where we will used the axiomatized Prop. \ref{propaxiommodule} to construct Sen--Fontaine theory over the Kummer tower.

\subsection{The Lie group $\hat{G}$}
 We set up some notations about the $p$-adic Lie group $\hat{G}$ defined in Notation \ref{notafields}.

\begin{Notation} \label{nota hatG}
Note $\hat{G}=\gal(L/K)$ is a $p$-adic Lie group of dimension 2. We recall the structure of this group in the following.
\begin{enumerate}
\item Recall that:
\begin{itemize}
\item if $K_{\infty} \cap K_{p^\infty}=K$ (always valid when $p>2$, cf. \cite[Lem. 5.1.2]{Liu08}), then $\gal(L/K_{p^\infty})$ and $\gal(L/K_{\infty})$ generate $\hat{G}$;
\item if $K_{\infty} \cap K_{p^\infty} \supsetneq K$, then necessarily $p=2$, and $K_{\infty} \cap K_{p^\infty}=K(\pi_1)$ (cf. \cite[Prop. 4.1.5]{Liu10}) and $\pm i \notin K(\pi_1)$, and hence $\gal(L/K_{p^\infty})$ and $\gal(L/K_{\infty})$   generate an open subgroup  of $\hat{G}$ of index $2$.
\end{itemize}

\item Note that:
\begin{itemize}
\item $\gal(L/K_{p^\infty}) \simeq \Zp$, and let
$\tau \in \gal(L/K_{p^\infty})$ be \emph{the} topological generator such that
\begin{equation} \label{eq1tau}
\begin{cases}
\tau(\pi_i)=\pi_i\mu_i, \forall i \geq 1, &  \text{if }  \Kinfty \cap \Kpinfty=K; \\
\tau(\pi_i)=\pi_i\mu_{i-1}=\pi_i\mu_{i}^2, \forall i \geq 2, & \text{if }  \Kinfty \cap \Kpinfty=K(\pi_1).
\end{cases}
\end{equation}

\item $\gal(L/K_{\infty})$ ($\subset \gal(K_{p^\infty}/K) \subset \Zp^\times$) is not necessarily pro-cyclic when $p=2$; however, this issue will \emph{never} bother us.
\end{itemize}
\item If we let $\Delta \subset \gal(L/K_{\infty})$ be the torsion subgroup, then $\gal(L/K_{\infty})/\Delta$  is pro-cyclic; choose $\gamma' \in \gal(L/K_{\infty})$ such that its image in $\gal(L/K_{\infty})/\Delta$ is a topological generator.
Let $\hat{G}_n \subset \hat{G}$ be the subgroup topologically generated by $\tau^{p^n}$ and $(\gamma')^{p^n}$. These subgroups then satisfy the conditions in Notation \ref{notaG_n}.
\end{enumerate}
\end{Notation}

\begin{Notation} \label{notataula}
We set up some notations with respect to representations of $\hat{G}$.
\begin{enumerate}
\item Given a $\hat{G}$-representation $W$, we use
$$W^{\tau=1}, \quad W^{\gamma=1}$$
to mean $$ W^{\gal(L/K_{p^\infty})=1}, \quad
W^{\gal(L/K_{\infty})=1}.$$
And we use
$$
W^{\tau\dla},   \quad  W^{\gamma\dla} $$
to mean
$$
W^{\gal(L/K_{p^\infty})\dla}, \quad
W^{\gal(L/K_{\infty})\dla}.  $$

\item  Let
$W^{\tau\dla, \gamma=1}:= W^{\tau\dla} \cap W^{\gamma=1},$
then by \cite[Lem. 3.2.4]{GP21}
$$ W^{\tau\dla, \gamma=1} \subset  W^{\hat{G}\dla}. $$
\end{enumerate}
\end{Notation}


\begin{notation} \label{notalieg}
For $g\in \hat{G}$, let $\log g$ denote  the (formally written) series $(-1)\cdot \sum_{k \geq 1} (1-g)^k/k$. Given a $\hat{G}$-locally analytic representation $W$, the following two Lie-algebra operators (acting on $W$) are well defined:
\begin{itemize}
\item  for $g\in \gal(L/\kinfty)$ enough close to 1, one can define $\nabla_\gamma := \frac{\log g}{\log(\chi_p(g))}$;
\item for $n \gg 0$ hence $\tau^{p^n}$ enough close to 1, one can define $\nabla_\tau :=\frac{\log(\tau^{p^n})}{p^n}$.
\end{itemize}
Clearly, these two Lie-algebra operators form a $\qp$-basis of $\Lie(\hat{G}$).
\end{notation}

\subsection{Locally analytic vectors in $\tilde{\mathbf{B}}^{I}$}
\begin{notation}
We briefly recall the rings $\wt{\mathbf{A}}^{I}$ and  $\wt{\mathbf{B}}^{I}$, see \cite[\S 2]{GP21} for detailed discussions, also see \cite[\S 2]{Gao23} for a faster summary.

Recall in Notation \ref{notaBKprism}, using the sequence $\pi_n$, we defined an embedding $\gs \into \ainf$ where $u$ maps to  $[\pi^\flat]$.
Henceforth, we  identify  $u$  with the element $[\pi^\flat] \in \ainf$.
For $n \geq 0$, let $r_n: =(p-1)p^{n-1}$.
Let $\wt{\mathbf{A}}^{[r_\ell, r_k]}$ be the $p$-adic completion of $ \ainf [\frac{p}{u^{ep^\ell}} , \frac{u^{ep^k}}{p}]$, and let
$$\wt{\mathbf{B}}^{[r_\ell, r_k]}: =\wt{\mathbf{A}}^{[r_\ell, r_k]}[1/p].$$
These spaces are equipped with $p$-adic topology. When $I  \subset J$ are two closed intervals as above, then by \cite[Lem. 2.5]{Ber02}, there exists a natural (continuous) embedding
$\wt{\mathbf{B}}^{J}   \hookrightarrow  \wt{\mathbf{B}}^{I}$. Hence we can define the nested intersection
$$\wt{\mathbf{B}}^{[r_\ell, +\infty)}: = \bigcap_{k \geq \ell} \wt{\mathbf{B}}^{[r_\ell, r_k]},$$
and equip it with the natural Fr\'echet  topology.  Finally, let
$$\wt{\mathbf{B}}_{  \rig}^{\dagger}: = \bigcup_{n \geq 0} \wt{\mathbf{B}}^{[r_n, +\infty)},$$
 which is a LF space.

\end{notation}

\begin{convention}
When $Y$ is a ring with a $G_K$-action, $X \subset \overline{K}$ is a subfield, we use $Y_X$ to denote the $\gal(\overline{K}/X)$-invariants of  $Y$.   Examples include when $Y=\wt{\mathbf{A}}^{I}, \wt{\mathbf{B}}^{I}$ and $X=L, K_\infty$. This ``style of notation" imitates that of \cite{Ber02}, which uses the subscript $\ast_{K}$ to denote $G_{\kpinfty}$-invariants.
\end{convention}

\begin{defn} \label{defnfkt}
(cf. \cite[\S 5.1]{GP21} for full details).
 Recall in Notation \ref{notaainflog}, we defined an element   $[\epsilon] \in \ainf$.
Let $t=\log([\underline \epsilon])  \in \bcrisplus$ be the usual element, where $\bcrisplus$ is the usual crystalline period ring.
Recall we already defined the element $\lambda :=\prod_{n \geq 0} (\varphi^n(\frac{E(u)}{E(0)}))  \in \bcrisplus.$ in Notation \ref{notafields}.
Define
$$ \mathfrak{t} = \frac{t}{p\lambda},$$
then it turns out $\mathfrak{t} \in \ainf$.
  \end{defn}

\begin{lemma} \label{lem b}
\cite[Lem. 5.1.1]{GP21}
There exists some $n=n(\fkt) \geq 0$, such that $\mathfrak{t}, 1/\mathfrak{t} \in
  \wt{\mathbf{B}}^{[r_n, +\infty)}$. In addition, $\mathfrak{t}, 1/\mathfrak{t} \in
 (\wt{\mathbf{B}}^{[r_n, +\infty)}_{ L})^{\hat{G}\dpa}$.
\end{lemma}

\begin{construction}\label{verifyaxiomwtb}
We verify that the ring $\wt{\mathbf{B}}^I_L:=(\wt{\mathbf{B}}^I)^{G_L}$ satisfies the axioms in Prop. \ref{propaxiomring}. This is indeed a summary of main results in \cite[\S 5]{GP21}, particularly \cite[Thm. 5.3.5]{GP21}.

Let $n \geq n(\fkt)$ as in Lem. \ref{lem b}, and
Let $I=[r_n, r_n]$. (In fact, all discussions here are valid for any $[r, s] \subset (0, +\infty)$ with  $r\geq r_n$).
 Consider $\hat G$-action on the  $\qp$-Banach ring $\wt{\mathbf{B}}^I_L$.
Consider the differential operator $\partial_{\gamma}: (\wt{\mathbf{B}}^I_L)^{\hat{G}\dla} \to (\wt{\mathbf{B}}^I_L)^{\hat{G}\dla}$
via
$$\partial_{\gamma}:=\frac{1}{\fkt}\nabla_{\gamma}.$$
Note we can divide by $\fkt$ because it is an unit in $(\wt{\mathbf{B}}^I_L)^{\hat{G}\dla} $.
Then $\partial_{\gamma}(\fkt)=1$ (by easy computation). Furthermore, $\fkt$ can be ``approximated" by a sequence of elements $\fkt_n \in (\wt{\mathbf{B}}_L^{I})^{\hat{G}\dla, \nabla_\gamma=0}$ (constructed via \cite[Lem. 5.3.2]{GP21}) in the sense that they satisfy the conditions in Prop. \ref{propaxiomring}. Hence all the conditions in Prop. \ref{propaxiomring} are satisfied. Thus the $\hat G$-locally analytic vectors can be described as
\begin{equation}
(\wt{\mathbf{B}}^I_L)^{\hat{G}\dla} = \bigcup_{n \geq 1} (\wt{\mathbf{B}}^I_L)^{\hat{G}_n\dan, \partial_\gamma=0}\dacc{\fkt-\fkt_n}_n
\end{equation}
Here, the elements in $(\wt{\mathbf{B}}^I_L)^{\hat{G}_n\dan, \partial_\gamma=0} =(\wt{\mathbf{B}}^I_L)^{\hat{G}_n\dan, \nabla_\gamma=0}$ can be \emph{explicitly} described by \cite[Thm. 4.2.9, Lem. 5.3.1]{GP21}; these explicit descriptions are not needed in this paper.
\end{construction}

\subsection{Locally analytic vectors in $\bdrplusml$}

\begin{notation}
Recall in Notation \ref{notabdrplus},  for $m \geq 1$ an integer, we denote
$$\bdrplusm =\bdrplus/(\xi)^m \simeq \ainf[1/p]/(\xi)^m$$
and consider it as a $\qp$-Banach space.
Denote
$$\bdrplusml =(\bdrplusm)^{G_L}.$$
\end{notation}

\begin{construction}\label{constiotanm}
Recall as in \cite{Ber02}, there exists a continuous embedding
$$\iota_n: \wtb^{[r_n, r_n]} \into  \bdrplus$$
which is the composite of the following continuous maps
$$\wtb^{[r_n, r_n]} \xrightarrow{\varphi^{-n}} \wtb^{[r_0, r_0]}\xrightarrow{\iota_0} \bdrplus.$$
For each $m \geq 1$,  define the composite map
\[\iota_{n, m} : \wtb^{[r_n, r_n]} \xrightarrow{\iota_n}  \bdrplus \onto \bdrplusm.\]
\end{construction}

\begin{lemma}
Consider $\fkt$ as an element of $\bdrplus$ via the inclusion $\ainf \subset \bdrplus$, then $\fkt^{\pm 1} \in \bdrpluslhatgpa$.
\end{lemma}

\begin{proof}
Choose $n \geq n(\fkt)$ as in Lem. \ref{lem b}  so that $\fkt^{\pm 1} \in (\wtb^{[r_n, r_n]}_L)^{\hat{G}\dla}$.
Consider the image of $\fkt^{\pm 1}$ under the following composite  map
$$\wtb^{[r_n, r_n]} \xrightarrow{\varphi^{-n}} \wtb^{[r_0, r_0]}\xrightarrow{\iota_0} \bdrplus.$$
Since both maps are continuous, we conclude
 $$ \varphi^{-n}(\fkt^{\pm 1}) \in \bdrpluslhatgpa.$$
Since $\varphi(\fkt) =\frac{pE(u)}{E(0)} \fkt$, we have
\begin{equation}\label{eqfktandphi}
\fkt = \varphi^{-n}(\fkt) \cdot \prod_{i=1}^n \varphi^{-i}(\frac{pE(u)}{E(0)}),
\end{equation}
which holds as an equality inside $\ainf$ hence also in $\bdrplus$. To see that $ \fkt^{\pm 1} $ are (pro)-locally analytic, it then suffices to see that each $ \varphi^{-i}(\frac{pE(u)}{E(0)})$ and their inverse are so; but these elements are locally analytic in $\wtb^{[r_0, r_0]}$ by \cite[Thm. 3.4.4]{GP21}.
\end{proof}

\begin{construction} \label{axiomverifybdr}
We verify that $\hat G$-action on the  $\qp$-Banach ring $\bdrplusml$ satisfies the axioms in Prop. \ref{propaxiomring}.
  Consider the  differential operator $\partial_{\gamma}: (\bdrplusml)^{\hat{G}\dla} \to (\bdrplusml)^{\hat{G}\dla} $
via
$$\partial_{\gamma}:=\frac{1}{\fkt}\nabla_{\gamma}.$$
Then $\partial_{\gamma}(\fkt)=1$.
Fix one $N \geq n(\fkt)$ as in Lem. \ref{lem b}, and let $I=[r_N, r_N]$. As we discussed in Construction \ref{verifyaxiomwtb}, the element $\fkt \in (\wt{\mathbf{B}}^I_L)^{\hat{G}\dla} $ can be approximated by a sequence $\fkt_n$ satisfying all the axioms in  Prop. \ref{propaxiomring}.
As the map
\[\iota_{N, m}: \wtb^I \to \bfb_{\dR, m}^+ \]
 in Construction \ref{constiotanm} is continuous, the images $\iota_{N, m}(\fkt_n)$ approximate $\iota_{N, m}(\fkt) =\varphi^{-N}(\fkt)$ (regarded as an element in $\bfb_{\dR, m}^+$).
Denote $\fkc= \prod_{i=1}^N \varphi^{-i}(\frac{pE(u)}{E(0)})$, then by \eqref{eqfktandphi}, the sequence $\fkc  \cdot \iota_{N, m}(\fkt_n)$ approximate $\fkt$. In fact, by refining the sequence $\fkt_n$ if necessary, we can make sure all the conditions in  Prop. \ref{propaxiomring} are satisfied; namely, we can ensure (changing the index $m$ there to $M$)
$$||\fkt -\fkc \cdot \iota_{N, m}(\fkt_n)||_{G_M} \leq p^{-n}, \quad \forall M \geq r(n);$$
this is because $\fkc$ is a fixed element (hence with fixed norm).
As a consequence of Prop. \ref{propaxiomring}, we have
\begin{equation}
 (\bdrplusml)^{\hat{G}\dla} = \bigcup_{n \geq 1} (\bdrplusml)^{\hat{G}_n\dan, \partial_\gamma=0}\dacc{\fkt- \fkc  \cdot \iota_{N, m}(\fkt_n)}_n
\end{equation}

\end{construction}

\begin{lemma} \label{lemdividet}
Let $x \in  \bdrplusml$.
\begin{enumerate}
\item If $tx \in  (\bdrplusml)^{\tau=1, \gamma\dla}$, then $x \in  (\bdrplusml)^{\tau=1, \gamma\dla}$.

\item If $\lambda x \in  (\bdrplusml)^{\gamma=1, \tau\dla}$, then $x \in   (\bdrplusml)^{\gamma=1, \tau\dla}$.

\item If $\lambda x \in  (\bdrplusml)^{\hat{G}\dla, \nabla_\gamma=0}$, then $x \in   (\bdrplusml)^{\hat{G}\dla, \nabla_\gamma=0}$.
\end{enumerate}
\end{lemma}
\begin{proof}
Item (1) follows by inspection of the formula in Def. \ref{defLAV}(1).

In Item (2), it suffices to show $x$ is $\tau$-locally analytic.
Note  $tx=p\fkt \lambda x$ is $\tau$-locally  analytic, hence an inspection of the formula in Def. \ref{defLAV}(1) shows again $x$ is $\tau$-locally analytic.

For Item (3), note that an element is killed by $\nabla_\gamma$ if and only if it is fixed by $\gal(L/\kpinfty(\pi_n))$ for some $n \geq 0$, and hence
\[(\bdrplusml)^{\hat{G}\dla, \nabla_\gamma=0} =
\cup_n (\bdrplusml)^{\tau\dla, \gal(L/\kpinfty(\pi_n))=1}.
\]
Then we can apply similar argument as in Item (2).
\end{proof}

\begin{prop}\label{propbdrliezero}
We have
\begin{enumerate}
\item $(\mathbf{B}_{\dR, m, L}^+)^{\tau=1, \gamma\dla}  =  \Kpinfty[[t]]/t^m$
\item $(\mathbf{B}_{\dR, m, L}^+)^{\gamma=1, \tau\dla }  =  \Kinfty[[\lambda]]/\lambda^m.$
\item $(\mathbf{B}_{\dR, m, L}^+)^{\hat{G}\dla, \nabla_\gamma=0 }  =  L[[\lambda]]/\lambda^m.$
\end{enumerate}
\end{prop}
\begin{proof}
We only prove the second item, the other items follow from similar argument.

It is easy to check that $ \Kinfty[[\lambda]]/\lambda^m \subset  (\mathbf{B}_{\dR, m, L}^+)^{\gamma=1, \tau\dla }. $
Now, let $x\in (\mathbf{B}_{\dR, m, L}^+)^{\gamma=1, \tau\dla } $ be a locally analytic vector. Consider the projection map \[\theta: \mathbf{B}_{\dR, m, L}^+ \to \mathbf{B}_{\dR, 1, L}^+ =C,\] the image $\theta(x) \in (\hat{L})^{\gamma=1, \tau\dla } =\kinfty$ by \cite[Prop. 3.3.2]{GP21}.
Since $\bdrplus$ is a $\overline{K}$-algebra, the element $x -\theta(x) \in (\mathbf{B}_{\dR, m, L}^+)^{\gamma=1, \tau\dla }$, and is divisible by $\lambda$. Lem. \ref{lemdividet} implies that
$$x_1=\frac{x -\theta(x)}{\lambda} \in (\mathbf{B}_{\dR, m, L}^+)^{\gamma=1, \tau\dla }$$
Apply $\theta$ map to $x_1$, and iterate the above argument; in the end, we see that $x\in  \Kinfty[[\lambda]]/\lambda^m$.
\end{proof}

\section{Sen--Fontaine theory for $\BdRp$- representations over Kummer tower}
\label{sec_kummersen}

In this section, we will construct a Sen--Fontaine theory over the Kummer tower; that is, we will construct certain log-$\lambda$-connections (cf. Def. \ref{defn_log_T_conn}) over the ring $\Kinfty[[\lambda]]$.




\begin{remark} \label{remuselambda}
Note that the ring $\Kinfty[[\lambda]]$ has several natural uniformizers, cf. Lem. \ref{lemgsdeunif}; for example
\[ \Kinfty[[\lambda]] = \Kinfty[[E(u)]] = \Kinfty[[u-\pi]]. \]
Hence by Lem. \ref{lemchangeconn}, it is equivalent to consider log-$\lambda$-connections resp. log-$E(u)$-connections resp. log-$(u-\pi)$-connections. In this section (and what follows), we choose to use  log-$\lambda$-connections. Note this will actually  complicate (although just slightly) our computations, however we have chosen to be stubborn by the following reasons:
\begin{itemize}
\item We regard the element $\lambda$ as a more natural ``analogue" of the element $t$ used in Fontaine's original theory of $\bdrplus$-representations \cite{Fon04}. Indeed, the important element $\mathfrak t$ in Def. \ref{defnfkt} is defined as the ratio between these two elements.

\item We need to use log-$\lambda$-connections so that the formula of the Sen--Fontaine operator over the Kummer tower in Thm. \ref{thm-compare-M-Ddif} matches that of  Sen operator over the Kummer tower defined in \cite{GMWHT}.
\end{itemize}
\end{remark}

\subsection{Sen--Fontaine theory over cyclotomic tower}
We quickly recall Fontaine's theory   studying $\bdrplus$-representations.

\begin{notation}
Recall in Def. \ref{defsemilinrep}, we defined the category $\rep_{\mathcal G}(R)$ consisting of \emph{finite free} semi-linear representations. In the following, we simply use $\rep^\fg_{\mathcal G}(R)$ denote the similar category consisting of finitely generated semi-linear representations.
\end{notation}

 \begin{theorem} \cite[\S 3.3]{Fon04}
 Base change functors induce exact tensor equivalences of categories
 \begin{equation*}
 \rep^\fg_\gammak(\kpinfty[[t]]) \simeq \rep^\fg_\gammak(\bdrpluskpinfty) \simeq \rep^\fg_\gk(\bdrplus).
 \end{equation*}
 Here, given $W\in \rep^\fg_\gk(\bdrplus)$, the corresponding object in $\rep^\fg_\gammak(\bdrpluskpinfty)$ is $W^{G_\kpinfty}$, and the corresponding object in $\rep^\fg_\gammak(\kpinfty[[t]])$ is
\begin{equation*}\label{senlav}
D^+_{\dif, \kpinfty}(W): =(W^{G_\kpinfty})^{\gammak\dla}
\end{equation*}
 \end{theorem}

\begin{construction} \cite[\S 3.4]{Fon04}
Given $W\in \rep^\fg_\gk(\bdrplus)$, since $\Gamma_K$ acts on $D^+_{\dif, \kpinfty}(W)$ locally analytically,
the Lie algebra operator
$$\nabla_\gamma: D^+_{\dif, \kpinfty}(W) \to D^+_{\dif, \kpinfty}(W)$$
defines an object in $\mathrm{MIC}^\fg_t(\kpinftyt)$. This construction induces an exact tensor functor:
$$\rep^\fg_\gk(\bdrplus) \to \mathrm{MIC}^\fg_t(\kpinftyt).$$
\end{construction}

 \begin{remark}
 Clearly, if we consider objects killed by $t$, namely objects in $\rep_\gk(C)$, then we recover Sen theory \cite{Sen80}.
 \end{remark}

\subsection{Sen--Fontaine theory over Kummer tower} \label{subsec KummerSen}
  We now develop an analogue of Sen--Fontaine theory over the Kummer tower $\kinfty$; this generalizes   Sen theory over the Kummer tower  in \cite{GMWHT}. We mention that these theories are inspired by quite similar constructions in overconvergent $(\varphi, \tau)$-modules in \cite{GL20, GP21}; cf. \cite[\S 6.2]{GMWHT} for a motivated review.

 \begin{theorem}\label{thm331kummersenmod}
 Given  $W\in \rep^\fg_\gk(\bdrplus)$, define
 \begin{equation*}
 D^+_{\dif, \kinfty}(W):= (W^{G_L})^{\tau\dla, \gamma=1}.
 \end{equation*}
 Then this is a finitely generated $\kinfty[[\lambda]]$-module, and there are identifications
 \begin{equation*}\label{eqkpk}
 D^+_{\dif, \kinfty}(W) \otimes_{\kinfty[[\lambda]]} (\B_{\dR, L}^+)^{\hat{G}\dla} = D^+_{\dif, \kpinfty}(W) \otimes_{\kpinfty[[t]]} (\B_{\dR, L}^+)^{\hat{G}\dla}=(W^{G_L})^{\hat{G}\dla}.
 \end{equation*}
 \end{theorem}
\begin{proof}
By considering inverse limits, it suffices to treat torsion objects, hence we can suppose  $W\in \rep^\fg_\gk(\bdrplusm)$ with some $m <+\infty$ (and hence also the relevant spaces are Banach representations).
By Lem. \ref{lemLAVqbasis}, we know
  \begin{equation*}\label{eq316first}
  (W^{G_L})^{\hat{G}\dla} =D_{\dif, \kpinfty}(W) \otimes_\kpinftyt (\bdrplusml)^{\hat{G}\dla}.
  \end{equation*}
Hence we can apply Prop. \ref{propaxiommodule} ---as the axioms are verified in Construction \ref{axiomverifybdr}--- to see that
$(W^{G_L})^{\hat{G}\dla, \nabla_\gamma=0}$
  is finitely generated over $(\mathbf{B}_{\dR, m, L}^+)^{\hat{G}\dla, \nabla_\gamma=0 }$  together with an identification
  \[  (W^{G_L})^{\hat{G}\dla, \nabla_\gamma=0} \otimes_{(\mathbf{B}_{\dR, m, L}^+)^{\hat{G}\dla, \nabla_\gamma=0 }} (\bdrplusml)^{\hat{G}\dla} =   (W^{G_L})^{\hat{G}\dla} \]
  Recall by Prop. \ref{propbdrliezero},
  $(\mathbf{B}_{\dR, m, L}^+)^{\hat{G}\dla, \nabla_\gamma=0 } =  L[[\lambda]]/\lambda^m$.
By considering $M=\kinfty(\mu_n)$ with $n \gg 0$, one sees that $(W^{G_L})^{\hat{G}\dla, \gal(L/M)=1}$ is finitely generated over  $(L[[\lambda]]/\lambda^m)^{\gal(L/M)=1}=M[[\lambda]]/\lambda^m$.
 Finally one applies \'etale descent (to the \'etale map $\kinfty[[\lambda]]/\lambda^m \to M[[\lambda]]/\lambda^m$) to conclude that
 \[D^+_{\dif, \kinfty}(W) =((W^{G_L})^{\hat{G}\dla, \gal(L/M)=1})^{\gal(M/\kinfty)=1} \]
  satisfies all the desired properties.
\end{proof}

\begin{defn} \label{defndiffwtb}
 Define a differential operator
$$N_\nabla: (\bdrplusl)^{\hat{G}\dla} \to (\bdrplusl)^{\hat{G}\dla}$$
by setting
\begin{equation}\label{eqnnring:dr}
{N_\nabla:=}
\begin{cases}
\frac{1}{p\mathfrak{t}}\cdot \nabla_\tau, &  \text{if }  \Kinfty \cap \Kpinfty=K; \\
& \\
\frac{1}{p^2\mathfrak{t}}\cdot \nabla_\tau=\frac{1}{4\mathfrak{t}}\cdot \nabla_\tau, & \text{if }  \Kinfty \cap \Kpinfty=K(\pi_1), \text{ cf. Notation \ref{nota hatG}. }
\end{cases}
\end{equation}
This operator also acts on the $\hat{G}$-locally analytic representation $(W^{G_L})^{\hat{G}\dla}$ for $W\in \rep^\fg_\gk(\bdrplus)$, and hence in particular induces a map
\begin{equation}\label{eq322tau}
N_\nabla: D^+_{\dif, \kinfty}(W) \to (W^{G_L})^{\hat{G}\dla}
\end{equation}
\end{defn}

\begin{theorem}\label{thmkummersenop}
Let $W \in \rep^\fg_\gk(\bdrplus)$, then the image of the map   \eqref{eq322tau} falls inside $D^+_{\dif, \kinfty}(W)$.
Furthermore, the following linear  scaling
\begin{equation}\label{eqnnablanorm}
\frac{1}{u\lambda'}\cdot N_\nabla: D^+_{\dif, \kinfty}(W) \to D^+_{\dif, \kinfty}(W).
\end{equation}
defines  an object in $\mathrm{MIC}^\fg_\lambda(\kinftylambda)$.
We call the operator $\frac{1}{u\lambda'}\cdot N_\nabla$   the  \emph{Sen--Fontaine operator over the Kummer tower}.
 This  induces an exact tensor functor
$$  \rep^\fg_\gk(\bdrplus) \to  \mathrm{MIC}^\fg_\lambda(\kinftylambda).$$
\end{theorem}
\begin{proof}
Note $N_\nabla$ commutes with $\gal(L/\kinfty)$, i.e., $gN_\nabla=N_\nabla g, \forall g\in \gal(L/\kinfty)$, cf. \cite[Eqn. (4.2.5)]{Gao23}; hence the image of \eqref{eq322tau} also satisfies $\gamma=1$, and hence falls inside $ (W^{G_L})^{\tau\dla, \gamma=1}= D^+_{\dif, \kinfty}(W).$
To see that $\frac{1}{u\lambda'}\cdot N_\nabla$ is a  log-$\lambda$-connection, it suffices to check that $$\frac{1}{u \lambda'}N_\nabla (\lambda) =\lambda.$$
\end{proof}

 \subsection{Property of the $\lambda$-connection}\label{subsec lambda conn}

\begin{prop}\label{proplambdaconniso}
Let $W \in \rep^\fg_\gk(\bdrplus)$. Let $(D_{\dif, \kpinfty}(W), \nabla_t)$ and $(D_{\dif, \kinfty}(W), \nabla_\lambda)$ be the corresponding log-$t$-connection and   log-$\lambda$-connection. Then $W^{G_K}$ is a finite dimensional $K$-vector space, and we have the following identifications
\begin{eqnarray}
\label{eqgkpinfty} W^{G_K}\otimes_K \kpinfty & \xrightarrow{\simeq} (D_{\dif, \kpinfty}(W))^{\nabla_t=0}. \\
\label{eqgkl} W^{G_K}\otimes_K L  & \xrightarrow{\simeq} (W^{G_L})^{\hat{G}\dla, \mathrm{Lie} \hat{G}=0}.\\
\label{eqgkinfty} W^{G_K}\otimes_K \kinfty   & \xrightarrow{\simeq} (D_{\dif, \kinfty}(W))^{\nabla_\lambda=0}.
\end{eqnarray}
\end{prop}
\begin{proof}
Finite dimensionality of $W^{G_K}$ and \eqref{eqgkpinfty}  are proved in \cite[Prop. 3.8(i)]{Fon04}.
For \eqref{eqgkl},  we have
\begin{align*}
(W^{G_L})^{\hat{G}\dla, \mathrm{Lie} \hat{G}=0} &= \cup_{n\geq 1} (W^{G_L})^{\hat{G}\dla, \nabla_\gamma=0, \gal(L/\kpinfty(\pi_n))=1}\\
&= \cup_{n\geq 1} (D_{\dif, \kpinfty(\pi_n)}(W|_{G_{K(\pi_n)}}))^{\nabla_t=0} \\
&= \cup_{n\geq 1} W^{G_{K(\pi_n)}}\otimes_{K(\pi_n)} \kpinfty(\pi_n), \text{ by \eqref{eqgkpinfty} applied to $W|_{G_{K(\pi_n)}}$,}\\
&= \cup_{n\geq 1}  W^{G_K}\otimes_K \kpinfty(\pi_n), \text{ by Galois descent,} \\
&=  W^{G_K}\otimes_K L
\end{align*}
 Finally, \eqref{eqgkinfty} follows from \eqref{eqgkl} by taking invariants of $\gal(L/\kinfty)$.
\end{proof}

The following result  is standard consequence of above Prop. \ref{proplambdaconniso}.

\begin{prop}
Let $W_1, W_2 \in \rep^\fg_\gk(\bdrplus)$, the following are equivalent:
\begin{enumerate}
\item $W_1$ and $W_2$ are isomorphic as objects in $\rep^\fg_\gk(\bdrplus)$.
\item  $D_{\dif, \kpinfty}(W_1)$ and $D_{\dif, \kpinfty}(W_2)$ are isomorphic as objects in $\mathrm{MIC}^\fg_t(\kpinftyt)$.
\item   $D_{\dif, \kinfty}(W_1)$ and $D_{\dif, \kinfty}(W_2)$ are isomorphic as objects in  $\mathrm{MIC}^\fg_\lambda(\kinftylambda)$.
\end{enumerate}
\end{prop}
\begin{proof}
The equivalence between the first two items follow from \cite[Prop. 3.8]{Fon04}.
Clearly Item (1) implies Item (3), it suffices to prove the converse. The argument is rather similar to \cite[Prop. 3.8]{Fon04}, but not completely the same: in \emph{loc. cit.}, one can completely work over modules over $\kpinfty[[t]]$; this is not the case for us as the Kummer tower is not Galois. We include a detailed argument for completeness.

Suppose now we are given an isomorphism in  $\mathrm{MIC}^\fg_\lambda(\kinftylambda)$:
\begin{equation}\label{eqgivenf}
f: D_{\dif, \kinfty}(W_1)\to D_{\dif, \kinfty}(W_2)
\end{equation}
Let $W=\hom_\bdrplus(W_1, W_2)$ denote the space of $\bdrplus$-linear homomorphisms, then it is an object in $\rep^\fg_\gk(\bdrplus)$, and   the space $\hom_\gk(W_1, W_2)$ of $G_K$-equivariant  morphisms is precisely $W^{\gk}$, which is a finite dimensional $K$-vector spaces (by Prop. \ref{proplambdaconniso}). In addition,  by \eqref{eqgkinfty} in Prop. \ref{proplambdaconniso}, we have
\begin{equation}\label{eqw1w2}
\Hom_\gk(W_1, W_2)\otimes_K \kinfty =\Hom_{\mathrm{MIC}_\lambda(\kinftylambda)}(D_{\dif, \kinfty}(W_1), D_{\dif, \kinfty}(W_2))
\end{equation}
Let $f_1, \cdots, f_n$ be a basis of the $K$-vector space $W^{\gk}$.
Let $y_1, \cdots, y_h$ (resp. $z_1, \cdots, z_h$) be a quasi-basis of the $\kinftylambda$-module $D_{\dif, \kinfty}^+(W_1)$ (resp. $D_{\dif, \kinfty}^+(W_2)$.
Recall each $f_j: W_1 \to W_2$ is a $G_K$-equivariant  morphism, and hence by functoriality  induces a $\nabla_\lambda$-horizontal morphism $D_{\dif, \kinfty}^+(W_1) \to  D_{\dif, \kinfty}^+(W_2)$. Thus in particular, each $f_j$ can be expressed as
$$f_j(y_1, \cdots, y_h) =(z_1, \cdots, z_h) A_j, \text{ with $A_j$ a matrix over $\kinftylambda$ }$$
Define a multi-variable polynomial
$$P(X_1, \cdots, X_n) =\mathrm{det}(X_1A_1+\cdots  +X_nA_n) \in \kinftylambda[X_1, \cdots, X_n]$$
and let $\bar{P}(X_1, \cdots, X_n) \in \kinfty[X_1, \cdots, X_n]$ be the reduction modulo $\lambda$.
Recall we are given $f$ in \eqref{eqgivenf}. By \eqref{eqw1w2}, we can express $f$ by:
$$f=\sum_{i=1}^n \lambda_i f_i, \text{ with $\lambda_i \in \kinfty$ }$$
 Since $f$ is an isomorphism, $P(\lambda_1, \cdots, \lambda_n) $ is a unit in $\kinftylambda$, and hence  $\bar P(\lambda_1, \cdots, \lambda_n) \neq 0 $, and hence in particular the polynomial $\bar{P}$ is not identically zero.
 Since $K$ is infinite, there exists $\mu_1, \cdots, \mu_n \in K$ such that $\bar P(\mu_1, \cdots, \mu_n) \neq 0$. Consider the morphism
$$f' :=\sum_i \mu_i f_i: W_1 \to W_2.$$
It is $G_K$-equivariant since $\mu_i \in K$.
The induced $\nabla_\lambda$-horizontal morphism
\[s': D_{\dif, \kinfty}^+(W_1) \to D_{\dif, \kinfty}^+(W_2)  \]
  is an isomorphism since its reduction modulo $\lambda$ (which is $\bar P(\mu_1, \cdots, \mu_n) $) is an isomorphism.
  Thus the   map $f' : W_1 \to W_2$, being the base change of the isomorphism $s'$ from $\kinftylambda$ to $\bdrplus$, is the desired isomorphism.
\end{proof}

\begin{rem}
\begin{enumerate}
\item
For an object $W\in \rep_\gk(\bdr)$, Fontaine's theory still works, cf. \cite[\S 3.5]{Fon04}.
Namely, if we let
\[ D_{\dif, \kpinfty}(W): =(W^{G_\kpinfty})^{\gammak\dpa} \]
then it is a finite dimensional vector space over $\kpinfty((t))$ with a natural identification
\[D_{\dif, \kpinfty}(W) \otimes_{\kpinfty((t))} \bdr =W. \]
Furthermore, $\nabla_\gamma$ induces a log-$t$-connection on $D_{\dif, \kpinfty}(W)$.

\item
In our Kummer tower setting, we  can similarly define
\[D_{\dif, \kinfty}(W):= (W^{G_L})^{\tau\dpa, \gamma=1};\]
then it is a finite dimensional vector space over $\kinfty((\lambda))$ with a natural identification
\[D_{\dif, \kinfty}(W) \otimes_{\kinfty((\lambda))} \bdr =W. \]
Then all the relevant results in this section hold analogously; we leave  the interested readers to formulate and prove them. (As mentioned in Rem. \ref{remnobdrtheory}, we do not know if the corresponding $\bbdr$-crystal theory works well).
\end{enumerate}
\end{rem}

\subsection{Sen--Fontaine theory and Galois cohomology}

\begin{theorem}\label{thm-SenFon coho Galois coho}
Let $W \in \rep^\fg_\gk(\bdrplus)$, then we have  a diagram of $K$-linear  quasi-isomorphisms
\begin{equation}\label{diagqisom}
\begin{tikzcd}
{\rgamma(G_K, W)} \arrow[rr, "\simeq"] &  & {\rgamma(\hat{G}, W^{G_L})} \arrow[rr, "\simeq"] \arrow[d, "\simeq"]                &  & {\rgamma(\gammak, W^{G_\kpinfty})} \arrow[d, "\simeq"]           \\
                                       &  & {\rgamma(\hat{G}, (W^{G_L})^{\hat{G}\dla})} \arrow[rr, "\simeq"] \arrow[d, "\simeq"] &  & {\rgamma(\gammak,  D^+_{\dif, \kpinfty}(W))} \arrow[d, "\simeq"] \\
                                       &  & {(\rgamma(\Lie \hat{G}, (W^{G_L})^{\hat{G}\dla}))^{\hat G}} \arrow[rr, "\simeq"]     &  & {(\rgamma(\Lie \gammak, D^+_{\dif, \kpinfty}(W)))^{\gammak}}
\end{tikzcd}
\end{equation}
In addition, we have $K$-linear quasi-isomorphisms
\begin{align}
\label{qisopinfty}\rgamma(G_K, W)\otimes_K \kpinfty & \simeq
[D^+_{\dif, \kpinfty}(W) \xrightarrow{\nabla_t} D^+_{\dif, \kpinfty}(W)] \\
\label{qisoL} \rgamma(G_K, W)\otimes_K L &\simeq \rgamma(\Lie \hat{G}, (W^{G_L})^{\hat{G}\dla}\\
\label{qisoinfty}\rgamma(G_K, W)\otimes_K \kinfty &\simeq
[D^+_{\dif, \kinfty}(W) \xrightarrow{\nabla_\lambda} D^+_{\dif, \kinfty}(W)]
\end{align}
\end{theorem}
\begin{proof}

Note that the functors passing from $W$ to $W^{G_L}$ resp. to $W^{G_\kpinfty}$ resp. to $D_{\dif, \kpinfty}(W)$ resp. to $D_{\dif, \kinfty}(W)$ etc. are all \emph{exact}; cf. discussion in the above subsections.
Thus by standard d\'evissage, it suffices to treat the torsion case, and hence indeed the case when $W\in \rep_\gk(C)$. This is treated in \cite[Thm. 7.17]{GMWHT}.
\end{proof}

\section{Cohomology of crystals II: vs. Galois cohomology}
\label{sec_compa_Galois}

 In this section, we compare (log-) prismatic cohomology of $\bbdrplus$-crystals with Galois cohomology.

\begin{notation} \label{notasecgalcoho}
Let $\ast \in \{ \emptyset, \log \}$.
Let $\bm \in \Vect((\calO_K)_{\Prism, \ast},\bB_{\dR,m}^+).$
\begin{enumerate}
\item Consider the evaluation 
\[ M=\bM((\gs, E, \ast)) \]
In addition, one can associate a log-$\lambda$-connection $(M, \nabla_{M,\lambda})$ by Thm. \ref{Thm-dRasLogconnection}.

\item Consider also
\[ W=\bM((\ainf, (\xi), \ast)) \]
 We have $W \in \rep_\gk(\bdrplus)$ and hence we can associate $(D_{\dif, \kinfty}^+(W), \nabla_\lambda) \in \mic_\lambda(\kinftylambda)$.
We shall consider $M$ as a subspace of $W$ via the identification $M\otimes_{K[[\lambda]]} \bdrplus =W$.

\item Recall we defined the scalar
\begin{equation*}
a=
\begin{cases}
  -E'(\pi), &  \text{if } \ast=\emptyset \\
 -\pi E'(\pi), &  \text{if } \ast=\log
\end{cases}
\end{equation*}
\end{enumerate}

\end{notation}

\begin{thm}\label{thm-compare-M-Ddif}
Use above notations.
\begin{enumerate}
\item We have
 \begin{equation}\label{eqMDsen}
  D_{\dif, \kinfty}^+(W) =M\otimes_{K[[\lambda]]}  \kinfty[[\lambda]],
  \end{equation}
  where both sides are regarded as subspaces of $W$.

 \item  Extend the log-$\lambda$-connection $(M, \nabla_{M, \lambda})$   to $\kinftylambda$ via
\[\nabla_{M, \lambda}\otimes 1+ 1\otimes \lambda\cdot \frac{d}{d\lambda}: M\otimes_{K[[\lambda]]}  \kinfty[[\lambda]] \to M\otimes_{K[[\lambda]]}  \kinfty[[\lambda]];\]
then this is the \emph{same} (via \eqref{eqMDsen}) as the  Sen--Fontaine operator over the Kummer tower
\begin{equation*}
\frac{1}{ u\lambda' }\cdot N_\nabla: D_{\dif, \kinfty}^+(W) \to D_{\dif, \kinfty}^+(W)
\end{equation*}
 defined in Thm. \ref{thmkummersenop}.
 \end{enumerate}
 \end{thm}
 \begin{proof}
For any $x\in M$, the $G_L$-action on $x$ is trivial, and hence $M$ is also a sub-$K[[\lambda]]$-space of $W^{G_L}$. By (\ref{Equ-MatrixCocycle}), for any $x\in M$, we have
\begin{equation}\label{eqtausec11}
\tau^{p^i} (x) = \sum_{n\geq 0}a^n\prod_{i=0}^{n-1}(\nabla_{M, u-\pi}-i)(x)(\frac{([\epsilon]^{c(\tau^{p^i})}-1)[\pi^{\flat}]}{a([\pi^{\flat}]-\pi)})^{[n]}.
\end{equation}
Here $\nabla_{M, u-\pi}$ is the log-$(u-\pi)$-connection associated to $\bM$. (This operator is related with $\nabla_{M, \lambda}$ via Lem. \ref{lemchangeconn}, which will be applied in the end of this proof).
Note that using Notation \ref{nota hatG},
\begin{equation*}
c(\tau^{p^i})=
\begin{cases}
p^i, & \text{ if } K_{\infty} \cap K_{p^\infty}=K\\
2p^i, & \text{ if } K_{\infty} \cap K_{p^\infty}=K(\pi_1)
\end{cases}
\end{equation*}
 Hence \eqref{eqtausec11} implies that the $\gal(L/\kpinfty)$-action on $x$ is \emph{analytic} (not just locally analytic!).
This implies that
\begin{equation}
  D_{\dif, \kinfty}^+(W) =M\otimes_{K[[\lambda]]}  \kinfty[[\lambda]].
  \end{equation}

  For Item (2), a simple computation using \eqref{eqtausec11} shows that
\begin{equation}
\nabla_\tau(x) =  \frac{ut}{u-\pi}\cdot \nabla_{M, u-\pi}(x) \quad \text{ resp. } =\frac{2ut}{u-\pi}\cdot \nabla_{M, u-\pi}(x),
\end{equation}
when $K_{\infty} \cap K_{p^\infty}$ equals to $K$ resp. $K(\pi_1)$,
where we use
\[ \lim_{i \to \infty} \frac{[\epsilon]^{p^i}-1}{p^i} =t. \]
Now,   using definition of $N_\nabla$ as in Def. \ref{defndiffwtb}, we see that
\begin{eqnarray*}
\frac{1}{ u\lambda' }\cdot N_\nabla (x)   &=&\frac{1}{ u\lambda' }\cdot \frac{u\lambda}{u-\pi}
\nabla_{M, u-\pi}(x) \\
&=& \nabla_{M,\lambda}(x),
\end{eqnarray*}
where the last equality follows from  Lem. \ref{lemchangeconn}.
Thus we can conclude.
 \end{proof}

\begin{theorem}\label{thm-pris coho and Galois coho} Use Notation \ref{notasecgalcoho}.
We have $K$-linear quasi-isomorphisms
\[\rgamma(\okprisast, \bm)
 \simeq  [M \xrightarrow{a\nabla_{M, \lambda}} M]
 \simeq \rgamma(\gk, W)
\]
which are functorial in $\bm$.
\end{theorem}
\begin{proof}
Thm. \ref{Thm-dRCohomology} already proves the first quasi-isomophism; it suffices to compare Sen-Fontaine cohomology with Galois cohomology.
Simply note
\begin{align*}
[M \xrightarrow{a\nabla_{M, \lambda}} M]\otimes_K \kinfty   & \quad \simeq \quad
 [M \xrightarrow{\nabla_\lambda}  M] \otimes_K \kinfty \\ \quad & \quad \simeq \quad [D_{\dif, \kinfty}(W) \xrightarrow{\nabla_\lambda} D_{\dif, \kinfty}(W)], \text{ by Thm. \ref{thm-compare-M-Ddif}(2) }\\
  & \quad \simeq \quad \rgamma(\gk, W)  \otimes_K \kinfty, \text{ by Thm. \ref{thm-SenFon coho Galois coho} }
\end{align*} 
\end{proof}

 \section{Log-nearly de Rham representations}\label{sec: ndR rep}
In this section, we classify $\bbdrplus$-crystals by (log)-nearly de Rham representations. The proof uses d\'evissage argument generalizing  known results in the Hodge--Tate case. However, as already noted in Rem. \ref{rem_intro_extension}, extension between \emph{crystals} (which are sheaves) is a subtle issue; in particular, it is not clear if ``fintiely generated crystals" is a well-behaved notion.
Hence it is crucial that in the first subsection, \S \ref{susec12.1}, we completely work on finitely generated $a$-nilpotent connections, which are concrete objects that can be \emph{d\'evissaged}.

Recall in Def. \ref{defnnht}, we say $W\in \rep_\gk(C)$ is \emph{nearly Hodge--Tate} (resp. \emph{log-nearly Hodge--Tate})  if all of its  Sen weights are in the subset
\[\mathbb{Z} + (E'(\pi))^{-1}\cdot \mathfrak{m}_{\O_{\overline{K}}}, \quad \text{ resp. } \mathbb{Z} +  (\pi\cdot E'(\pi))^{-1}\cdot \mathfrak{m}_{\O_{\overline{K}}}.  \]

\begin{defn}[generalizing Def. \ref{defnndR}]
Say $W\in \rep^\fg_\gk(\bdrplus)$ is \emph{nearly de Rham} (resp. \emph{log-nearly de Rham}) if the $C$-representation $W/tW$ is nearly Hodge--Tate (resp. log-nearly Hodge--Tate).
Use
$$\rep^{\fg, \mathrm{ndR}}_\gk(\bdrplus), \text{ resp. } \rep^{\fg, \mathrm{lndR}}_\gk(\bdrplus)$$
to  denote the subcategory of $\rep^\fg_\gk(\bdrplus)$ consisting of these objects.
 Use
$$\rep_\gk^{\mathrm{ndR}}(\bdrplusm) \quad \text{  resp. }  \rep_\gk^{\mathrm{lndR}}(\bdrplusm) $$
 to  denote the subcategory consisting of objects which are finite free over $\bdrplusm$.
\end{defn}

With this terminology, a (log-) nearly Hodge--Tate $C$-representation is also regarded as a (log-) nearly de Rham $\mathbf{B}_{\dR, 1}^+$-representation.

\subsection{$a$-nilpotent connections and Galois representations} \label{susec12.1}

\begin{construction}\label{constnilconntorep}
Let $a=  -E'(\pi)$ or $-\pi E'(\pi)$.
Let $(M, \nabla_{M, u-\pi}) \in \mic^{\fg, a}_{u-\pi}(K[[u-\pi]])$.
Let $W=M\otimes_{K[[u-\pi]]} \bdrplus$.
For $x\in M$ and $g\in \gk$, we copy formula \eqref{Equ-MatrixCocycle}:
 \begin{equation} \label{NEWEqu-MatrixCocycle}
        g(x) = \sum_{n\geq 0}a^n\prod_{i=0}^{n-1}(\nabla_{M, u-\pi}-i)(x)(\frac{([\epsilon]^{c(g)}-1)[\pi^{\flat}]}{a([\pi^{\flat}]-\pi)})^{[n]}.
    \end{equation}
     (Note that this formula converges for any $a\in K$ such that $v_p(a) \leq v_p(\pi  E'(\pi)).$
We claim that, after semi-linearly extended over $\bdrplus$, this formula induces a $G_K$-action on $W$.  Note in the situation of Prop. \ref{Prop-MatrixCocycle}, we  are   working with a prismatic crystal $\bm$ and its evaluation $W$, hence there is \emph{a priori} already a $G_K$-action on $W$. Whereas here, we are only starting with a log connection. Indeed, we need to check
\begin{equation} \label{g1g2}
g_1(g_2(x)) =(g_1g_2)(x), \quad \forall g_1, g_2 \in G_K, \forall x \in W.
\end{equation}
 A direct verification would be rather complicated. Instead, choose a quasi-basis $\underline{e}=(e_1, \cdots, e_l)$ of $M$ and write $\nabla_{M, u-\pi}(\underline{e}) =\underline{e} A$ with $A$  a matrix in $K[[u-\pi]]$.  The matrix $A$ is certainly in general not unique, but we fix one.
 Define a \emph{finite free} log-$(u-\pi)$-connection $(\hat{M}, \hat{\nabla})$ by requiring the matrix $\hat{\nabla}$ with respect to a basis is $A$.
  Since $a$-nilpotency of a log-$(u-\pi)$-connection can be checked by modulo $(u-\pi)$, $(\hat{M}, \hat{\nabla})$ is $a$-nilpotent.
 By Thm. \ref{Thm-dRasLogconnection}, this \emph{finite free} nilpotent log-connection induces a (log-) prismatic crystal $\hat{\bm}$ and hence a finite free $\bdrplus$-representation $\hat{W}$. The Galois action on $\hat{W}$ can be described, via Prop. \ref{Prop-MatrixCocycle}, by the formula
 \begin{equation}\label{HATEqu-MatrixCocycle}
  \forall x\in \hat{M}, \quad      g(\hat x) = \sum_{n\geq 0}a^n\prod_{i=0}^{n-1}(\hat{\nabla} -i)(\hat x)(\frac{([\epsilon]^{c(g)}-1)[\pi^{\flat}]}{a([\pi^{\flat}]-\pi)})^{[n]}.
    \end{equation}
In particular, the above formula (when semi-linearly extended to $\hat W$) satisfies
 \[g_1(g_2(\hat x)) =(g_1g_2)(\hat x),\quad  \forall g_1, g_2 \in G_K, \forall \hat{x} \in \hat W.\]
Thus \eqref{g1g2} also holds.

Thus in summary, we can define  a functor
\begin{equation}\label{functor_conn_rep}
\mic^{\fg, a}_{u-\pi}(K[[u-\pi]]) \to \rep^{\fg}_\gk(\bdrplus).
\end{equation}
\end{construction}


\begin{prop} \label{thm-comparegeneral}
Use   notations in Construction \ref{constnilconntorep}.
Namely, let $(M, \nabla_{M, u-\pi}) \in \mic^{\fg, a}_{u-\pi}(K[[u-\pi]])$, and let $W$ be the corresponding finitely generated $\bdrplus$-representation.
And use $(M, \nabla_{M, \lambda})$  to denote the corresponding log-$\lambda$-connection via Lem. \ref{lemchangeconn}.
\begin{enumerate}
\item We have
 \begin{equation}\label{neweqMDsen}
  D_{\dif, \kinfty}^+(W) =M\otimes_{K[[\lambda]]}  \kinfty[[\lambda]],
  \end{equation}
  where both sides are regarded as subspaces of $W$.

 \item  Extend the log-$\lambda$-connection $(M, \nabla_{M, \lambda})$   to $\kinftylambda$ via
\[\nabla_{M, \lambda}\otimes 1+ 1\otimes \lambda\cdot \frac{d}{d\lambda}: M\otimes_{K[[\lambda]]}  \kinfty[[\lambda]] \to M\otimes_{K[[\lambda]]}  \kinfty[[\lambda]];\]
then this is the \emph{same} (via \eqref{neweqMDsen}) as the  Sen--Fontaine operator over the Kummer tower
\begin{equation*}
\frac{1}{ u\lambda' }\cdot N_\nabla: D_{\dif, \kinfty}^+(W) \to D_{\dif, \kinfty}^+(W)
\end{equation*}
 defined in Thm. \ref{thmkummersenop}.
 \end{enumerate}
 \end{prop}
 \begin{proof}
 This is generalization of Thm. \ref{thm-compare-M-Ddif} which treated the case when $M$ comes from a  (log-) prismatic crystal (i.e., when $M$ is finite free over   $\bdrplusm$ for some $m$). The proof there works verbatim in the general case.
 \end{proof}

\begin{thm} \label{thm124}
The functor \eqref{functor_conn_rep} induces a diagram where both vertical arrows are   equivalences
\begin{equation*}
\begin{tikzcd}
{\mic^{\fg, -E'(\pi)}_{u-\pi}(K[[u-\pi]])} \arrow[rr, hook] \arrow[d, "\simeq"] &  & {\mic^{\fg, -\pi E'(\pi)}_{u-\pi}(K[[u-\pi]])} \arrow[d, "\simeq"] \\
{\rep^{\fg, \ndR}_\gk(\bdrplus)} \arrow[rr, hook]                               &  & {\rep^{\fg, \mathrm{lndR}}_\gk(\bdrplus)}
\end{tikzcd}
\end{equation*}
Both equivalences are bi-exact.
\end{thm}
\begin{proof}
By considering inverse limit, it suffices to treat the torsion case where all relevant objects are killed by some powers of $(u-\pi)$.
Namely, it suffices to prove equivalences in the following diagram
\begin{equation*}
\begin{tikzcd}
{\mic^{\fg, -E'(\pi)}_{u-\pi}(K[[u-\pi]]/(u-\pi)^m)} \arrow[rr, hook] \arrow[d, "\simeq"] &  & {\mic^{\fg, -\pi E'(\pi)}_{u-\pi}(K[[u-\pi]]/(u-\pi)^m)} \arrow[d, "\simeq"] \\
{\rep^{\fg, \ndR}_\gk(\bdrplusm)} \arrow[rr, hook]                               &  & {\rep^{\fg, \mathrm{lndR}}_\gk(\bdrplusm)}
\end{tikzcd}
\end{equation*}
When $m=1$, i.e., the case with (log-) nearly Hodge--Tate  representations, these are proved in  \cite{GMWHT}.
We now prove the general (torsion) case by a  d\'evissage  argument   inspired by \cite[Thm. 3.6]{Fon04}.
In the following, we only prove the nearly de Rham case (i.e., the left vertical arrow), as the log  case is similar.

For a general $m \geq 2$, first note the functor
 $\mic^{\fg, -E'(\pi)}_{u-\pi}(K[[u-\pi]]/(u-\pi)^m) \to  \mathrm{Rep}_\gk(\bdrplusm)$ factors through $\mathrm{Rep}_\gk^{\mathrm{ndR}}(\bdrplusm)$  as the $m=1$ case is known.
Thus, we have the desired functor
\[\mic^{\fg, -E'(\pi)}_{u-\pi}(K[[u-\pi]]/(u-\pi)^m) \to \mathrm{Rep}_\gk^{\fg,\mathrm{ndR}}(\bdrplusm).\]
We first prove the functor is fully faithful. Let $(M_1, \nabla_1), (M_2, \nabla_2)$ be two $a$-nilpotent log-$(u-\pi)$-connections, and let $W_1, W_2 \in \rep_\gk^{\fg, \mathrm{ndR}}(\bdrplusm)$ be the corresponding representations.
It suffices to show
  \[ \Hom_{\mathrm{MIC}_\lambda} ((M_1, \nabla_1), (M_2, \nabla_2))  \to \Hom_\gk(W_1, W_2).\]
  is a bijection.
Tensor the above map over $\kinfty$, and use Prop. \ref{thm-comparegeneral}, then we recover \eqref{eqw1w2} which is an isomorphism.

Now we prove essential surjectivity by induction on $m$.
Suppose the results are valid for objects killed by $(u-\pi)^{m-1}$.
Let $X \in \mathrm{Rep}_\gk^{\fg, \mathrm{ndR}}(\bdrplusm)$. (Here we are using the notation $X$ to imitate that in  \cite[Thm. 3.6]{Fon04}, as the following argument   resembles Fontaine's.)
Consider  the short exact sequence
\begin{equation*}
0 \to X'=(u-\pi)^{m-1}X \to X \to X'' =X/X' \to 0
\end{equation*}
Suppose $(M', f')$ and $(M'', f'')$ are the log-$(u-\pi)$-connections that map to $X'$ and $X''$ respectively. It suffices to construct some connection  $(M, f)$ that maps to $X$.

We first set up some notations. Let $\mathrm{Ext}^1_\gk(X'', X')$ be the Yoneda extension group in the (exact) category $\mathrm{Rep}_\gk^{\fg}(\bdrplus)$; namely, an object is an equivalence class of $\gk$-equivariant extensions  $0 \to X' \to Y \to X''\to 0$ where $Y$ is a  finitely generated $\bdrplus$-module with a continuous $\gk$-action.
It is an abelian group.
Let $\mathrm{Ext}_{\gk, 0}^1(X'', X')$ be the subgroup consisting of extensions where $Y \simeq X'\oplus X''$ as $\bdrplus$-modules (but not $\gk$-equivariantly).
Indeed, there is a left exact sequence of abelian groups
\[0 \to  \mathrm{Ext}_{\gk, 0}^1(X'', X') \to \mathrm{Ext}_\gk^1(X'', X') \to \mathrm{Ext}_{\bdrplus}^1(X'', X')\]
where $\mathrm{Ext}_{\bdrplus}^1(X'', X')$ denotes extensions of (finitely generated) $\bdrplus$-modules.
One can similarly define the groups $\Ext^1(M'', M')$ and $\Ext^1_0(M'', M')$, whose elements consist of extensions of (finitely generated) log-$(u-\pi)$-connections.

Let $x_1'', \cdots, x_l''$ be a  basis of $M''$, which is also a basis of $X''$.
Consider the matrix of $f''$ on $x_1'', \cdots, x_l''$, and lift it to a $l\times l$ matrix $A''$ in $K[[u-\pi]]$; note $A'' \pmod{(u-\pi)^{m-1}}$ is  uniquely determined.

Lift $x_1'', \cdots, x_l''$  to some $x_1, \cdots, x_l \in X$, which then is a basis for $X$. Let $M_0$ be the sub-$K[[u-\pi]]$-module of $X$ generated by $x_1, \cdots, x_l$.
Use the above chosen $A''$ to define a $(u-\pi)$-connection on  $(M_0, f_0)$ so that the matrix of $f_0$ with respect to  $x_1, \cdots, x_l$ is $A''$.
 As the connection $(M'', f'')$ is $E'(\pi)$-nilpotent, hence so is $(M_0, f_0)$; thus  $(M_0, f_0)$ induces a nearly de Rham representation $X_0$.
  Note that as  a $\bdrplus$-module, $X_0$ is isomorphic to $X$. In addition, one sees that $X_0 \in \mathrm{Ext}^1_\gk(X'', X')$; this is essentially because the Galois action on $X'$ is determined by the matrix $A''\pmod{u-\pi}$.

In particular, we have
\[
[X]-[X_0] \in \Ext^1_{\gk, 0}(X'', X')
\]
It now suffices to check
\[
\Ext^1_0(M'', M') \to \Ext^1_{\gk, 0}(X'', X')
\]
is a surjection. Indeed we will show it is an isomorphism.
Note we have
\[\Ext^1_{\gk, 0}(X'', X')= \Ext^1_{\gk, 0}(X''/(u-\pi) X'', X'),\]
essentially because $\hom_\bdrplus(X'', X') =\hom_\bdrplus(X''/(u-\pi) X'', X')$.
For similar reasons,
 \[\Ext^1_0(M'', M') = \Ext^1_0(M''/(u-\pi) M'', M').\]
 Finally, note
 \[\Ext^1_{\gk, 0}(X''/(u-\pi) X'', X') =\Ext^1_0(M''/(u-\pi) M'', M')\]
  as both sides compute extensions which are killed $u-\pi$, and hence are equal by the $m=1$ case.
\end{proof}

\subsection{$\bbdrplus$-crystals and Galois representations}

\begin{theorem} \label{thm_ndR_rep}
The evaluation functors  in Notation \ref{notasecgalcoho}(2)  induce bi-exact  equivalences of categories:
\begin{eqnarray*}
\Vect(\okpris, \bbdrplusm) & \xrightarrow{\simeq} & \mathrm{Rep}_\gk^{\mathrm{ndR}}(\bdrplusm)\\
\Vect(\okprislog, \bbdrplusm) &  \xrightarrow{\simeq} & \mathrm{Rep}_\gk^{\mathrm{lndR}}(\bdrplusm)
\end{eqnarray*}
\end{theorem}
\begin{proof}
This follows from Thm. \ref{Thm-dRasLogconnection} and Thm. \ref{thm124}.
\end{proof}

We thus have proved all the vertical equivalences in the diagram of Thm. \ref{thmintrondr}. All arrows in the bottom row are obvious inclusions, and hence all the arrows in the top row are fully faithful (as all the squares are obviously commutative).

  \bibliographystyle{alpha}

\end{document}